\documentclass[reqno]{amsart}

\usepackage{amsmath, amssymb, cases, color,amsthm, mathrsfs}
\usepackage{hyperref}
\usepackage{todonotes}

\usepackage{bbm}
\usepackage{stmaryrd}
\usepackage{amsmath, amssymb,amscd, mathtools}
\usepackage{amsfonts}
\usepackage{graphicx}
\usepackage[top=1in, bottom=1.25in, left=1.10in, right=1.10in]{geometry}

\usepackage{subfig}
\usepackage{graphicx}
\usepackage{pgfplots}
\usetikzlibrary{patterns}
\usepgfplotslibrary{fillbetween}

\usepackage{lipsum}
\usepackage{amsfonts}
\usepackage{graphicx}
\usepackage{epstopdf}
\usepackage{algorithmic}

 \usepackage{upref}
\hypersetup{linkcolor=blue, colorlinks=true,citecolor = red}
\newtheorem{thm}{Theorem}[section]
\newtheorem{prop}[thm]{Proposition}
\newtheorem{lem}[thm]{Lemma}

\newtheorem{rmk}[thm]{Remark}

\theoremstyle{definition}
\newtheorem{definition}[thm]{Definition}

\theoremstyle{remark}
\numberwithin{equation}{section}

\newcommand{\BE}{\begin{equation}}
\newcommand{\EEQ}{\end{equation}}
\newcommand{\rfb}[1]{\mbox{\rm
		(\ref{#1})}\ifx\undefined\stillediting\else:\fbox{$#1$}\fi}

\newfont{\roma}{cmr10 scaled 1200}

\newcommand{\nline}  {{\mathbb N}}

\newcommand{\rline}  {{\mathbb R}}

\newcommand{\dd}  {{\rm d}\hbox{\hskip 0.5pt}}
\renewcommand{\leq} {\leqslant}
\renewcommand{\geq} {\geqslant}
\newcommand{\mm}    {{\hbox{\hskip 0.5pt}}}
\newcommand{\m}     {{\hbox{\hskip 1pt}}}

\newcommand{\bluff} {{\hbox{\raise 15pt \hbox{\mm}}}}
\newcommand{\sbluff}{{\hbox{\raise 10pt \hbox{\mm}}}}
\newcommand{\Om}    {{\Omega}}
\renewcommand{\div} {{\rm div\,}}
\newcommand{\eps}    {{\varepsilon}}

\newcommand{\prt}      {{\partial}}

\newcommand{\Ascr} {\mathcal{A}}
\newcommand{\Bscr} {\mathcal{B}}

\newcommand{\Fscr}{\mathcal{F}}

\newcommand{\Jscr} {\mathcal{J}}
\newcommand{\Kscr} {\mathcal{K}}

\newcommand{\Mscr} {\mathcal{M}}
\newcommand{\Oscr} {\mathcal{O}}

\newcommand{\Rscr} {\mathcal{R}}
\newcommand{\Sscr} {\mathcal{S}}
\newcommand{\Uscr} {\mathcal{U}}

\DeclareMathOperator{\dist}{dist}


\begin{document}

\title[Compressible Euler interacting with a rigid body]
{Well-posedness of the motion of  a rigid body immersed  in a compressible inviscid fluid}

 \date{\today}

\author{Fr\'ed\'eric Rousset}
\address{Laboratoire de Math\'ematiques d'Orsay (UMR 8628),
	Universit\'e Paris-Saclay,
	Bâtiment 307, rue Michel Magat,
	F - 91405 Orsay, France}
\email{frederic.rousset@universite-paris-saclay.fr}

\author{Pei Su}
\address{Laboratoire de Math\'ematiques d'Orsay (UMR 8628),
	Universit\'e Paris-Saclay,
	Bâtiment 307, rue Michel Magat,
	F - 91405 Orsay, France}
\email{pei.su@universite-paris-saclay.fr}



\begin{abstract}
We consider a rigid body freely moving in a compressible inviscid fluid within a bounded domain $\Om\subset\rline^3$. The fluid is thereby governed by the non necessarily isentropic compressible Euler equations, while the rigid body obeys the conservation of linear and angular momentum. This forms a coupled system comprising an ODE and the initial boundary value problem (IBVP) of a hyperbolic system with characteristic boundary in a moving domain, where the fluid velocity matches the solid velocity along the normal direction of the solid boundary.
We establish the existence of a unique local  classical solution to this coupled system. To construct the solution, we first perform a change of variables to reformulate the problem in a fixed spatial domain, and then analyze an approximate system with a non-characteristic boundary. For this nonlinear approximate system, we use the better regularity for the trace of the pressure on the boundary to contruct a solution by a fixed-point argument  in which  the fluid motion and the solid motion are updated in successive steps.
We are then able to derive estimates independent of the regularization parameter and  to pass to the limit  by a strong  compactness arguments.

\end{abstract}
\maketitle



\section{Introduction}\label{sec_intro}

We study the interaction between a compressible inviscid fluid and a rigid body, which is completely immersed in the fluid. The fluid is described by the three dimensional non necessarily isentropic  {\em compressible Euler equations}. The motion of the rigid body obeys the conservation of linear and angular momentum. We are interested in establishing the local well-posedness of this fluid-body coupled system in  bounded domains.  The approach we are using can be extended to an arbitrary number of solids in a straightforward way. We shall now precisely state the problem.

\subsection{Statement of the problem}\label{sec_pro}

The domain occupied by the fluid-body system, denoted by $\Om\subset \rline^3$, is assumed to be a regular bounded domain with impermeable boundary $\partial\Om$. The rigid body we consider  is defined  by  the connected   set  $\Sscr(t)\subset\Om$ at time $t$ with center of mass at $h(t)\in \rline^3$  which is moving in the fluid,  we can allow an   {\em arbitrary shape}, we just assume that 
 $\mathcal{S}(0)$ is the closure of a smooth bounded open set.
The fluid domain, denoted by $\Fscr(t)$, is the exterior part of $\Sscr(t)$ in $\Om$, i.e. $\Fscr(t)=\Om\setminus\Sscr(t)$. The boundary of $\Fscr(t)$  is thus made  of two parts $\partial\Om$ and $\partial\Sscr(t)$, i.e. $\partial\Fscr(t)=\partial\Om\cup \partial\Sscr(t)$. We denote by $n(t,y)$ the unit normal vector field of $\partial\Fscr(t)$, which directs toward the interior of $\Sscr(t)$ along $\partial\Sscr(t)$ and toward the exterior of $\Om$ along $\partial\Om$, respectively. We use $p(t,y)$, $u(t,y)$ and $s(t,y)$ to describe the evolution of the fluid, representing the pressure, the velocity, and the physical entropy, respectively. With these notations, the governing equations describing the rigid body moving in the fluid,  reads for $t>0$,
\begin{itemize}
	\item Fluid equations:
	\begin{flalign}{\label{ns-ob1}}
	\begin{split}
	\partial_t p+(u\cdot\nabla) p+\rho c^2\m\div u=0 \qquad  \forall \m y 
	\in \Fscr(t), \\
	\partial_t u+(u\cdot\nabla) u+\rho^{-1}\nabla p=0 \qquad \forall  \m y \in \Fscr(t),\\
	\partial_t s+(u\cdot \nabla) s=0 \qquad \forall \m y\in \Fscr(t),
	\end{split}
	\end{flalign}
	\item Rigid body equations:
	\begin{align}
	m h''(t) &= \int_{\partial \Sscr(t)} 
	p\m n\m \dd \Gamma, \;\;\;{\label{solid1}} \\
	(J \omega)'(t) &= \int_{\partial \Sscr(t)} 
	(y-h) \times p\m n\m \dd \Gamma, 
	\;\;\; {\label{solid2}}
	\end{align}
	\item Boundary conditions:
	\begin{equation}\label{boundarycon1}
	\begin{aligned}
	u\cdot n&= u_\Sscr\cdot n & \qquad \forall \m   y \in \partial \Sscr(t),\\
	u\cdot n&=0  \qquad & \qquad \forall \m   y\in \partial\Om,
	\end{aligned}
	\end{equation}
	\item Initial conditions:
	\begin{equation}{\label{initialcon1}}
	\begin{aligned}
	(p, u, s)(0,\m y)&=(p_0, u_0, s_0)(y)  \qquad \forall \m  y\in\Fscr(0),\\
	\quad
	h (0)=0, \quad 
h'(0) &= l_0, \quad
	\omega(0) = \omega_0.
	\end{aligned}
	\end{equation}
\end{itemize}

Without loss of generality, we assume that the center of mass of the rigid body is at the origin at $t=0$.

We have chosen to describe the fluid by using the pressure instead of the density and thus to see the density $\rho$ and
 the sound speed $c$ as functions of $p$ and $s$.
  This is equivalent for smooth enough solutions to the more standard description where we start from the mass conservation
  \begin{equation}
  \label{masscons}
  \partial_{t} \rho + \div(\rho u) =0
  \end{equation}
   instead of the first equation of  \eqref{ns-ob1} and we assume an equation of state where  the pressure $p$ is  given as a  function of the density $\rho$ and the entropy $s$, i.e. $p=p(\rho, s)$ for some  pressure law. For example,  for  {\em ideal gas}, $p(\rho, s)=\kappa e^{\frac{s}{c_\nu}}\rho^\gamma$, where $\gamma>1$, $\kappa,\,  c_\nu>0$  or if an adiabatic process is assumed in addition,   $p(\rho)=a\rho^\gamma$ where $a>0$ and $\gamma\geq 1$.
The hyperbolicity of the system for compressible gas taking $ (\rho, u, s)$ as unknowns  requires $\rho >0$ and  $\partial_{\rho} p>0$, so that the sound
speed $c$ defined by 
 $c=(\partial p/\partial\rho )^{ 1 \over 2}(\rho, s) $ is well-defined and positive.
Since $\rho \mapsto p(\rho, s)$ is  then invertible,  $(p, u, s)$ can be  chosen as the unknowns and the first equation of  \rfb{ns-ob1}
deduced from the mass conservation \eqref{masscons}, $\rho$ and $c$ being seen as  positive functions of $(p, s)$. In order to consider
 general  laws, we assume that there exists $\Uscr\subset\rline^2$ such that
 $\rho(p,s)$ and $c(p, s)$ are smooth on $\Uscr$ and such that
  $\rho>0$ and $c>0$ for every $(p,s)\in \Uscr$.  This  represents the hyperbolic region for  the system \rfb{ns-ob1}. 

We shall now explain the derivation of the motion of the rigid body. To describe it, it is convenient to use  the rotation matrix $Q(t)\in SO(3)$ ($Q(0)=I_{3\times 3}$). Then the velocity of the solid $u_\Sscr$, corresponding to the particle $y(t)=h(t)+Q(t)y^0$ ($y^0$ is the initial position), is given by 
\begin{align*}
u_\Sscr&=h'(t)+Q'(t)y^0=h'(t)+Q'(t)Q^{-1}(t)(y-h(t))=h'(t)+Q'(t)Q^\intercal(t)(y-h(t)).
\end{align*} 
Since $Q'(t)Q^\intercal(t)$ is skew-symmetric, there exists a unique rotation vector $\omega(t)\in \rline^3$ such that 
$$Q'(t)Q^\intercal(t)y=\omega(t)\times y \qquad \forall \m  y\in\rline^3,$$
which implies that 
\begin{equation}\label{us}
u_\Sscr(t,y)=h'(t)+\omega(t)\times (y-h(t)) \qquad \forall \m  y\in \Sscr(t).
\end{equation} 
We further denote $l(t)=h'(t)$, then $l$ and $\omega$ above represents the translation velocity and the angular velocity of the rigid body, respectively.

The density of the body is a positive function $\rho_\Sscr\in L^\infty(\Sscr(0); \mathbb{R})$. 
The constant $m$ and the matrix $J$ stands for the mass and the inertia tensor of the body and can be expressed as below, respectively:
$$m:=\int_{\Sscr(0)}\rho_\Sscr \m\dd y, $$
$$J(t):=Q(t) J_0 Q^\intercal(t), $$
where $J_0$ is the initial value of $J$ given by
$$J_0:=\int_{\Sscr(0)}\rho_{\Sscr}\left(|y|^2I_{3\times 3}-y \otimes y \right)\dd y. $$

Many studies have been devoted to the subject of fluid-rigid body interaction systems, depending on the type of  fluids and the domain as well as the parameters of the rigid body. In the Navier-Stokes setting,  that is to say for viscous fluids, there are a lot of contributions on the well-posedness
and the qualitative behaviour of the system, we can refer  for instance to  \cite{san2002global, desjardins1999existence, gunzburger2000global, feireisl2003motion, nevcasova2022motion} for weak solutions, to \cite{takahashi2003analysis, wang2011analyticity} for strong solutions, for the large time behaviour to  \cite{ervedoza2014long, ervedoza2023large, maity2023motion} and for the study of some  singular limit problems to \cite{lacave2017small, he2024vanishing}. In the absence of viscosity, the case of  incompressible fluids, that is to say when the fluid is described by the incompressible Euler equation,  are also well studied,  see   for instance  \cite{glass2012smoothness, houot2010existence, ortega2007motion, rosier2009smooth}.
As far as we know, for the case of a rigid body interacting with a compressible inviscid fluid, there is only one work by Feireisl and M\'acha \cite{feireisl2021motion} where the existence of  weak solutions is achieved by using the method of convex integration. 

In the case of viscous fluids, the  local well-posedness results are usually obtained through fixed-point arguments relying on the smoothing
effect of the viscous term, while for incompressible inviscid fluids, the fluid system can be reduced to a transport equation coupled with
an elliptic equation by using the vorticity. 
  
Here, we have to consider the coupling between the motion of the solid and the fluid as an initial boundary value problem (IBVP) for the  hyperbolic system  \eqref{ns-ob1} in a moving domain with dynamical boundary conditions. 
As long as the solid does not touch the boundary, this moving  spatial domain has for fixed time  smooth boundaries, nevertheless, the time regularity
of the solid boundary is limited due to the coupling with the fluid motion. This has to be carefully addressed in our analysis since
the analysis of IBVP for first order  hyperbolic systems necessitates to involve both  time and space regularity.
The boundary conditions \eqref{boundarycon1}
 make the boundary (both the exterior boundary and the surface of the solid) characteristic.
 Without the solid motion, the existence theory for hyperbolic IBVPs with non-characteristic boundary is well developed, see for instance   \cite{rauch1974differentiability, majda1983stability, metivier2001stability, benzoni2006multi}. When the boundary is characteristic, a loss of control of the normal derivatives occurs in the sense that two  tangential (including time) derivatives are needed to control one normal derivative. Such phenomenons are treated, for instance, in the works of Secchi \cite{secchi1996well} and Gu\`es \cite{gues1990probleme}, concerning general quasi-linear hyperbolic systems. The situation is more favorable for the compressible Euler system with impermeable boundary condition. For example  in Schochet \cite{schochet1986compressible} and Ohkubo \cite{ohkubo1989well}, the structure of the vorticity equation is used to control the normal derivatives of the characteristic component and thus $H^m$-initial data yields a solution in $H^m$ space.
 Finally we mention that, in \cite{lannes20232} 
Iguchi and Lannes studied a coupled hyperbolic system where a fixed rigid body  is floating in a fluid described by the nonlinear shallow water equations.

%
%
%

\subsection{Main result}

We recall that we have defined  $  \Uscr \subset \mathbb{R}^2$ as  an open set such that $(p,s) \mapsto \rho(p, s)$, $ (p,s) \mapsto c(p, s)$ 
are smooth on $ \Uscr$ and such that $\rho>0$, $c>0$.

To quantify the regularity of  solutions in a time-space domain  $\mathcal{O}_T=(0, T)\times \mathcal{O}$,  for some smooth open set $\mathcal{O}$, we shall use  the space $X^m([0, T];\mathcal{O})$:
$$X^m([0, T];\mathcal{O}):=\bigcap_{k=0}^m C^k([0, T]; H^{m-k}(\mathcal{O}))$$ 
with associated norm 
$$\|f\|_{X^m_T(\mathcal{O})}=\sup _{t\in [0, T]}\interleave f(t)\interleave_{m,\mathcal{O}}, \quad \quad \interleave f(t)\interleave_{m,\mathcal{O}}=\sum_{i=0}^m\|\partial_t^i f\|_{H^{m-i}(\mathcal{O})}, $$
where $f$ can be a scalar or vector function. For the vector function $f$, the norm $\| f(t)\|_{X^m}$ is the sum of the $X^m$-norm of each component of $f$.
Note that we have in particular $H^{m+1}(\mathcal{O}_T)\subset X^m([0, T];\mathcal{O})\subset H^m(\mathcal{O}_T)$. 

Our main result which  is the existence of a local  classical solution of the fluid-rigid body system \rfb{ns-ob1}--\rfb{initialcon1} in the bounded domain $\Om$ reads:
\begin{thm}\label{main}
Let $m$ be an integer and $m\geq 3$. Suppose that
\begin{itemize}
	\item the initial data $(p_0, u_0, s_0, l_0,\omega_0)\in (H^m(\Fscr(0)))^3\times (\rline^3)^2$;
	\item $(p_0, u_0, s_0, l_0, \omega_0)$ satisfy the {\em compatibility conditions} up to order $m-1$ in the sense of Definition \ref{compatibility} below;
	\item 
	 For every $x\in \overline{\Fscr(0)}$, $(p_0(x), s_0(x))\in\Uscr$.
	 \end{itemize}  
Then there exists $T>0$ such that the coupled fluid-rigid body system \rfb{ns-ob1}--\rfb{initialcon1} has  a unique classical solution $(p, u, s, l, \omega)$ in the sense that after  the change of variable to a fixed spatial domain $\Fscr(0)$ defined below, the new unknowns satisfy
$$(\overline p, \overline u, \overline s, \overline l, \overline \omega)\in \left(X^m([0, T];\Fscr(0))\right)^3\times \left(C^m[0, T]\right)^2. $$
Moreover, for every $(t,x)\in [0, T]\times \overline{\Fscr(0)}$, we have $(\overline p(t,x), \overline  s(t,x))\in \Uscr$.
\end{thm}

\begin{rmk}
{\rm
We obtain only a short-time existence result. As with all fluid–rigid body systems, one restriction is that collision between the solid and the boundary of the fluid domain is not allowed.  Nevertheless, in addition, here the time interval must be chosen sufficiently small to prevent the formation of singularities in the fluid. Indeed, since the fluid is governed by the compressible Euler system, we cannot, in general, expect a global classical solution up to the time of collision with the boundary.  By using the finite speed of propagation property of the Euler equations, it is easy to see that the motion of the solid cannot prevent shock formation in the fluid. For instance, 
one can take initial data that leads to shock formation and place the solid in a large region 
of  constant pressure and zero velocity, so that at the time of singularity the solid will be still
immerged in a region 
of  constant pressure and zero velocity and thus is still not moving.
}
\end{rmk}
\begin{rmk}
{\rm
For the sake of simplicity,  we just treat one rigid body in this statement. The results in Theorem \ref{main} can be directly extended to the system with many bodies, as long as there is also no collision between the bodies.
To do so, it suffices to adapt the  cut-off function in  the change of variables. The details can be found in Remark \ref{manybodies}.
}
\end{rmk}
The first difficulty  in the analysis is to find a suitable change of coordinates  that reformulates the fluid-solid system \rfb{ns-ob1}--\rfb{initialcon1} in a fixed spatial domain. On the one hand, 
for inviscid fluids, along the solid boundary the tangential component of the fluid velocity is in general different from that of the solid. Consequently, the use of a Lagrangian map derived from the fluid's Lagrangian coordinates does not yield a time-independent domain, since the solid continues to move with nonzero tangential velocity (except in the very special case where the solid is a sphere).
On the other hand, the choice of the frame  moving with the solid  for performing the  change of variables, i.e. taking $x=Q^\intercal(t)(y-h(t))$, is not suited to fluid domains with boundaries as this makes  the exterior boundary  to be  moving  in the transformed frame. Moreover, this approach is also not well-suited to configurations involving multiple rigid bodies.
To overcome these issues, we introduce an artificial velocity field that smoothly transitions between the solid velocity on the solid boundary and the zero velocity on the exterior boundary of the fluid domain. The flow map associated with this vector field is then used to define the change of variables. More details are presented in Section \ref{changesec}. 

After reformulating the problem on a fixed spatial domain, we can derive the high-order a-priori estimates for the coupled model, using tools similar to those developed for the compressible Euler system in \cite{schochet1986compressible}. Based on the conservation of total energy in the coupled system, we first obtain high-order energy estimates for tangential and time derivatives.
 To recover normal derivatives, we can use the transport equation solved by the entropy and the vorticity on the one-hand
and  the fluid system  itself  to convert time and tangential derivatives into normal derivatives for the acoustic components close to the boundary
 on the other hand.
The main difficulty is then the construction of a solution.
A first  attempt would be to use an iterative approach, where assuming that
the solid motion is given, we solve the fluid system and then using in particular the obtained pressure we update the motion
of the solid by using \eqref{solid1}, \eqref{solid2}. Nevertheless, this scheme yields a loss of derivative at each step due to the lack
of control of the trace of the pressure on the boundary.
We shall thus  first use a non-characteristic approximation, in the sense that the fluid part  of the coupled system
 is modified through the addition of small transport terms  to make the boundary non-characteristic.  The resulting approximate system allows for control of the trace of the pressure on the boundary  and thus it can be solved eventually using the iterative scheme described above. We can then  derive uniform estimates with respect to the small perturbation parameter by adapting  the strategy of \cite{schochet1986compressible} above.
  Finally, we obtain the existence  of solution for the original fluid–solid coupled system by passing to the limit using  strong compactness arguments.
Note that the method of non-characteristic approximations has been adopted for instance in \cite{schochet1986compressible} (see also  \cite{lannes20232} for a more recent utilisation)  to construct solutions to the compressible Euler system in a bounded domain. This technique provides an effective way to develop existence theories for nonlinear hyperbolic systems, while circumventing the delicate analysis of  linear characteristic hyperbolic IBVP.

\subsection{Organization of the paper}

In Section \ref{changesec} we first explain the  change of variables to reformulate the fluid-solid system in a fixed spatial domain.  Then in Section \ref{approx} we introduce the  approximate hyperbolic system with non-characteristic boundary, and discuss  the effect on the compatibility conditions. We thereby build the  existence theory  for the approximate system.  Finally we derive in Section \ref{uniform} uniform estimates with respect to the
approximation parameter.
Section \ref{limit} is devoted to passing to the limit in  the approximate system and obtaining the solution for the original fluid-solid coupled system as well as deriving the desired regularity.

\subsection{Notations}\label{notation}

We shall use the following notations throughout the paper.

For any smooth bounded domain $\mathcal{O}\subset\rline^3$, the notation $L^p(\mathcal{O})$ and $W^{m,p}(\mathcal{O})$, for every $m\in \nline$ and $1\leq p<\infty$, represents the standard Lebesgue and Sobolev spaces, respectively. For $p=2$, as above we use $H^m(\mathcal{O})$.

To perform our energy estimates we will have to use  {\em conormal Sobolev spaces}.
They are defined in the following way.
 Let $(Z_k)_{1\leq k\leq N}$ be a finite set of smooth vector field which is a  generating  set  for $x \in \mathcal{O}$  and  which are tangent to the boundary of $\mathcal{O}$.  We shall also include the time derivative in the estimate and we set $Z_0=\partial_t$. 
Define 
$$Z^\beta=Z_0^{\beta_0}Z_1^{\beta_1}\cdots Z_N^{\beta_N}, \qquad \beta=(\beta_0, \beta_1,\cdots \beta_N)\in \nline^{N+1}.$$
For $p=2$ or $p=\infty$ we thereby introduce the following conormal Sobolev spaces:
$$L^p(0, T; H^m_{co}(\mathcal{O})):=\left\{f\in L^p(0,T; L^2(\mathcal{O})), Z^\beta f\in L^p(0, T; L^2(\mathcal{O}))\right\}, $$
endowed with the norm
\begin{equation}\label{norm}
\|f\|_{L^p_TH^m_{co}(\mathcal{O})}=\sum_{|\beta|\leq m}\|Z^\beta f\|_{L^p(0, T; L^2(\mathcal{O}))},
\end{equation}
where $|\beta|:=\beta_0+\beta_1+\cdots+\beta_N$. At a given time $t_0$, we use the semi-norms:
$$ \|f(t_0)\|_{H^m_{co}(\mathcal{O})}=\sum_{|\beta|\leq m}\|(Z^\beta f)(t_0)\|_{L^2(\mathcal{O})}.$$
For a vector function $u$, we say that $u\in L^p(0, T; H^m_{co}(\mathcal{O}))$ if each component of $u$ is in $L^p(0, T; H^m_{co}(\mathcal{O}))$, i.e.
$$\|u\|_{L^p_TH_{co}^m(\mathcal{O})}=\sum _{i=1}^3\sum_{|\beta|\leq m}\|Z^\beta u_i\|_{L^p(0, T; L^2(\mathcal{O}))}. $$

To define the  vector fields $(Z_k)_{1\leq k\leq N}$,  since $\mathcal{O}\in \rline^3$ is a smooth bounded domain, we can use the existence of a covering of $\mathcal{O}$ under the following form:
$$\mathcal{O}=\mathcal{O}_0\cup_{i=1}^n\mathcal{O}_i, \quad \mathcal{O}_0\Subset \mathcal{O}, \quad \mathcal{O}_i\cap \partial\mathcal{O}\neq \emptyset. $$
Assume that the boundary $\partial\mathcal{O}$ is described locally by $x_3=\varphi_i(x_1, x_2)$ in each chart such that we have
$$\mathcal{O}\cap \mathcal{O}_i=\left\{x=(x_1, x_2, x_3)| \m x_3>\varphi_i(x_1, x_2)\right\}, $$
and 
$$\partial\mathcal{O}\cap \mathcal{O}_i=\left\{x=(x_1, x_2, x_3)| \m x_3=\varphi_i(x_1, x_2)\right\}. $$

We denote by $\partial_{x_k}$ with $k=1,2,3$ the derivatives with respect to the original coordinates $(x_1, x_2, x_3)$ in $\rline^3$. In $\mathcal{O}_0$ we take directly the vector fields $\partial_{x_k}$ for $k=1,2,3$. Near the boundary we consider the local coordinates in each chart $\mathcal{O}_i$:
\begin{equation*}
\begin{aligned}
\Phi_i: (-\delta_i, \delta_i)\times (0, \epsilon_i)&\longrightarrow \mathcal{O}\cap \mathcal{O}_i\\
(y, z)&\mapsto \Phi_i(y,z):=(y, \varphi_i(y)+z).
\end{aligned}
\end{equation*}
A local basis is thus given by the vector fields $(\partial_{y_1}, \partial_{y_2}, \partial_z)$ with $\partial_{z}=\partial_{x_3}$. In these coordinates  we define the conormal vectors as follows:
\begin{equation*}
\begin{aligned}
Z_k^i&=\partial_{y_k}=\partial_{x_k}+\partial_{x_k}\varphi_i\partial_{z}, \qquad \text{for}~ k=1,2\\
Z_3^i&=\phi(z) \partial_{z},
\end{aligned}
\end{equation*}
where $\phi$ is smooth such that $\phi(0)=0$, $\phi(s)>0$ for $s>0$ (for example one can take  $\phi(s) = \frac{s}{1+s}$).

\begin{rmk}\label{understand}
{\rm 
In $\mathcal{O}_0$, $\|\cdot\|_{L^2_TH_{co}^m}$ defined in \rfb{norm} yields a control of the standard $H^m$-norm including time and space. In the case that  $\partial\mathcal{O}\cap \mathcal{O}_i\neq \emptyset$, there is no control of the normal derivatives near the boundary. 
}
\end{rmk}

We shall  use the notation $D^\beta$ to denote the full derivative either in the original coordinates
or in a coordinate patch near the boundary depending on the context: that is to say 
\begin{equation}\label{D}
	D^\beta=\partial_t^{\beta_0}\partial_{x_1}^{\beta_1}\partial_{x_2}^{\beta_2}\partial_{x_3}^{\beta_3},
\qquad   \mbox{ or } \qquad 
D^\beta =\partial_t^{\beta_0}\partial_{y_1}^{\beta_1}\partial_{y_2}^{\beta_2}\partial_z^{\beta_3},
\end{equation}
with $\beta=(\beta_0,\beta_1, \beta_2, \beta_3)\in \nline^{4}$ and $|\beta|=\beta_0+\beta_1+\beta_2+\beta_3$ to include time
and all the space derivatives.

\bigskip

Finally the notation $C[ \cdots ]$ will stand for a harmless  continuous non-decreasing function of each of its arguments.


\section{Formulation in a time independent domain}\label{changesec}
In this section, we aim to rewrite the fluid-rigid body system \rfb{ns-ob1}--\rfb{initialcon1} in a fixed spatial domain. 

Before doing this, we mention some conventions for the notation.
For a matrix $M$, we denote by $M^\intercal$ the transpose of $M$. If a function $f$ only depends on time $t$, we denote by $f'$ its derivative with respect to $t$; and if $f$ depends on both time and space, we use $\prt_t f$ for its partial time-derivative. We use ${\bf 0}$ in boldface to represent the matrix  consisting of zero entries, which can be row vector, column vector or square matrix depending on the context.

Assume that at $t=0$ the solid does not touch the boundary $\partial\Om$, i.e. $\dist(\partial\Sscr(0), \partial\Om)>0$. We take $R_0<\dist(\partial\Sscr(0), \partial\Om)$ such that
$$\Sscr(0)\subset \Sscr(0)+\overline{B(0, R_0)}\subset \Sscr(0)+\overline{B(0, 2R_0)}\subset \Om. $$
Based on this, we use  a cut-off function $\chi\in C_c^\infty(\rline^3)$ such that
\begin{equation}\label{cutoff}
\chi(y):=\left\{\begin{aligned}
&\m 1 \quad &y\in \Sscr(0)+\overline{B(0, R_0)},\\
&\m 0 \quad &y\in \rline^3\setminus (\Sscr(0)+\overline{B(0, 2R_0)}).
\end{aligned}\right.
\end{equation}
Now we define 
$$V(l(t), \omega(t), h(t),y)=\chi(y-h(0))\m u_\Sscr(t,y)=\chi(y)\m u_\Sscr(t,y), $$
where $h(0)=0$ and $u_\Sscr$ is the solid velocity introduced in \rfb{us}. With this we define the corresponding flow, denoted by $\phi$, which is the solution of the equation:
\begin{equation}\label{flow}
\left\{\begin{aligned}
&\partial_t\phi(t,x)=V(l(t), \omega(t), h(t), \phi(t,x)),\\
&\phi(0, x)=x.
\end{aligned}\right.
\end{equation}
Note that for small times $0<t<t_0$, 
the map $\phi(t,\cdot)$ defined on $\mathbb{R}^3$ is a diffeomorphism.
Moreover,  the solid $\Sscr(t)$ will  stay in $\Sscr(0)+\overline{B(0, R_0)}$. Using the definition \rfb{cutoff}, we see that, for $0<t<t_0$, $\phi(t,x)=h(t)+Q(t) x$ solves the equation \rfb{flow} for $ x\in  \Sscr(0)$ so that $\phi(t, \cdot)$ maps   $\mathcal{S}(0)$ onto $\mathcal{S}(t)$. 
Moreover, $\phi(t,\cdot)$ is the identity near the boundary of $\Omega$. We thus see 
that the flow $\phi(t, \cdot) $ actually gives us the desired change of variables:
\begin{equation}\label{change}
\begin{aligned}
\phi(t,\cdot): \Fscr(0) &\longrightarrow \Fscr(t)\\
x&\mapsto y=\phi(t,x).
\end{aligned}
\end{equation}

Based on the change of variable \rfb{change}, we derive in the following the new fluid-rigid body system in the fixed spatial domain $\Fscr(0)$. We denote the Jacobian matrix of the transform $\phi^{-1}$ by $\Jscr$, i.e. $\Jscr(t,x)=(\nabla \phi(t,x))^{-1}$. Then we introduce the following new unknowns:
\begin{equation}\label{function}
\begin{aligned}
&\overline u(t,x):=\Jscr(t,x)\m u(t,\phi(t,x)),
\qquad  &\overline u_\Sscr(t,x):=\Jscr(t,x) \m\partial_t\phi (t,x),\\
&\overline p(t,x):=p(t, \phi(t,x)),\quad &\overline s(t,x):=s(t,\phi(t,x)),\\
&\overline l(t):=Q(t)^\intercal\m l(t),\qquad &\overline \omega(t):=Q(t)^\intercal\m \omega(t).
\end{aligned}
\end{equation}
Moreover, we use the notation 
\begin{equation*}\label{alphaeta}
	\eta(t,x):=\rho(\overline p, \overline s),
	\qquad \alpha(t,x):=\rho c^2(\overline p, \overline s),
\end{equation*}
which are smooth functions with respect to $\overline p$ and $\overline s$.  

By using the chain rule we rewrite the fluid-rigid system \rfb{ns-ob1}--\rfb{initialcon1} in terms of the functions defined in \rfb{function}, for every  $x\in \Fscr(0)$, as follows:
\begin{equation}\label{newsystem}
\left\{\begin{aligned}
&\partial_t\overline p+\left((\overline u-\overline u_\Sscr)\cdot \nabla\right) \overline p+\alpha \div \overline u=-\alpha\m \text{tr}(\Jscr_{2}  \overline u),\\
&\partial_t\overline u+\left((\overline u-\overline u_\Sscr)\cdot \nabla\right)\overline u+ (\eta M)^{-1}\nabla\overline p= - \Jscr_{1}^{-1} \Jscr_{2}\m \overline u\cdot (\overline u-\overline u_\Sscr)-\nabla_{y}V\overline u,\\
&\partial_t\overline s+\left((\overline u-\overline u_\Sscr)\cdot \nabla\right)\overline s=0,\\
&m\m\overline l'(t)=m\overline l\times \overline \omega+\int_{\partial\Sscr(0)}\overline p\m n_0 \m \dd \Gamma,\\
&J_0 \m \overline \omega'(t)=(J_0\overline \omega)\times \overline \omega+\int_{\partial\Sscr(0)} x\times \overline p n_0\dd \Gamma,\\
&h'(t)=Q(t)\overline l(t),\\
&\partial_t\phi(t, x)=V(\overline l(t), \overline \omega(t), h(t),\phi(t,x)),\\
&\partial_t \Jscr_{1}(t,x)=\nabla_{y}V( \overline l(t), \overline \omega(t), h(t), \phi(t,x))\cdot \Jscr_{1},\\
& \partial_{t} \Jscr_{2}(t,x)= \nabla_{y}V(\overline l(t), \overline \omega(t), h(t), \phi(t,x))\cdot \Jscr_{2} + \nabla_{y}^2V( \overline l(t), \overline \omega(t), h(t), \phi(t,x))\cdot \begin{bmatrix} \Jscr_{1}  & \Jscr_{1}\end{bmatrix}, \\
&Q'(t)=Q(t)D_{\overline \omega},
\end{aligned}\right.
\end{equation}
with the boundary condition 
\begin{equation}\label{boun2}
\overline u\cdot n_0=\overline u_\Sscr\cdot n_0\qquad \forall \m x\in \partial\Fscr(0),
\end{equation}
and the initial data
\begin{equation}\label{data}
\begin{aligned}
	(\overline p, \overline u, \overline s)(0,x)&=( p_0, u_0, s_0)(x),\\
	(\overline l(0), \overline \omega(0), h(0), \phi(0, x), \Jscr_{1}(0, x), \Jscr_{2}(0, x), Q(0))&=(l_0, \omega_0, 0, x, I_{3\times 3}, {\bf 0}_{3\times 3 \times 3},  I_{3\times 3}).
\end{aligned}
\end{equation}
Note that in the extended system  \rfb{newsystem}, we necessarily have  by ODE uniqueness  $\Jscr_{1}(t,x)= \nabla \phi(t,x)$ and  $\Jscr_{2}(t,x)$
the three-indices tensor given  by 
$ \Jscr_{2}(t,x)= \nabla^2 \phi(t,x)$. In addition $\Jscr_{1}$ will be invertible as long as $\overline{l}$ and  $\overline{\omega}$ are well defined.
For time small enough so that $\phi(t, \cdot)$ is invertible, we must have 
 $\Jscr_{1}= \Jscr^{-1}.$
 
The skew-symmetric angular velocity matrix  $D_{\overline \omega}$ associated with $\overline \omega$ and the $3\times 3$ symmetric matrix $M$ are defined by
$$\forall x\in \rline^3,\quad D_{\overline \omega}x=\overline\omega\times x,\qquad  \qquad M=\Jscr^{-\intercal}\m\Jscr^{-1}.$$

In   the boundary condition \eqref{boun2},  $n_0(x)$ is the unit outer normal along $\partial\Fscr(0)(=\partial\Sscr(0)\cup \partial\Om)$.
Since the solid is not deformable, for $x \in \partial \Sscr(0)$ we have $\phi(t,x) \in \partial \Sscr(t)$ and $$n(t, \phi(t,x))=\Jscr_1(t,x)n_{0}(x)= Q(t) n_{0}(x),$$
which indicates that 
 the boundary condition  \eqref{boun2} follows from  \eqref{boundarycon1}.

\begin{rmk}\label{usexplain}
{\rm Note that thanks  to the definition \rfb{flow}, the vector field  $\overline u_\Sscr$ in \rfb{function} is globally defined in $\Om$ and the system \rfb{newsystem} includes in a compact way the boundary conditions on $\partial\Om$ and $\partial\Sscr(t)$. In particular for small times, we have
\begin{equation}
\label{usaubord}
\overline u_\Sscr(t,x)=\left\{
\begin{aligned}
&\overline l(t)+\overline\omega(t)\times x\quad &\qquad \forall \m  x\in \Sscr(0),\\
&0 &\qquad \forall \m  x\in \partial\Om.
\end{aligned}\right.
\end{equation}	
For the same reason,  we get  that near $\partial\Om$, we have for the  Jacobian $\Jscr$  that  $\Jscr=I_{3\times 3}$ and near $\partial\Sscr(0)$  that  $\Jscr=Q^\intercal(t)$. In particular this yields 
 in a vicinity of the whole boundary $\partial \mathcal{F}(0)$, 
\begin{equation}\label{Mbound}
	M(t, x)=\nabla\phi^\intercal\nabla\phi=I_{3\times 3} \quad \text{near } \partial\Fscr(0),
\end{equation}
since $Q(t)\in SO(3)$.

}
\end{rmk}

\begin{rmk}\label{manybodies}
{\rm To handle  the case of many bodies $(\Sscr_i)_{i=1}^N$,  assuming  that at $t=0$ the bodies are separated from each other and from the boundary $\partial\Omega$ by a positive distance.
For fixed $i$ we take the distance $R_i$ satisfying 
$ 0<R_i<\min\left\{\dist\left(\partial\Sscr_i(0), \partial\Omega\right), \dist\left(\partial\Sscr_i(0), \partial\Sscr_j(0)\right)\m\forall \m  j\neq i\right\}$, such that
$$\Sscr_i(0)\subset \Sscr_i(0)+\overline{ B(0,R_i)}\subset  \Sscr_i(0)+\overline{B(0, 2R_i)}\subset \Omega,$$
and 
$$\left(\Sscr_i(0)+\overline{B(0, 2R_i)}\right)\cap \left(\Sscr_j(0)+\overline{B(0, 2R_j)}\right)=\emptyset\qquad \forall \m  j\neq i. $$
Then for any $\Sscr_i$ we define the cut-off function $\chi_i\in C_c^\infty(\rline^3)$ as below:
\begin{equation*}
\chi_i(y):=\left\{\begin{aligned}
&1  &y\in \Sscr_i(0)+\overline{B(0, R_i)},\\
&0 &y\in \rline^3\setminus (\Sscr(0)+\overline{B(0, 2R_i)}).
\end{aligned}\right.
\end{equation*}
Finally, it suffices to replace the artificial velocity  $V$ in the Lagrange transform \rfb{flow} by $\sum_{i=1}^N\chi_i(y)u_{\Sscr_i}(t,y)$. This in particular allows that the bodies move in the fluid domain with distinct velocities.    
	}
\end{rmk}


Now we rewrite the system \rfb{newsystem} into the abstract form of a hyperbolic system coupled with an ODE. To do this, for $j=1,2,3$ we introduce some matrix notations 
$$U=\begin{bmatrix} \overline p \\ \overline u \\ \overline s \end{bmatrix},\qquad \qquad \qquad
\Theta=\begin{bmatrix} \overline l \\ \overline \omega \end{bmatrix}, \qquad \qquad \qquad \Upsilon=\begin{bmatrix} h \\ \phi \\ \Jscr_{1}  \\ \Jscr_{2} \\ Q \end{bmatrix}, $$
\begin{equation}\label{matrices}
\begin{aligned}
A^0(\Theta,\Upsilon,U)&=\begin{bmatrix}
		\alpha^{-1} & {\bf 0} & 0 \\ {\bf 0} & \eta M & {\bf 0}\\ 0 & {\bf 0}& 1
	\end{bmatrix},\\	
A^j(\Theta, \Upsilon ,U)&=
\begin{bmatrix}
\alpha^{-1}(\overline u^j-\overline u_\Sscr^j) & \begin{bmatrix} \delta_{1j} & \delta_{2j} & \delta_{3j} \end{bmatrix} & 0\\
\begin{bmatrix} \delta_{1j}& \delta_{2j}& \delta_{3j}\end{bmatrix}^\intercal & \eta M(\overline u^j-\overline u_\Sscr^j) & {\bf 0} \\
0 & {\bf 0} & \overline u^j-\overline u_\Sscr^j
\end{bmatrix},\\
B(\Theta, \Upsilon,U)&=\begin{bmatrix} 0  & \m \text{tr}(\Jscr_{2}
	\m \cdot) & 0	\\ 
	0&	\eta M \left((\overline u-\overline u_\Sscr)\Jscr_{1}^{-1} \m  \Jscr_{2}\cdot+ \nabla_{y} V\cdot\right) & 0 \\ 
	0 & {\bf 0}& 0\end{bmatrix},
\end{aligned}
\end{equation}
where $M_{ij}$ respresents the entry of the matrix $M$ in row $i$ and column $j$. 


With the above notation, we reformulate the fluid equations \rfb{newsystem}  into the following form, for  $x\in\Fscr(0)$,
\begin{equation}\label{first}
\left\{\begin{aligned}
&A^0(\Theta, \Upsilon,U)\partial_t U+\sum_{j=1}^3 A^j(\Theta, \Upsilon,U)\partial_j U+B(\Theta, \Upsilon,U)U= 0, \\
& U(0, x)=U_0(x)=\begin{bmatrix} p_0 & u_0 & s_0 \end{bmatrix}^\intercal(x),\\
&G(x) U=\overline u_\Sscr\cdot n_0 \qquad \forall \m  x\in \partial\Fscr(0),
\end{aligned}\right.
\end{equation}
where the boundary matrix $G(x)$ is given by
\begin{equation}\label{G}
G(x)=\begin{bmatrix} 0 & n_0(x)^\intercal& 0\end{bmatrix}.
\end{equation}
The boundary data $\overline u_\Sscr$ is determined by $(\Theta, \Upsilon)$ thanks to \eqref{usaubord} 
which  further solves
\begin{equation}\label{solidbar}
\left\{\begin{aligned}
&\partial_t\begin{bmatrix} \Theta \\ \Upsilon \end{bmatrix}= \begin{bmatrix} R_\Theta(\Theta, U) \\ R_{\Upsilon}(\Theta, \Upsilon) \end{bmatrix},\\
&\begin{bmatrix} \Theta(0) \\ \Upsilon(0, x) \end{bmatrix}=\begin{bmatrix} \Theta_0 \\ \Upsilon_0(x)\end{bmatrix},
\end{aligned}\right.
\end{equation}
where 
\begin{equation}
\label{Rthetadef}
R_\Theta(\Theta, U)=\left\{\begin{aligned}
&{\overline l}\times {\overline \omega}+m^{-1}\int_{\partial\Sscr(0)}\overline p\m n_0 \m \dd \Gamma,\\
&\overline \omega\times (J_0^{-1}\overline \omega)+J_0^{-1}\int_{\partial\Sscr(0)}x\times \overline p n_0\dd \Gamma,
\end{aligned}\right. 
\end{equation}
\begin{equation}
\label{Rupsilondef}
R_\Upsilon(\Theta, \Upsilon)=\begin{bmatrix} Q(t)\overline l(t) \\ V(\overline l(t),\overline \omega(t), h(t), \phi(t,x)) \\ \nabla_y V(\overline l(t), \overline \omega(t), h(t),  \phi(t,x))\cdot\Jscr_{1}\\ 
\nabla_{y}V(\overline l(t), \overline \omega(t), h(t), \phi(t,x))\cdot \Jscr_{2} + \nabla_{y}^2V( \overline l(t), \overline \omega(t), h(t), \phi(t,x))\cdot [\Jscr_{1}, \Jscr_{1}]\\
  Q(t)D_{\overline \omega} \end{bmatrix}.
\end{equation}
and the initial data in \rfb{solidbar} is
\begin{equation}\label{initialdata}
\begin{aligned}
\Theta_0=\begin{bmatrix} l_0 \\ \omega_0 \end{bmatrix}, \qquad \Upsilon_0(x)=\begin{bmatrix} 0 \\ x \\ I_{3\times 3}, \\ {\bf 0}_{3\times 3 \times 3} \\ I_{3\times 3} \end{bmatrix}.
\end{aligned}
\end{equation}

 Recalling the definition of $\overline u_\Sscr$ in \rfb{function} and in particular its simple form on the boundary \eqref{usaubord}, we see that the fluid equation \rfb{first} is coupled  with the solid equations \rfb{solidbar} through $\overline u_\Sscr$ in the coefficients  of the matrices $A^0$, $A^j$, $B$ and at the boundary $\partial\Fscr(0)$. 
 An important observation is that  $\overline u_\Sscr$ is smooth in space, while its time regularity is limited since it depends on time
 through $\overline{l}$, $\overline{\omega}$ which themselves depend on the pressure $\overline{p}(t)$ at the boundary.

Finally,  we observe that   the  normal boundary matrix $A_n(\Theta, \Upsilon,U):=\sum_{i=1}^3A^j(\Theta, \Upsilon,U)n_{0, j}$ at $\partial\Fscr(0)$
is given by:
\begin{equation*}\label{nonchar}
A_n(\Theta, \Upsilon, U)=\begin{bmatrix} 0 & n_0^\intercal & 0 \\
	n_0 & {\bf 0} & {\bf 0} \\
	0 & {\bf 0} & 0 
\end{bmatrix}.
\end{equation*}
The multiplicity of the zero eigenvalue is $d=3$ 
which implies that the boundary of the system \rfb{first} is characteristic (both the external boundary and the solid boundary).

\medskip

Now that we have reformulated the system, we can discuss the compatibility conditions for the coupled system \rfb{first}--\rfb{solidbar}. 

For smooth solutions of the system \rfb{first}--\rfb{solidbar}, $\partial_t^k U(0)$, $\partial_t^k\overline l(0)$ and $\partial_t^k\overline \omega(0)$ are determined  inductively by $U_0$, $l_0$ and $\omega_0$ by using the system \eqref{first} and \rfb{solidbar}. 
Differentiating the equations $k$-times with respect to $t$, we have a recursion relation at $t=0$:
\begin{align}
&
\partial_t^k U(0)=-\sum_{i=0}^{k-1}\begin{pmatrix} k-1 \\ i \end{pmatrix}\left\{\partial_t^i\left((A^0)^{-1}A^j\right)(\Theta, \Upsilon,U)\partial_j+\partial_t^i\left((A^0)^{-1}B\right)(\Theta, \Upsilon, U)\right\}(0)\partial_t^{k-1-i}U(0),\nonumber\\
&\partial_t^k\overline l(0)=\sum_{i=0}^{k-1}\begin{pmatrix} k-1 \\ i \end{pmatrix}\left\{\partial_t^i\overline l\times \partial_t^{k-1-i}\overline \omega\right\}(0)+\frac{1}{m}\int_{\partial\Sscr(0)}\partial_t^{k-1}\overline p(0) n_0\dd\Gamma,\label{compatibegin}\\
&\partial_t^k\overline \omega(0)=\sum_{i=0}^{k-1}\begin{pmatrix}k-1 \\ i \end{pmatrix}\left\{\partial_t^i\overline \omega\times (J_0^{-1}\partial_t^{k-1-i}\overline \omega)\right\}(0)+J_0^{-1}\int_{\partial\Sscr(0)}x\times \partial_t^{k-1}\overline p(0)n_0\dd\Gamma.\nonumber
\end{align}
Observe that by using  the equation for $\Upsilon$ in \eqref{solidbar} and the definition of $R_{\Upsilon}$ in \eqref{Rupsilondef}, 
we can express the successive derivative $\partial_{t}^k \Upsilon(0)$ in terms of the  derivatives $(\partial_{t}^j \Theta (0))_{0\leq j \leq k-1}$, which in turn can be computed by using the above expression.
Together with \rfb{data}, we notice that $\partial_t^kU(0)$, $\partial_t^k\overline l(0)$ and $\partial_t^k\overline \omega(0)$ only depend on the initial data $(U_0, l_0, \omega_0)$.
We thus define a vector from the resulting expressions:
\begin{equation*}
\mathcal{I}^k( U_{0}, l_{0}, \omega_{0})(x):= (\mathcal{I}^k_{p}, \mathcal{I}^{k}_u, \mathcal{I}^k_{s},  \mathcal{I}^{k}_{l}, \mathcal{I}^{k}_\omega)^\intercal(U_{0}, l_{0}, \omega_{0}) := 
(\partial_{t}^kU(0), \partial_{t}^k\overline{l}(0), \partial_{t}^k \overline{\omega}(0))^\intercal.
\end{equation*}

The boundary condition in \rfb{first} implies that a smooth enough solution should satisfy  $\partial_t^k\overline u(t)\cdot n_0=(\partial_t^k\overline l+\partial_t^k\overline \omega\times x)\cdot n_0$ for $x\in\partial\Sscr(0)$ and $\partial_t^k\overline u(t)\cdot n_0=0$ for $x \in \partial \Omega$. At $t=0$ and $x\in\partial\Fscr(0)$, smooth enough solutions must therefore satisfy 
\begin{equation}\label{compat}
\mathcal{I}^k_{u}\cdot n_0=	
\left\{\begin{aligned}
&\bigl(\mathcal{I}^k_{l}+ \mathcal{I}^k_{\omega}\times x\bigr)\cdot n_0 \, &x \in\partial \Sscr(0),\\ & 0\, &x \in \partial \Omega.
\end{aligned}\right.
\end{equation}

To avoid confusion, from now on we use the notation $\mathcal{I}^k_f$ for the expressions involving the initial data corresponding to  the computation of the   successive time derivatives at $t=0$ arising from the recursion of the equation, and the usual notation $\partial_t^k f(0)$ for the $k$-th time derivative of $f$ evaluated at $t=0$.

\begin{definition}[Compatibility conditions]\label{compatibility}
Let $m\geq 3$ be an integer. For the initial boundary value problem \rfb{first}--\rfb{solidbar}, we say that the data $(U_0, l_0, \omega_0)$ with $U_0\in H^m(\Fscr(0))$ satisfy the compatibility condition of order $m-1$  if\  \rfb{compat} is verified for every $k \leq m-1$. 
\end{definition}

\section{Non-characteristic approximation}\label{approx}

  In this section we shall  approximate \rfb{newsystem} by another  system with a non-characteristic boundary.   We first define the approximate model by taking into account the procedure of modifying  the compatibility conditions, then we explain how  the existence theory for the coupled  approximate system can be settled by using the better regularity of the pressure.

In order to justify energy estimates in the following, it is necessary to have smoother initial data. Inspired by \cite{schochet1986compressible} we approximate the initial data $U_0\in H^m(\Fscr(0))$ by $U_0^n=\begin{bmatrix} p_0^n & u_0^n & s_0^n \end{bmatrix}^\intercal\in H^{m+2}(\Fscr(0))$, such that $U_0^n \in \mathcal{C}$ where $\mathcal{C} \subset \mathcal{U}$ is a fixed compact subset (independent of $n$),  $(U_0^n, l_0,\omega_0)$ satisfy the compatibility condition of order $m$ such that the convergence $U^n_0\longrightarrow U_0$ happens in $H^m(\Fscr(0))$\footnote{More details about this approximation can be found in Appendix A.}.  With the initial data $(U_0^n, l_0, \omega_0)\in H^{m+2}(\Fscr(0))\times (\rline^3)^2$ we derive from
the equations \rfb{first}--\rfb{solidbar} that $\mathcal{I}^k( U_{0}^n, l_{0}, \omega_{0})\in H^{m+2-k}(\Fscr(0))\times (\rline^3)^2$, for $0\leq k\leq m+2$.  According to Theorem 2.5.7 in \cite{hormander1963linear}, for  every $T_L>0$ there exists $\tilde U^n=\begin{bmatrix} \tilde p^n & \tilde u^n & \tilde s^n\end{bmatrix}^\intercal\in H^{m+2}((0, T_L)\times \Fscr(0))$ satisfying 
\begin{equation}\label{sourceapprox}
\partial_t^k\tilde U^n(0)=\mathcal{I}^k_{U}(U^n_{0}, l_0, \omega_0)  \qquad\forall \m 0\leq k\leq {m+1},\m\m x\in\Fscr(0).
\end{equation}
Moreover, we extend the normal vector $n_0$ in \rfb{boun2} into a smooth vector field on  the whole domain $\Fscr(0)$ denoted by  $\nu$, 
we thus have   $\nu\in C^\infty(\Fscr(0))$ and $\nu|_{\partial\Fscr(0)}=n_0$.

Now we introduce the approximate system. We make a perturbation directly in the notation defined in the previous section as:
$$U^\eps=\begin{bmatrix} \overline p^\eps \\ \overline u^\eps \\ \overline s^\eps \end{bmatrix},\qquad \qquad \qquad
\Theta^\eps=\begin{bmatrix} \overline l^\eps \\ \overline \omega^\eps \end{bmatrix}, \qquad \qquad \qquad \Upsilon^\eps=\begin{bmatrix} h^\eps \\ \phi^\eps \\ \Jscr_{1}^\eps \\ \Jscr_{2}^\eps \\ Q^\eps \end{bmatrix}.$$
We construct the approximate system in the following form:
\begin{equation}\label{forcc}
\left\{\begin{aligned}
&A^0(\Theta^\eps, \Upsilon^\eps,U^\eps)\partial_t U^\eps+\sum_{j=1}^3 \Ascr^j(\Theta^\eps, \Upsilon^\eps, U^\eps)\partial_j U^\eps+B(\Theta^\eps, \Upsilon^\eps,U^\eps)U^\eps= F(t, x), \\
& U^\eps(0, x)=U_0^n(x)=\begin{bmatrix} p^n_0 & u^n_0 & s^n_0 \end{bmatrix}^\intercal(x),\\
&G(x) U^\eps=\overline u_{\Sscr}^\eps\cdot n_0 \qquad \forall \m  x\in \partial\Fscr(0).\\
\end{aligned}\right.
\end{equation}
The boundary data $\overline u_\Sscr^\eps$ is  defined  following \eqref{function} by 
$$  \overline u_\Sscr^\eps= \Jscr_{1}^{\eps}(t,x)^{-1}V(l^\eps(t), \omega^\eps(t), h^\eps(t), \phi^\eps(t,x)),$$
which is thus determined by  $\Theta^\eps$ and $\Upsilon^\eps$ satisfying
\begin{equation}\label{solideps}
\left\{\begin{aligned}
&\partial_t\begin{bmatrix} \Theta^\eps \\ \Upsilon^\eps \end{bmatrix}= \begin{bmatrix} R_{\Theta}(\Theta^\eps, U^\eps) \\ R_{\Upsilon}(\Theta^\eps, \Upsilon^\eps) \end{bmatrix},\\
&\begin{bmatrix} \Theta^\eps(0) \\ \Upsilon^\eps(0, x) \end{bmatrix}=\begin{bmatrix} \Theta_0 \\ \Upsilon_0(x)\end{bmatrix}.
\end{aligned}\right.
\end{equation}
In \rfb{forcc}, $G(x)$ has been introduced in \rfb{G}, the matrices $\Ascr^j(\Theta^\eps, \Upsilon^\eps,U^\eps)$ and $F(t,x)$ are in particular

\begin{align}
 \label{matrice2}
 \Ascr^j(\Theta^\eps, \Upsilon^\eps,U^\eps)&=A^j(\Theta^\eps, \Upsilon^\eps,U^\eps)+	\eps\nu^j
\begin{bmatrix}
		1 & {\bf 0} & 0\\
		{\bf 0}& I_{3\times 3} & {\bf 0}\\
		0 & {\bf 0} & 1
			\end{bmatrix},
\\
\label{Ftheta}	F(t,x)&=\begin{bmatrix}
		\eps(\nu\cdot\nabla)\tilde p^n \\ \eps(\nu\cdot\nabla)\tilde u^n \\ \eps(\nu\cdot\nabla) \tilde s^n 
	\end{bmatrix}.
\end{align}
In \rfb{solideps}, $R_{\Theta}$ and $R_{\Upsilon}$ are defined  in \eqref{Rthetadef} and \eqref{Rupsilondef}, respectively. The initial data for  $\Theta^\eps(0)$ and $\Upsilon^\eps(0, x)$  are still given by  \rfb{initialdata}, which are independent of $\eps$.


\medskip

In the remaining part of this section, we focus on establishing the well-posedness of the coupled system \eqref{forcc}--\eqref{solideps}. Roughly speaking, the strategy is to fix the boundary data and the coefficients depending on the solid
motion in the matrices  $A^0(\Theta^\eps, \Upsilon^\eps,U^\eps)$, $\Ascr^j(\Theta^\eps, \Upsilon^\eps,U^\eps)$ and $B(\Theta^\eps, \Upsilon^\eps,U^\eps)$ in \rfb{forcc}, i.e. for given $\overline l^\eps$ and $\overline\omega^\eps$, we solve  the fluid system  \rfb{forcc} and  get  a unique solution $U^\eps$.  In particular, this provides the pressure on the boundary of the solid, which is then used to update the solid's motion in \rfb{solideps}.
The existence of a solution to the coupled system \eqref{forcc}--\eqref{solideps} is thus obtained via a fixed-point argument applied to this iterative procedure.

Note that in the system \eqref{forcc}--\eqref{solideps} for any given initial data $U_{0}^n$, $l_{0}$, $\omega_{0}$ and
source term $F$, we can compute the successive time derivatives of $U^\eps$ at $t=0$, by using the new system \eqref{forcc}
 and thus get an appropriate notion of compatibility conditions.
For every $0\leq k\leq m$, we get that  $\partial_t^k U^\eps(0)$ can be expressed  as below:
\begin{equation*}
\begin{aligned}
	\partial_t^k U^\eps(0)=&-\sum_{i=0}^{k-1}\begin{pmatrix} k-1 \\ i \end{pmatrix}\left\{\partial_t^i\left((A^0)^{-1}\Ascr^j\right)(\Theta^\eps, \Upsilon^\eps,U^\eps)\partial_j+\partial_t^i\left((A^0)^{-1}B\right)(\Theta^\eps, \Upsilon^\eps,U^\eps)\right\}(0)\partial_t^{k-1-i}U^\eps(0)\\
	&+\sum_{i=0}^{k-1}\begin{pmatrix} k-1 \\ i \end{pmatrix}\partial_t^i(A^0)^{-1}(\Theta^\eps, \Upsilon^\eps,U^\eps)(0)\partial_t^{k-1-i}F(0).
\end{aligned}
\end{equation*}
While $\partial_t^k\overline l^\eps(0)$ and $\partial_t^k\overline \omega^\eps(0)$ are defined similarly as in \rfb{compatibegin}, by replacing $\overline p$, $\overline l$ and $\overline \omega$ by $\overline p^\eps$, $\overline l^\eps$ and $\overline \omega^\eps$, respectively, 
and  the derivative of $\Upsilon^\eps$ at $t=0$ can be expressed in terms of the successive derivatives of $\Theta^\eps$ at zero
with the same expressions as before.
This yields as above  a sequence of functions  $\mathcal{I}^{k, \eps}(U^n_{0}, l_{0}, \omega_{0}, F^k)$ depending
only on the initial data and  $F^k$ defined by $(F(0), \partial_{t}F(0), \cdots, \partial_{t}^{k-1} F(0))$.
 The compatibility conditions of order $m$ for the system \eqref{forcc}--\eqref{solideps} are, for every $k \leq m$,
 \begin{equation} \label{compateps}
  \mathcal{I}^{k, \eps}_{u} \cdot n_{0}=	\left\{\begin{aligned}
 	&(\mathcal{I}^{k, \eps}_{l} + \mathcal{I}^{k, \eps}_{\omega} \times x) \cdot n_{0}, \,& x
 	\in \partial\mathcal{S}(0),\\
 	&0, \, &x \in \partial \Omega.
 	\end{aligned}\right.
 \end{equation}

Recalling that the source term $F$ is given by \eqref{Ftheta}, with the initial data $U_0^n$, $ l_0$, $\omega_0$, we observe that
$$ \mathcal{I}^{k, \eps}(U_{0}^n, l_{0}, \omega_{0}, F^k)=\mathcal{I}^k(U_{0}^n, l_{0}, \omega_{0}).$$
This implies that the compatibility conditions of order $m-1$ and $m$ for the original coupled system \rfb{first}--\rfb{solidbar} are also satisfied.

Our goal is to prove the following  existence result for the  system \rfb{forcc}--\rfb{solideps}.

To state the result, it  is convenient to define another open set $\mathcal{V}$ such that $\overline{\mathcal{V}}$
 is compact and
\begin{equation*}
\label{Vposition}
\mathcal{C} \subset \mathcal{V} \subset \overline{\mathcal{V}} \subset \mathcal{U}
\end{equation*}
and to define  for $U= (p, u, s) \in \mathbb{R}^{d+2}$,   $\pi U$ the projection on the first and
last component so that $\pi U= (p, s)$.

\begin{thm}\label{existeps}
For every $\eps\in (0,1]$ fixed and $n\in\nline$ fixed, let $U_0^n\in H^{m+1}(\Fscr(0))$,  $F \in H^{m+1}((0, T_L)\times \mathcal{F}(0))$  
 such that  $(U_0^n, l_0, \omega_0, F)$ satisfies the compatibility condition of order $m$ in \rfb{compateps}  and that for every $x\in \mathcal{F}(0),$  $\pi U_{0}^n(x) \in \mathcal{C}$.
Then there exists $T^\varepsilon>0$,  such that the system \rfb{forcc}--\rfb{solideps} has  a unique solution $(U^\eps, \Theta^\eps, \Upsilon^\eps)$ satisfying
$$ (U^\eps, \Theta^\eps)\in  (X^{m+1}([0, T^\varepsilon]; \Fscr(0)))^3 \times (C^{m+1}[0, T^\varepsilon])^2$$
and $\pi U^\eps (t,x) \in \mathcal{V}$, for every $(t,x) \in (0, T^\eps)\times \mathcal{F}(0).$
Moreover, for every $T\leq T^\varepsilon$, we have the estimate
\begin{equation}
\label{continuationcriterion}
 \| \Theta^\eps \|_{C^{m+1}([0, T])} + \| U^\varepsilon\|_{X^{m+1}_{T}}
   \leq C_{\varepsilon}\left[ \|U_{0}^n \|_{H^{m+1}}, \|F \|_{H^{m+1}((0, T) \times \mathcal{F}(0))} ,  \| \Theta^\eps \|_{C^{m}([0, T])}, \| U^\varepsilon\|_{X^m_{T} }\right].
\end{equation}
\end{thm}

\begin{rmk}
{\rm
Note that the existence time in the above result depends on $\eps$. 
From the local existence result and the estimate \eqref{continuationcriterion}, we deduce that the solution can be continued in $X^{m+1}_{T}$
as long as  $\pi U^\eps (t,x) \in  \mathcal{U}$ and that the norm $ \| \Theta^\eps \|_{C^{m}([0, T])}+ \| U^\varepsilon\|_{X^m_{T} }$
does not blow up.
}
\end{rmk}

Based on the structure of the system  \rfb{forcc}--\rfb{solideps},   the component  $\Upsilon^\eps$ is  completely determined  by  $\Theta^\eps$ from standard ODE theory
with  $\Jscr_{1}$ invertible.  Moreover,  we have for every $T$, $\alpha$ and $m$  that 
\begin{equation*}
\label{UepsThetaeps}
 \sup_{x} \| \partial^\alpha_{x} \Upsilon^\eps \|_{C^{m+1}[0, T]} \leq C[T,  \|\Theta^\eps\|_{C^m[0, T]}].
 \end{equation*}
This is the reason why we did not emphasize the regularity of $\Upsilon^\eps$ in Theorem \ref{existeps}.
In the remaining part of the section, we shall use the following  simplified notation:
\begin{equation}
\label{defmatricesfix}
 \begin{aligned}
&  \mathcal{A}^j[\Theta, V](t,x)= \mathcal{A}^j(\Theta(t), \Upsilon(t,x), V(t,x)),  \\
&   A^0[\Theta, V](t,x)=A^0(\Theta(t), \Upsilon(t,x), V(t,x)),  \\
&  B[\Theta, V](t,x)= B(\Theta(t), \Upsilon(t,x), V(t,x)).
 \end{aligned}
 \end{equation}
 
 Now we focus on the proof of Theorem \ref{existeps}.
  
   \subsection{Existence and regularity for a  given solid motion}

We now consider the nonlinear system: 
 \begin{equation}\label{modelnonlinear}
\left\{\begin{aligned}
	&A^0[ \tilde{\Theta}, U] \partial_{t} U+\sum_{j=1}^3 \Ascr^j[ \tilde{\Theta},U] \partial_j U +B[ \tilde{\Theta},
	U] U= F(t,x), \\
	& U(0,x)=U_0^n (x),\\
	&G(x) U=\overline u_\mathcal{S}[ \tilde \Theta]\cdot n_0 \qquad \forall \m  x\in \partial\Fscr(0),
\end{aligned}\right.
\end{equation} 
where $\tilde{\Theta}$ is  assumed to be given by $\begin{bmatrix} L & W \end{bmatrix}^\intercal$ so  that $\overline u_\Sscr[\tilde \Theta]=L+W\times x $ at $\partial\Fscr(0)$.

\begin{prop}\label{schochet}
Let $m\geq 3$. For every $\eps\in (0, 1)$,  $U_{0}^n \in H^{m+1}$,  $F \in H^{m+1}((0, T_{L}) \times \mathcal{F}(0))$,
as in Theorem \ref{existeps} and 
   $R>0$.  There exists
 $T_{0}>0$ (depending on $\eps$) such that 
   for every  $\tilde{\Theta}\in C^{m+1}([0, T_0])$ satisfying
\begin{equation*}
\label{compatintermediate} \partial_{t}^k \tilde{\Theta}(0)= \mathcal{I}^{k, \varepsilon}_{\Theta}(U_{0}^n, \Theta_{0}, F^k)\qquad \forall \m k \leq m, 
\end{equation*}
and 
\begin{equation*}
\label{bornethetatilde}
\| \tilde{\Theta}\|_{C^{m+1}[0, T_{0}]} \leq R, 
\end{equation*}
There exists  a  unique solution $U_{\Theta}^\eps \in X^{m+1}_{T_{0}}$  of \eqref{modelnonlinear}
such that $ U_{\Theta}^\eps \in \mathcal{V}$.
Moreover, for every $T \in [0, T_{0}],$  we have the estimates: 
\begin{equation}
\label{bornuschochetnonunif}
\begin{aligned}
& \|U_\Theta^\eps\|_{L^\infty_{T} H^{m+1}_{co}}^2 + \|U_\Theta^\eps\|_{H^{m+1}((0, T)\times \partial\Fscr(0))}^2\\
&\leq C_\eps\left[\|U_\Theta^\eps\|_{X^m_T}, \|\tilde\Theta\|_{C^m[0, T]}\right]\left(\|U_0^n\|_{H^{m+1}(\Fscr(0))}^2+\|F\|_{H^{m+1}((0, T)\times \Fscr(0))}^2\right.\\
&\hspace{6cm} \left.+ \int_0^T(1+\|\tilde\Theta\|_{C^{m+1}[0, t]}^2+\| U_\Theta^\eps\|_{X^{m+1}_t}^2)\dd t\right),
  \end{aligned}
  \end{equation}
  and
 \begin{equation}
\label{bornuschochetnonunif2}
 \|U_\Theta^\eps\|_{X^{m+1}_{T}}  \leq C_{\varepsilon}\left[  \| \tilde \Theta\|_{C^m[0, T]}, \| U_{\Theta}^\varepsilon\|_{X^m_{T} }, \|U_0^n\|_{H^{m+1}(\Fscr(0))} \right] \Bigl( 
 \|U_{\Theta}^\varepsilon\|_{X^m_{T}} + \|U_{\Theta}^\varepsilon\|_{L^\infty_{T}H^{m+1}_{co}} + \|F\|_{X^m_{T}} \Bigr). 
  \end{equation}
  In particular, there exists $M_{\varepsilon}= M_{\eps}[R,   \|U_{0}^n\|_{H^{m+1}},  \|F\|_{H^{m+1}((0, T)\times \mathcal{F}(0))}]$ so that
  \begin{equation}
  \label{latroisieme} \|U_\Theta^\eps\|_{X^{m+1}_{T}}  + \|U_\Theta^\eps\|_{H^{m+1}((0, T)\times \partial\Fscr(0))} \leq M_{\varepsilon}.
  \end{equation}
\end{prop}

\begin{proof}

We split the proof in several steps.
 
{\bf Step 1}. {\em We solve the system \rfb{modelnonlinear} with a given smoother solid motion}.
We replace $\tilde \Theta$ by $\hat{\Theta}$   in the boundary data $\overline u_\Sscr[\tilde\Theta]$ in \rfb{modelnonlinear}, 
where 
 $\hat{\Theta} \in C^{m+2}(\mathbb{R})$,  the compatibility conditions are still satisfied and
 $$ \|\hat{\Theta}\|_{C^{m+1}([0, T])} \leq 2  \|\tilde{\Theta}\|_{C^{m+1}([0, T])}.$$
 Indeed,  we can always find a family $\tilde{\Theta}_{\mu}  \in C^{m+2}(\mathbb{R})$  such that   $\tilde{\Theta}_{\mu}\to \tilde{\Theta}$ in $C^{m+1}[0, T]$ when $\mu$ tends to zero
 such that the compatibility conditions are still satisfied 
 that we  can use  in the final step to pass to the limit and recover the result for $\tilde \Theta$.
  To construct such a family, starting from  $\tilde\Theta\in C^{m+1}[0, T]$,  we first extend it to a compactly supported function on  $\mathbb{R}$ in  
 the following way. Let $\tilde{\chi}$ a smooth compactly supported function which is $1$ on $[-T_{L}, T_{L}]$
  for a large fixed $T_{L}>0$, we set 
 \begin{equation*}
  \Theta_{ext}(t) =: \tilde\chi(t) \cdot\left\{
 \begin{aligned}
 & \tilde{\Theta} (t) &\qquad \forall \m t\in[0, T],\\
 &  P_{0}(t)& \qquad \forall \m t \leq 0, \\
 & P_{T}(t)& \qquad \forall t \geq T, 
 \end{aligned}\right.
 \end{equation*}
 where 
 and $P_{0}$  and $ P_{T}$ are  the Taylor polynomials  given by
 $$P_{0}(t)=\sum_{k=1}^{m+1}\frac{\partial_t^k\tilde\Theta(0)}{k!}t^k, \quad
 P_{T}(t)=\sum_{k=1}^{m+1}\frac{\partial_t^k\tilde\Theta(T)}{k!}(t-T)^k. $$
 From this definition, we have that  $\Theta_{ext}$  satisfies $\Theta_{ext}\in C^{m+1}(\rline)$, $\Theta_{ext}=\tilde\Theta$ on $[0, T]$ and in particular it keeps the compatibility condition since we have  $\partial_t^k\check\Theta(0)=\partial_t^k\tilde \Theta(0)$ for $0\leq k\leq m+1$. To improve the regularity of $\Theta_{ext}$, we use the standard mollifier $\eta_\mu\in C_c^\infty(\rline)$ and introduce
$$\hat\Theta_\mu(t):=(\Theta_{ext}*\eta_\mu)(t)=\int_\rline\Theta_{ext}(s)\eta_\mu(t-s)\dd s. $$
At this stage, we obtain $\hat\Theta_\mu\in C^\infty(\rline)$ and $\hat\Theta_\mu$ restricted to  $[0, T]$ converges to $ \tilde\Theta$ in $C^{m+1}[0, T]$ as $\mu\to 0$. Due to the mollification, now the compatibility conditions are no longer verified. To compensate this, it suffices to change the behaviour of $\hat\Theta_\mu$ at  $0$ by adding a correction term.  We thereby define $\tilde \Theta_{\mu}$ as 
\begin{equation*}\label{hattheta}
\tilde\Theta_{\mu}(t):=\hat\Theta_\mu(t)+\tilde\chi(t)Q_\mu(t),
\end{equation*}
where  the correction polynomial $Q_\mu$ is :
$$ Q_\mu(t):=\sum_{k=0}^{m+1}\frac{(\partial_t^k\tilde\Theta(0)-\partial_t^k\hat\Theta_\mu(0))}{k!}t^k.$$
Note that $Q_\mu\to 0$ uniformly on $[0, T]$. The restriction of  $\tilde{\Theta}_{\mu}$ to $[0, T]$  thus gives  the desired function.

 Now
 we solve \eqref{modelnonlinear} with  $\tilde \Theta$ replaced by $\hat \Theta$ in the boundary condition.
 We introduce $\hat{U}= U- U^c$ with $U^c:=\begin{bmatrix}0, \overline u_{\mathcal{S}}[\hat{\Theta}], 0\end{bmatrix}^\intercal$ and we thereby obtain the system described by $\hat{U}$ as follows:
 \begin{equation}\label{shift}
 \left\{\begin{aligned}
	&A^0[ \tilde{\Theta}, U^c + \hat U] \partial_{t} \hat U+\sum_{j=1}^3 \Ascr^j[ \tilde{\Theta},U^c + \hat U ] \partial_j  \hat U +\Bscr[ \tilde{\Theta},
	U^c,  \hat U] \hat  U= F^c(t,x), \\
	& U=U_0^n (x),\\
	&G(x) U= 0 \qquad \forall\m x\in \partial\Fscr(0),
\end{aligned}\right.
\end{equation} 
where 
$$F^c  =-A^0[\tilde\Theta, U^c]\partial_t U^c- \sum_{j=1}^d \mathcal{A}^j[\tilde{\Theta}, U^c] \partial_{j}U^c - B[\tilde{\Theta}, U^c] U^c + F,  $$
\begin{align*} 
	\Bscr[ \tilde{\Theta},
	U^c,  \hat U] \hat  U  &= B[\tilde\Theta, U^c+\hat U]\hat U+\left(\int_0^1D_U A^0[\tilde\Theta, U^c+s\hat U]\cdot \hat U\dd s\right)\partial_t U^c\\
	&\quad +\sum_{j=1}^d \left( \int_{0}^1 D_{U}\mathcal{A}^j[\tilde\Theta, U^c + s \hat U] \cdot \hat U\, \dd s\right)\partial_{j} U^c+ \left( \int_{0}^1 D_{U}B[\tilde\Theta, U^c + s \hat U] \cdot \hat U\, \dd s\right)U^c.
\end{align*}
In particular, we observe that $F^c \in H^{m+1}([0, T_{0}]\times \mathcal{F}(0)) $ since $\hat\Theta\in C^{m+2}$. The  loss of one derivative is  due to the term in $F^c$ involving
$ \partial_{t} U^c$.

Now that we have homogeneous  boundary conditions, we obtain by using \cite[Theorem A2]{schochet1986compressible} that there exists $T>0$ such that the system \rfb{shift} admits a unique solution $\hat U\in X_{T}^{m+1}$. Moreover, from \cite[Theorem A6]{schochet1986compressible}, we know that $\hat U$ can be continued in $X^{m+1}_T$ as long as $ \hat U + U^c \in \mathcal{U}$ and its $X^m_{T}$-norm remains finite.
 By setting back $U= U^c + \hat U$, this yields a solution $U  \in X_{T}^{m+1}$
 for \eqref{modelnonlinear}
 with $\tilde\Theta$ replaced by $\hat \Theta$ in the boundary condition and in the coefficients of \rfb{modelnonlinear} and $U$ can be continued as long as $U \in \mathcal{U}$ and $U$ does not blow up in $X^m_{T}$.

{\bf Step 2}. {\em We prove that the solution obtained in Step 1, denoted by $U_\Theta^\eps$, enjoys the estimates \eqref{bornuschochetnonunif} and \eqref{bornuschochetnonunif2}}.  
The following energy estimates can be justified in a classical way by  taking  a smoother approximation of  
 $\tilde \Theta$  and  by approximating the initial   data and the source term  by smoother ones  in the same way as in the beginning of section \ref{approx}
  in order to maintain the compatibility conditions.
 We then get that the  solution is actually  in $X^{m+2}_{T}$ and  can pass to the limit in the final estimates.  Note that,  indeed,  by the continuation criterion mentioned in  Step 1, the solution in $X^{m+2}_{T}$ can be continued as long
 as the $X^m_{T}$ norm is bounded so that  it indeed exists on the same interval of time.


We use the conormal vector fields introduced in Subsection \ref{notation}. Applying $Z^\beta$ with $\beta=(\beta_0, \beta_1, \beta_2, \beta_3)\in \nline^4$ and $|\beta|\leq m+1$ to the system \rfb{modelnonlinear}, we have 
\begin{align}
&A^0[\tilde \Theta, U_\Theta^\eps]\partial_t(Z^\beta U_\Theta^\eps)+\sum_{j=1}^3\Ascr^j[\tilde \Theta, U_{\Theta}^\eps]\partial_j(Z^\beta U_\Theta^\eps)\nonumber\\
&=\left[A^0[\tilde \Theta, U_\Theta^\eps]\partial_t,\m Z^\beta\right]U_\Theta^\eps+\left[\Ascr[\tilde \Theta, U_\Theta^\eps] \cdot \nabla,\m Z^\beta\right]U_\Theta^\eps +Z^\beta\left(-B([\tilde \Theta, U_\Theta^\eps]U_\Theta^\eps+F( t,x)\right),\label{2need}
\end{align}
where we use the bracket notation for the commutator between operators:  $[a, b]=ab-ba$. Take the inner product of \rfb{2need} and $Z^\beta U_\Theta^\eps$ in $\Fscr(0)$. 
By integration by parts and  by using the symmetry of the matrices, we obtain
\begin{align*}
& \int_{\Fscr(0)}A^0[\tilde \Theta, U_\Theta^\eps]\partial_t(Z^\beta U_\Theta^\eps)\cdot Z^\beta U_\Theta^\eps\dd x\\
 &=\frac{1}{2}\frac{\dd}{\dd t}\int_{\Fscr(0)}A^0[\tilde \Theta, U_\Theta^\eps]Z^\beta U_\Theta^\eps\cdot Z^\beta U_\Theta^\eps\dd x-\frac{1}{2}\int_{\Fscr(0)}\partial_tA^0[\tilde \Theta, U_\Theta^\eps]Z^\beta U_\Theta^\eps\cdot Z^\beta U_\Theta^\eps\dd x.
 \end{align*}
\begin{align*}
&\int_{\Fscr(0)}\sum_{j=1}^3\Ascr^j[\tilde \Theta, U_\Theta^\eps]\partial_j(Z^\beta U_\Theta^\eps)\cdot Z^\beta U_\Theta^\eps\dd x\\
&=\frac{1}{2}\int_{\partial\Fscr(0)}\Ascr_n[\tilde \Theta, U_\Theta^\eps]Z^\beta U_\Theta^\eps\cdot Z^\beta U_\Theta^\eps\dd x-\frac{1}{2}\int_{\Fscr(0)}\div \Ascr[\tilde \Theta, U_\Theta^\eps]Z^\beta U_\Theta^\eps\cdot Z^\beta U_\Theta^\eps\dd x.
\end{align*}
Similarly, for the boundary integral we have 
\begin{equation}\label{Zestimate}
\frac{1}{2}\int_{\partial\Fscr(0))}\Ascr_n[\tilde \Theta, U_\Theta^\eps]Z^\beta U_\Theta^\eps\cdot Z^\beta U_\Theta^\eps\dd x=\frac{\eps}{2}\int_{\partial\Fscr(0)}|Z^\beta U_\Theta^\eps|^2\dd \Gamma+\int_{\partial\Fscr(0)}Z^\beta \overline p_\Theta^\eps (Z^\beta \overline u_\Theta^\eps)\cdot n_0\dd \Gamma.
\end{equation}
We continue the estimate of the second boundary integral on the right side of \rfb{Zestimate}:
\begin{align*}
\int_{\partial\Fscr(0)}Z^\beta \overline p_\Theta^\eps (Z^\beta \overline u_\Theta^\eps)\cdot n_0\dd \Gamma&\lesssim \frac{\eps}{4}\int_{\partial\Fscr(0)}|Z^\beta\overline p_\Theta^\eps|^2\dd\Gamma+c(\eps)\int_{\partial\Fscr(0)}|(Z^\beta\overline u_\Theta^\eps)\cdot n_0|^2\dd\Gamma.
\end{align*}
Observe that
$$(Z^\beta\overline u_\Theta^\eps)\cdot n_0=Z^\beta(\overline u_\Theta^\eps\cdot n_0)-\sum_{|\gamma|\neq 0, |\sigma+\gamma|\leq m+1 } Z^\sigma\overline u_\Theta^\eps\cdot Z^\gamma n_0, $$
which together with the boundary condition in \rfb{modelnonlinear} gives further the following estimate:
\begin{align*}
\int_{\partial\Fscr(0)}|(Z^\beta\overline u_\Theta^\eps)\cdot n_0|^2\dd\Gamma
&\lesssim \int_{\partial\Fscr(0)}|Z^\beta(\overline u_\Theta^\eps\cdot n_0)|^2\dd\Gamma+\int_{\partial\Fscr(0)}|\sum_{|\gamma|\neq 0, |\sigma+\gamma|\leq m+1} Z^\sigma\overline u_\Theta^\eps\cdot Z^\gamma n_0|^2\dd\Gamma\\
&\lesssim \int_{\partial\Sscr(0)}|Z^\beta((L+W\times x)\cdot n_0)|^2\dd\Gamma +\sum_{|\gamma|\neq 0, |\sigma+\gamma|\leq m+1} \int_{\partial\Fscr(0)}| Z^\sigma\overline u_\Theta^\eps\cdot Z^\gamma n_0|^2\dd\Gamma\\
&\lesssim \int_{\partial\Sscr(0)}|\partial_t^{\beta_0}L\cdot Z^{|\beta|-\beta_0}n_0 +\partial_t^{\beta_0}W\times Z^{|\beta|-\beta_0}(x\times n_0)|^2\dd\Gamma\\
&\qquad \qquad +\sum_{|\gamma|\neq 0, |\sigma+\gamma|\leq m+1}\|Z^\sigma\overline u_\Theta^\eps\|_{L^2(\partial\Fscr(0))}^2\|Z^\gamma n_0\|_{L^\infty(\partial\Fscr(0))}^2\\
&\lesssim \|\tilde\Theta\|_{C^{m+1}[0, t]}^2+\|U_\Theta^\eps\|_{H^{m}_{co}(\partial\Fscr(0))}^2\lesssim \|\tilde\Theta\|_{C^{m+1}[0, t]}^2+\|U_\Theta^\eps\|_{H^{m+1}_{co}(\Fscr(0))}^2.
\end{align*}

Note that we further have the estimate:
\begin{align*}
&\int_{\Fscr(0)}\partial_tA^0[\tilde \Theta, U_\Theta^\eps]Z^\beta U_\Theta^\eps\cdot Z^\beta U_\Theta^\eps\dd x+\int_{\Fscr(0)}\div \Ascr[\tilde \Theta, U_\Theta^\eps]Z^\beta U_\Theta^\eps\cdot Z^\beta U_\Theta^\eps\dd x\\
&\lesssim \left(\|\partial_tA^0[\tilde \Theta, U_\Theta^\eps]\|_{L^\infty(\Fscr(0))}+\|\div \Ascr[\tilde \Theta, U_\Theta^\eps]\|_{L^\infty(\Fscr(0))}\right)\|Z^\beta U_\Theta^\eps\|_{L^2(\Fscr(0))}^2\\
&\lesssim C_\eps\left[\|U_\Theta^\eps\|_{X^m_T}, \|\tilde\Theta\|_{C^m[0,T]}\right]\|Z^\beta U_\Theta^\eps\|_{L^2(\Fscr(0))}^2,
\end{align*}
where we used the estimate \rfb{smoothmatrix} in Proposition \ref{productestimate} for $A^0$ and $\Ascr$.

Combining the above estimates, for every $t\in [0, T]$, we arrive at
\begin{align}
&\frac{1}{2}\frac{\dd}{\dd t}\int_{\Fscr(0)}A^0[\tilde \Theta, U_\Theta^\eps]Z^\beta U_\Theta^\eps\cdot Z^\beta U_\Theta^\eps\dd x+\frac{\eps}{2}\int_{\partial\Fscr(0)}|Z^\beta U_\Theta^\eps|^2\dd \Gamma\nonumber\\
&\lesssim \int_{\Fscr(0)}[A^0[\tilde \Theta, U_\Theta^\eps]\partial_t,\m Z^\beta]U_\Theta^\eps\cdot Z^\beta U_\Theta^\eps\dd x+\int_{\Fscr(0)}[\Ascr[\tilde \Theta, U_\Theta^\eps]\cdot \nabla,\m Z^\beta]U_\Theta^\eps\cdot Z^\beta U_\Theta^\eps\dd x\nonumber\\
&\qquad +\int_{\Fscr(0)}Z^\beta\left(-B[\tilde \Theta, U_\Theta^\eps]U_\Theta^\eps+F \right)\cdot Z^\beta U_\Theta^\eps \dd x +\frac{\eps}{4}\int_{\partial\Fscr(0)}|Z^\beta\overline p_\Theta^\eps|^2\dd\Gamma\label{secondstage}\\
&\qquad +C_\eps\left[\|U_\Theta^\eps\|_{X^m_T}, \|\tilde\Theta\|_{C^m[0,T]}\right]\|Z^\beta U_\Theta^\eps\|_{L^2(\Fscr(0))}^2+\|\tilde\Theta\|_{C^{m+1}[0, t]}^2+\|U_\Theta^\eps\|_{H^{m+1}_{co}(\Fscr(0))}^2.\nonumber
\end{align}
Now we focus on estimating the integrals with the commutators in \rfb{secondstage}. Taking integration from $0$ to $T$ and using the estimates \rfb{comm},  we have
\begin{equation*}\label{commum1}
\begin{aligned}
&\int_0^T\int_{\Fscr(0)}[A^0[\tilde \Theta, U_\Theta^\eps]\partial_t,\m Z^\beta]U_\Theta^\eps\cdot Z^\beta U_\Theta^\eps\dd x\dd t\\
&\leq \int_0^T\|[A^0[\tilde \Theta, U_\Theta^\eps]\partial_t,\m Z^\beta]U_\Theta^\eps\|_{L^2(\Fscr(0))}\|Z^\beta U_\Theta^\eps\|_{L^2(U_\Theta^\eps)}\dd t\\
&\leq C_\eps\left[\|U_\Theta^\eps\|_{X_T^m}, \|\tilde\Theta\|_{C^m[0, T]}\right]\int_0^T(1+\|U_\Theta^\eps \|_{X^{m}_t}+\|\tilde\Theta\|_{C^{m}[0,t]})\|U_\Theta^\eps\|_{X^{m+1}_t}^2\dd t\\
&\quad +C_\eps\left[\|U_\Theta^\eps\|_{X^m_T}, \|\tilde \Theta\|_{C^m[0, T]}\right]\int_0^T(1+\|\tilde \Theta\|_{C^{m+1}[0, t]}+\|U_\Theta^\eps\|_{X^{m+1}_t})\|U_\Theta^\eps\|_{X^m_t}\|U_\Theta^\eps\|_{X^{m+1}_t}\dd t,
\end{aligned}
\end{equation*}
as well as
\begin{equation*}\label{commum2}
\begin{aligned}
&\int_0^T\int_{\Fscr(0)}[\Ascr[\tilde \Theta, U_\Theta^\eps]\cdot \nabla,\m Z^\beta]U_\Theta^\eps\cdot Z^\beta U_\Theta^\eps\dd x\dd t\\
&\leq\int_0^T \|[\Ascr[\tilde \Theta, U_\Theta^\eps]\cdot \nabla,\m Z^\beta]U_\Theta^\eps\|_{L^2(\Fscr(0))}\|Z^\beta U_\Theta^\eps\|_{L^2(\Fscr(0))}\dd t\\
&\leq C_\eps\left[\|U_\Theta^\eps\|_{X_T^m}, \|\tilde\Theta\|_{C^m[0, T]}\right]\int_0^T(1+\|U_\Theta^\eps \|_{X^{m}_t}+\|\tilde\Theta\|_{C^{m}[0,t]})\|U_\Theta^\eps\|_{X^{m+1}_t}^2\dd t\\
&\quad + C_\eps\left[\|U_\Theta^\eps\|_{X_T^m}, \|\tilde\Theta\|_{C^m[0, T]}\right]\int_0^T(1+\|\tilde \Theta\|_{C^{m+1}[0, t]}+\|U_\Theta^\eps\|_{X^{m+1}_t})\|U_\Theta^\eps\|_{X^m_t}\|U_\Theta^\eps\|_{X^{m+1}_t}\dd t.
\end{aligned}
\end{equation*} 

Using \rfb{proXm} and \rfb{smoothmatrix} we continue dealing with the integrals depending on $B$ and $F$:
\begin{align*}
	&\int_0^T\int_{\Fscr(0)}Z^\beta\left(-B[\tilde \Theta, U_\Theta^\eps]U_\Theta^\eps+F \right)\cdot Z^\beta U_\Theta^\eps \dd x\dd t\\
	&\lesssim \int_0^T\left(\|B[\tilde\Theta, U_\Theta^\eps]\|_{X^{m}_t}\|U_\Theta^\eps\|_{X^{m+1}_t}+\|B[\tilde\Theta, U_\Theta^\eps]\|_{X^{m+1}_t}\|U_\Theta^\eps\|_{X^{m}_t}+\|Z^\beta F\|_{L^2(\Fscr(0))}\right)\|Z^\beta U_\Theta^\eps\|_{L^2(\Fscr(0))}\dd t\\
	&\leq C_\eps\left[\|U_\Theta^\eps\|_{X^m_T}, \|\tilde\Theta\|_{C^{m}[0, T]}\right] \int_0^T(1+\|\tilde\Theta\|_{C^{m+1}[0, t]}+\| U_\Theta^\eps\|_{X^{m+1}_t})\|U_\Theta^\eps\|_{X^{m+1}_t}\dd t\\
	&\quad \quad +\|F\|_{H^{m+1}((0, T)\times \Fscr(0))}^2+\int_0^T\|U_\Theta^\eps\|_{X^{m+1}_t}^2\dd t.
\end{align*}

Summing over $0\leq |\beta|\leq m+1$, by the ellipticity of $A^0[\tilde \Theta, U_\Theta^\eps]$ in $\mathcal{V}$, we conclude from \rfb{secondstage} that
\begin{equation*}\label{firstp}
\begin{aligned}
&\sup_{t\in [0, T]}\|U_\Theta^\eps(t)\|_{H^{m+1}_{co}(\Fscr(0))}^2+\int_0^T\|U_\Theta^\eps\|_{H^{m+1}_{co}(\partial\Fscr(0))}^2\dd t\\
&\leq C_\eps\left[\|U_\Theta^\eps\|_{X^m_T}, \|\tilde\Theta\|_{C^m[0, T]}\right]\left(\|U_0^n\|_{H^{m+1}(\Fscr(0))}^2+\|F\|_{H^{m+1}((0, T)\times \Fscr(0))}^2\right.\\
&\quad \quad \quad \quad \left.+ \int_0^T(1+\|\tilde\Theta\|_{C^{m+1}[0, t]}^2+\| U_\Theta^\eps\|_{X^{m+1}_t}^2)\dd t\right).
\end{aligned}
\end{equation*}
According to the definition of the conormal vector fields, we see that  the above boundary term contains all the tangential and time
derivatives. Therefore we actually have
$$\|U_\Theta^\eps\|_{L^2(0, T; H^{m+1}_{co}(\partial\Fscr(0)))}=\|U_\Theta^\eps\|_{H^{m+1}((0, T)\times \partial\Fscr(0))}.$$
Hence we obtain the estimate \rfb{bornuschochetnonunif}. 

We shall now prove the estimate \eqref{bornuschochetnonunif2}. Thanks to the definition
of the conormal space, 
only  the control of the normal derivatives close to the boundary $\partial\Fscr(0)$ is missing according to Remark \ref{understand}. Actually,  we just need to control them near the boundary $\partial\Fscr(0)$.
To do this,
we shall work in a small vicinity of the boundary in a local chart $\mathcal{O}_{i}$ where the boundary is given by the equation $x_{3}= \varphi_i(x_{1}, x_{2})$.
 The conormal norm  controls the tangential fields $Z_{k}= \partial_{y_{k}}=  \partial_{x_{k}} + \partial_{x_{k}} \varphi_i\m \partial_z $ for $k=1,2$, while the missing direction to be controlled  is  $\partial_{z}=\partial_{x_{3}}$.
  Since \rfb{modelnonlinear} is with non-characteristic boundary, in a vicinity of the boundary we express it by using the local coordinates $(y_1, y_2, z)$ as
  \begin{equation}\label{normalutheta}
  \partial_{z} U_{\Theta}^\eps= - (A_{z}[\tilde\Theta, U_\Theta^\eps])^{-1} \left(A^0[\tilde\Theta, U_\Theta^\eps] \partial_{t}U_\Theta^\eps + \sum_{j=1}^2 \mathcal{A}^j[\tilde\Theta, U_\Theta^\eps] \partial_{y_j} U_{\Theta}^\eps + B[\tilde\Theta, U_\Theta^\eps]U_{\Theta}^\eps - F \right),
  \end{equation}
  where 
  $ A_{z}= \mathcal{A}^{3} -  \sum_{j=1}^2 \partial_{y_j} \varphi_i\mathcal{A}^{j}$
  is invertible close to the boundary.
 We shall  use the notation \begin{equation}\label{newgradient}
 \nabla_y:=(Z_1, Z_2)=(\partial_{y_1}, \partial_{y_2}).
 \end{equation}
 Note that  by rearranging the derivatives, we can first write
\begin{equation*}
\|U_\Theta^\eps\|_{X^{m+1}_T(\mathcal{O}_i)} \leq \sum_{i=0}^m\|\partial_t^i\partial_z^{m+1-i} U_\Theta^\eps\|_{X^0_{T}(\mathcal{O}_i)}
 + \|\nabla_{y}U_\Theta^\eps\|_{X^m_{T}(\mathcal{O}_{i})}
+ \|U_\Theta^\eps\|_{X^m_{T}} +  \|U_\Theta^\eps\|_{L^\infty_{T} H^{m+1}_{co}}.
\end{equation*}
By using \eqref{calculus1} in  Proposition \ref{calculus} (for $\mu= 1$), we can further simplify into
\begin{equation}\label{notice}
\|U_\Theta^\eps\|_{X^{m+1}_T(\mathcal{O}_i)} \lesssim  \sum_{i=0}^m\|\partial_t^i\partial_z^{m+1-i} U_\Theta^\eps\|_{X^0_{T}(\mathcal{O}_i)}
+ \|U_\Theta^\eps\|_{X^m_{T}} +  \|U_\Theta^\eps\|_{L^\infty_{T} H^{m+1}_{co}}.
\end{equation}
We then apply $\partial_t^i\partial_z^{m-i}$ with $0\leq i\leq m$ to \rfb{normalutheta} and obtain that
\begin{multline*}
\partial_t^i\partial_z^{m+1-i} U_\Theta^\eps=
 -(A_z^{-1}A^0)[\tilde\Theta, U_\Theta^\eps]\partial_t^{i+1}\partial_z^{m-i} U_\Theta^\eps-A_z^{-1}\sum_{j=1}^2\Ascr^j[\tilde\Theta, U_\Theta^\eps]\partial_t^i\partial_z^{m-i}\partial_{y_{j}} U_\Theta^\eps
  \\- \mathcal{C} - \partial_t^i\partial_z^{m-i}(A_z^{-1} ( B[\tilde\Theta, U_\Theta^\eps]U_{\Theta}^\eps - F)),
\end{multline*}
where $\mathcal{C}$ is a commutator.
 Thanks to the estimates in Proposition \ref{productestimate}
 (note that here we apply $m$ derivatives with $m \geq 3$ on the right hand side so that the commutator estimates are indeed valid), 
 we obtain that for $0\leq i\leq m$,
%
%
\begin{multline*}
\|\partial_t^i\partial_z^{m+1-i}U_\Theta^\eps\|_{X^0_{T}(\mathcal{O}_{i})}\leq C_\eps[\|U_\Theta^\eps\|_{X^m_T}, \|\tilde\Theta\|_{C^m[0, T] }]\left(\|\partial_t^{i+1}\partial_z^{m-i} U_\Theta^\eps\|_{X^0_{T}(\mathcal{O}_{i})}+\|\partial_t^i\partial_z^{m-i}\nabla_y U_\Theta^\eps\|_{X^0_{T}(\mathcal{O}_{i})}\right.\\
 \left. +\|F\|_{X^m_T}+ \|U_\Theta^\eps\|_{L^\infty_TH^{m+1}_{co}}+\|U_\Theta^\eps\|_{X^{m}_T}\right).
\end{multline*}
By using  \eqref{calculus1} in  Proposition \ref{calculus} and by choosing $\mu$ sufficiently 
small this yields for $ 0 \leq i  \leq m$, 
\begin{multline*}
\|\partial_t^i\partial_z^{m+1-i}U_\Theta^\eps\|_{X^0_{T}(\mathcal{O}_{i})}\leq C_\eps[\|U_\Theta^\eps\|_{X^m_T}, \|\tilde\Theta\|_{C^m[0, T] }]\left(\|\partial_t^{i+1}\partial_z^{m-i} U_\Theta^\eps\|_{X^0_{T}(\mathcal{O}_{i})}\right.\\
 \left. +\|F\|_{X^m_T}+ \|U_\Theta^\eps\|_{L^\infty_TH^{m+1}_{co}}+\|U_\Theta^\eps\|_{X^{m}_T}\right).
\end{multline*}
We see that there is now  $\|\partial_t^{i+1}\partial_z^{m-i}U_\Theta^\eps\|_{X^0_{T}}$ on the right hand side above. If $i \leq m-1$, we can  thus  iterate
by using the estimate for $\|\partial_t^{i+1}\partial_z^{m+1-(i+1)}U_\Theta^\eps\|_{X^0_{T}}$
and iterate the process to  finally arrive at 
\begin{equation*}
\|\partial_t^i\partial_z^{m+1-i}U_\Theta^\eps\|_{X^0_{T}(\mathcal{O}_{i})}\leq C_\eps[\|U_\Theta^\eps\|_{X^m_T}, \|\tilde\Theta\|_{C^m[0, T] }]\left( \|F\|_{X^m_T}+ \|U_\Theta^\eps\|_{L^\infty_TH^{m+1}_{co}}+\|U_\Theta^\eps\|_{X^{m}_T})\right).
\end{equation*}
Plugging this estimate into \eqref{notice}, we thus get 
\begin{equation*}
\|U_\Theta^\eps\|_{X^{m+1}_T(\mathcal{O}_i)} \leq  C_\eps[\|U_\Theta^\eps\|_{X^m_T}, \|\tilde\Theta\|_{C^m[0, T] }]\left( \|F\|_{X^m_T}+ \|U_\Theta^\eps\|_{L^\infty_TH^{m+1}_{co}}+\|U_\Theta^\eps\|_{X^{m}_T}
\right).
\end{equation*}
Since a neighborhood of the boundary is covered  by the  $\mathcal{O}_i$, we 
finally obtain  \rfb{bornuschochetnonunif2}.

{\bf Step 3}. It remains to prove that we can find $T_{0}$ which depends only
 on $R$ and the data such that the solution can be continued on $[0, T_{0}]$
  and that \eqref{latroisieme} holds.
 
 Let us assume that \eqref{latroisieme} holds for  some $T$  small with $M_{\eps}$
  to be chosen. Then, 
  thanks to  the estimates \eqref{bornuschochetnonunif} and \eqref{bornuschochetnonunif2}, we deduce that
\begin{equation*}
 \|U_\Theta^\eps\|_{X^{m+1}_{T}}  +  \|U_\Theta^\eps\|_{H^{m+1}([0, T] \times \partial \mathcal{F}(0))} 
  \leq  C_{\varepsilon}\left[\| U_{\Theta}^\varepsilon\|_{X^m_{T} }, R\right]
  \Bigl( \|U_0^n\|_{H^{m+1}}+\|F\|_{H^{m+1}((0, T)\times \Fscr(0))} +T^{\frac{1}{2}}(1+R+ M_{\eps})\Bigr),
\end{equation*}
where we used $\|\tilde{\Theta} \|_{C^{m+1}([0, T])} \leq R$.
 Since we deduce from a rough time integration of  \eqref{modelnonlinear} that 
 $$  \| U_{\Theta}^\varepsilon\|_{X^m_{T} } \leq   \| U^n_{0}\|_{H^{m+1}} + T  C_{\varepsilon}\left[\| U_{\Theta}^\varepsilon\|_{X^m_{T} }, R\right]    \| U_{\Theta}^\varepsilon\|_{X^{m+1}_{T} } +  T\|F \|_{X^m_{T}}
  \leq \| U^n_{0}\|_{H^{m+1}} +  T\| F \|_{X^m_{T}} +  T  C_{\varepsilon}\left[M_{\eps}, R\right],  $$
  we have
\begin{equation}
\label{presqueOK}
\begin{aligned}
   \|U_\Theta^\eps\|_{X^{m+1}_{T}}  +  \|U_\Theta^\eps\|_{H^{m+1}([0, T] \times \partial \mathcal{F}(0))} 
&  \leq C_{\eps}\Bigl[ \| U^n_{0}\|_{H^{m+1}}+ T \|F\|_{X^m_{T}} +  T  C_{\varepsilon}\left[M_{\eps}, R\right], R  \Bigr] 
  \Bigl( \|U_0^n\|_{H^{m+1}}  \\
  &\quad \quad \quad +\|F\|_{H^{m+1}((0, T)\times \Fscr(0))} +T^{\frac{1}{2}}(1+R+ M_{\eps})\Bigr).
  \end{aligned}
  \end{equation}
Then for every $R$, we can choose
  $M_\eps$ sufficiently large so that
  $$C_{\eps}\Bigl[ \| U^n_{0}\|_{H^{m+1}}, R\Bigr]( \| U^n_{0}\|_{H^{m+1}}+  \|F\|_{H^{m+1}((0, T_{L})\times \Fscr(0))}) \leq {M_{\eps} \over 2},$$
   and  then  $T_{0}$ sufficiently small so that the right hand side of \eqref{presqueOK} for $T=T_{0}$ 
   is  still smaller than   $M_{\eps}$  while ensuring  that $ \pi U_{\Theta}^\eps \in \mathcal{V}$ on $[0, T_{0}]$.
   Note that by Sobolev embedding, we have that
   $$ \|F\|_{X^m_{T}} \lesssim_{T_{L}, \mathcal{F}(0)} \| F\|_{H^{m+1}((0, T_{L})\times \Fscr(0))}).$$
   Then, by a standard boostrap argument the solution can be continued
   on $[0, T_{0}]$ and \eqref{latroisieme} holds for $T \in [0,  T_{0}].$

To conclude, we can take for $\hat \Theta$ a sequence $\hat{\Theta}^n= \tilde{\Theta}_{\mu_{n}}$
with $\mu_{n}$ tending to zero where  $ \tilde{\Theta}_{\mu}$ was constructed in Step 1. 
By using the above estimates which are independent of $\mu_{n}$, we can use  classical arguments to  pass to the limit and get the solution
$U^\varepsilon_{\Theta}$ as stated in Proposition \ref{schochet}. We do not give full details here since the same arguments
will be used below  to pass to the limit when $\varepsilon$ tends to zero in order to get Theorem \ref{main}.
Note that the uniqueness follows directly from the $L^2$-estimate for the difference of two solutions.
 \end{proof}

\subsection{Proof of Theorem \ref{existeps}}
We are now in a position to give the proof of Theorem \ref{existeps}.
We shall use a fixed-point argument.
To set up the fixed-point procedure, we introduce the following space:
\begin{equation}\label{space}
\Kscr_{T}:=\begin{Bmatrix}\Theta= (L, W)\in C^{m+1}([0, T])\left|\m \partial_t^k \Theta(0)= \mathcal{I}_{\Theta}^{k, \eps}(U_{0}^n, \Theta_{0}, F^k),  k\leq m\right.
\end{Bmatrix}.
 \end{equation}
For $\Theta= (L, W) \in \Kscr_{T}$ we introduce the map $\Lambda(\Theta)$ as below:
\begin{equation}\label{map}
\Lambda(\Theta):=
\left\{\begin{aligned}
& l_0+\int_0^t\left(L\times W +\frac{1}{m}\int_{\partial\Sscr(0)} \overline p_\Theta^\eps n_0\dd \Gamma \right)\dd \sigma,\\
& \omega_0+\int_0^t\left(W\times (J_0^{-1}W)+J_0^{-1}\int_{\partial\Sscr(0)} x\times \overline p_\Theta^\eps \m n_0\dd \Gamma\right)\dd \sigma,
\end{aligned}
\right.
\end{equation}
where $U^\varepsilon_{\Theta}$ is the solution of \eqref{modelnonlinear} for a given $\Theta$ which is provided by Proposition \ref{schochet}.

Our main task in the following is to show that the map $\Lambda$ possesses a fixed point in the space $\Kscr_{T}$.

Let $B_{T, R}^{\Kscr}$ be the  ball defined as
$$B_{T, R}^\Kscr:=\left\{\Theta\in \Kscr_{T} ~\text{with}~\|\Theta\|_{C^{m+1}([0, T])}\leq R \right\}, $$
for some $R$ large enough and $T$ small enough. We shall always consider $T$ smaller than $T_{0}$ given by Proposition \ref{schochet}.

\begin{prop}\label{fixedpoint}
There exists a time $T>0$ and $R>0$, such that the map $\Lambda$ in \rfb{map} is well-defined on $B_{T, R}^\Kscr$, i.e.
\begin{equation*}
\begin{aligned}
\Lambda: B_{T, R}^\Kscr&\longrightarrow B_{T, R}^\Kscr\\
\Theta&\mapsto \Lambda(\Theta).
\end{aligned}
\end{equation*}
Moreover, $\Lambda$ has a fixed point in $ B_{T, R}^\Kscr.$ 
\end{prop}
\begin{proof}
We first prove that  the map $\Lambda$ sends the  ball $B_R^\Kscr$ into itself  for
$R$ sufficiently large and $T$ sufficiently small.
To do this,  we need to show that  $\Lambda(\Theta)\in C^{m+1}[0, T]$ and that it satisfies the compatibility  conditions in \rfb{space}.
 
 We first note  that $\Lambda(\Theta)$ verifies the compatibility condition in \rfb{space}.
 This follows directly from the definition of $ \mathcal{I}_{\Theta}^{k, \eps}(U_{0}^n, \Theta_{0}, F^k)$.

Now we study the mapping property  of $\Lambda(\Theta)$.
Note that from Proposition \ref{schochet}, for a given $R>0$  there exists $T_{0}$ such that for  every  $\Theta \in \mathcal{B}_{T, R}^{\Kscr}$,
 for $T \leq T_{0}$  there exists a unique
 $U^\varepsilon_{\Theta} \in X^{m+1}_{T}$ solution of \eqref{modelnonlinear}
 and \eqref{latroisieme} implies that
$$\|U_\Theta^\eps\|_{H^{m+1}((0, T)\times \partial\Fscr(0))} \leq M_{\varepsilon}.$$
From the the definition of $\Lambda(\Theta)$, we clearly get that
 $\Lambda(\Theta) \in C^{m+1}([0, T])$ and that 
 $$  \| \Lambda(\Theta)\|_{C^{m+1}[0, T]} \leq C_{0} + C T^{1\over 2}M_{\varepsilon},$$
 where $C_{0}$ only depends on the initial data and $C$ is harmless.
 Indeed, we estimate
\begin{align*}
\|\Lambda(\Theta)\|_{C^{m+1}[0, T]} 
&\lesssim \left|\begin{bmatrix} l_0\\ \omega_0\end{bmatrix}\right|+T^{\frac{1}{2}}\left\|\begin{bmatrix}L \\ W\end{bmatrix} \right\|_{C^{m+1}[0, T]}\|W\|_{C^{m+1}[0, T]}\nonumber\\
&\quad +\|n_0\|_{L^2(\partial_\Sscr(0))}\max_{|\beta|\leq m+1}\int_0^T\bigl(\|\partial_t^\beta\overline p_\Theta^\eps\|_{L^2(\partial\Sscr(0))}+\|\partial_t^{|\beta|-1}\overline p_\Theta^\eps\|_{L^2(\partial\Sscr(0))}\bigr)\dd t \nonumber\\
&\lesssim |\Theta_0|+T^{\frac{1}{2}}\|\Theta\|_{C^{m+1}[0, T]}^2+c(n_0)T^{\frac{1}{2}}\|\overline p_\Theta^\eps\|_{H^{m+1}((0, T)\times \partial\Sscr(0)))}\label{ok}\\
&\lesssim |\Theta_0|+T^{\frac{1}{2}}R^2+c(n_0)T^{\frac{1}{2}}\|\overline p_\Theta^\eps\|_{H^{m+1}((0, T)\times \partial\Sscr(0))}\nonumber\\
&\lesssim  |\Theta_0|+ T^{1 \over 2}R^2 + T^{\frac{1}{2}}M_\eps,\nonumber
\end{align*}
where $M_\eps$ depending  on $R$, $\|U_0^n\|_{H^{m+1}}$ and $\|F\|_{H^{m+1}((0, T)\times \Fscr(0))}$ is introduced in \rfb{latroisieme}. 

Consequently, by choosing $R$ large enough and then $T$ sufficiently small, we obtain that 
$$ \| \Lambda(\Theta)\|_{C^{m+1}[0, T]} \leq R.$$

Note that by using estimates similar to the ones above, we also obtain that
\begin{equation}
\label{estforascoli}
|  \partial_{t}^{m+1}\Lambda(\Theta)(t)- \partial_{t}^{m+1}\Lambda(\Theta) (s) | \leq C _{0} + (t-s)^{ 1 \over 2} C(R) \qquad \forall \m s, \, t, \, 0 \leq s \leq t \leq T.
\end{equation}

We are now in a position to prove the existence of a fixed point.
We pick $\Theta^0 \in  \mathcal{B}_{T, R}^{\Kscr}$ and consider the iteration
\begin{equation}
\label{iteration} \Theta^{n+1} = \Lambda (\Theta^n)  \qquad \forall \m n \geq 0.
\end{equation}
Based on the choice of $T$ and $R$, for every $n$ we have $\Lambda( \Theta^n) \in \mathcal{B}_{T, R}^{\Kscr}$ and 
$$  \| \Theta^n\|_{C^{m+1}[0, T]} \leq  R \qquad \forall \m  n\geq 0.$$
Moreover, from \eqref{estforascoli}, we also get that
$$  | \partial_{t}^{m+1} \Theta^n(t)-\partial_{t}^{m+1}\Theta^n (s) | \leq C _{0} + (t-s)^{ 1 \over 2} M_{\varepsilon}\qquad \forall \m  0 \leq s \leq t \leq T.$$
 According to the the Arzel\`a–Ascoli theorem, this yields that up to a subsequence, $\Theta^n$ converges to $\Theta$ in $C^{m+1}[0, T]$.
 To conclude for a fixed point  from \eqref{iteration}  it suffices to prove that $\Lambda(\Theta^n)$ converges   in any topology weaker than $C^{m+1}[0, T]$.
 
 This will be a direct consequence of the following Lemma:
 \begin{lem}\label{difference}
With the same assumptions as in Proposition \ref{schochet}, and for given $\Theta_{1}$, $\Theta_{2}\in \mathcal{B}_{T, R}^\mathcal{K}$,  then  the corresponding solutions of the system \rfb{modelnonlinear}, denoted by $U_{\Theta_{1}}^\eps$ and $U_{\Theta_{2}}^\eps$, satisfy
\begin{equation*}\label{pro2}
\|U_{\Theta_{1}}^\eps-U_{\Theta_{2}}^\eps\|_{L^2((0, T)\times \partial\Fscr(0))}\leq C_{\varepsilon}[M_{\eps}, R] T^{1 \over 2}\|\Theta_{1}- \Theta_{2}\|_{C^m([0, T])}.
	\end{equation*}
\end{lem}

Before, giving the proof of this Lemma, we shall finish the proof of  Proposition \ref{fixedpoint}.
Using Lemma \ref{difference} and the definition of $\Lambda$, we obtain that
\begin{align*}
&\|\Lambda(\Theta_{1})-\Lambda(\Theta_{2})\|_{C^0([0, T])}\\
&\leq T\|\Theta_1-\Theta_2\|_{C^0[0, T]}(\|\Theta_1\|_{C^0[0, T]}+\|\Theta_2\|_{C^0[0, T]})+T^{\frac{1}{2}}\|\overline p_{L_1}^\eps-\overline p_{L_2}^\eps\|_{L^2((0,T)\times\partial\Sscr(0))}\|n_0\|_{L^2(\partial\Sscr(0))}\\
&\leq 2R T\|\Theta_1-\Theta_2\|_{C^0[0, T}+C_\eps[M_{\eps}, R] T\|\Theta_1-\Theta_2\|_{C^m[0, T]}\\
&\leq C_\eps[M_{\eps}, R] T \|\Theta_1-\Theta_2\|_{C^m[0, T]}.
\end{align*}
This yields that $\Lambda(\Theta^n)$ converges towards $\Lambda(\Theta)$ in the $C^0$-topology
and thus that $\Lambda(\Theta)= \Theta$. This ends the proof of Proposition \ref{fixedpoint}. 
\end{proof}

We present in the following the proof of  Lemma \ref{difference}.
 
\begin{proof}[Proof of Lemma \ref{difference}]
We first derive the equation for the difference $U_{\Theta_1}^\eps-U_{\Theta_2}^\eps$ which corresponds to $\Theta_1$ and $\Theta_2$, respectively. 
Still based on the model \rfb{modelnonlinear}, we make the difference and reformulate it for every $x\in \Fscr(0)$ as:
\begin{equation}\label{diffeq}
\left\{\begin{aligned}
&A^0[\Theta_1, U_{\Theta_1}^\eps]\partial_t (U_{\Theta_1}^\eps-U_{\Theta_2}^\eps)+(\Ascr[\Theta_1, U_{\Theta_1}^\eps]\cdot \nabla) (U_{\Theta_1}^\eps-U_{\Theta_2}^\eps)+B[ \Theta_1, U_{\Theta_1}^\eps](U_{\Theta_1}^\eps-U_{\Theta_2}^\eps)\\
&\qquad\qquad\qquad =-(A^0[ \Theta_1, U_{\Theta_1}^\eps]-A^0[ \Theta_2, U_{\Theta_2}^\eps])\partial_t U_{\Theta_2}^\eps-\left((\Ascr[ \Theta_1, U_{\Theta_1}^\eps]-\Ascr[ \Theta_2, U_{\Theta_2}^\eps])\cdot \nabla\right) U_{\Theta_2}^\eps\\
&\qquad \qquad\qquad\quad  -(B[ \Theta_1, U_{\Theta_1}^\eps]-B[ \Theta_2, U_{\Theta_2}^\eps])U_{L_2}^\eps,\\
& (U_{\Theta_1}^\eps-U_{\Theta_2}^\eps)(0, x)={\bf 0},  \\
& G(x)(U_{\Theta_1}^\eps-U_{\Theta_2}^\eps)=(\overline u_{\Sscr}^\eps[\Theta_1]-\overline u_{\Sscr}^\eps)[\Theta_2]\cdot n_0 \qquad \forall \m  x\in \partial\Fscr(0),
\end{aligned}\right.
\end{equation}
where the matrix $G(x)$ is introduced in \rfb{G}.

Taking the inner product of \rfb{diffeq} and $U_{\Theta_1}^\eps-U_{\Theta_2}^\eps$, we have 
\begin{align*}
&\frac{1}{2}\frac{\dd}{\dd t}\int_{\Fscr(0)}A^0[ \Theta_1, U_{\Theta_1}^\eps](U_{\Theta_1}^\eps-U_{\Theta_2}^\eps)\cdot (U_{\Theta_1}^\eps-U_{\Theta_2}^\eps)\dd x+\frac{\eps}{2}\int_{\partial\Fscr(0)}|U_{\Theta_1}^\eps-U_{\Theta_2}^\eps|^2\dd \Gamma\\
& =\frac{1}{2}\int_{\Fscr(0)}\partial_t A^0[ \Theta_1, U_{\Theta_1}^\eps](U_{\Theta_1}^\eps-U_{\Theta_2}^\eps)\cdot (U_{\Theta_1}^\eps-U_{\Theta_2}^\eps)\dd x
 +\frac{1}{2}\int_{\Fscr(0)}\div \Ascr[\Theta_1, U_{\Theta_1}^\eps](U_{\Theta_1}^\eps-U_{\Theta_2}^\eps)\cdot (U_{\Theta_1}^\eps-U_{\Theta_2}^\eps)\dd x\\
& \quad -\int_{\partial\Sscr(0)}(\overline p_{\Theta_1}^\eps-\overline p_{\Theta_2}^\eps)(\overline u_{\Theta_1}^\eps-\overline u_{\Theta_2}^\eps)\cdot n_0\dd \Gamma
 -\int_{\Fscr(0)}(A^0[ \Theta_1, U_{\Theta_1}^\eps]-A^0[\Theta_2, U_{\Theta_2}^\eps])\partial_t U_{\Theta_2}^\eps\cdot  (U_{\Theta_1}^\eps-U_{\Theta_2}^\eps)\dd x\\
&\quad -\int_{\Fscr(0)}((\Ascr[ \Theta_1, U_{\Theta_1}^\eps]-\Ascr[\Theta_2, U_{\Theta_2}^\eps])\cdot\nabla) U_{\Theta_2}^\eps\cdot  (U_{\Theta_1}^\eps-U_{\Theta_2}^\eps)\dd x\\
&\quad -\int_{\Fscr(0)}B[\Theta_1, U_{\Theta_1}^\eps](U_{\Theta_1}^\eps-U_{\Theta_2}^\eps)\cdot  (U_{\Theta_1}^\eps-U_{\Theta_2}^\eps)\dd x
 -\int_{\Fscr(0)}(B[ \Theta_1, U_{\Theta_1}^\eps]-B[\Theta_2, U_{\Theta_2}^\eps])U_{\Theta_2}^\eps\cdot  (U_{\Theta_1}^\eps-U_{\Theta_2}^\eps)\dd x.
\end{align*}
For the boundary term, we write again:
$$\int_{\partial\Sscr(0)}(\overline p_{\Theta_1}^\eps-\overline p_{\Theta_2}^\eps)(\overline u_{\Theta_1}^\eps-\overline u_{\Theta_2}^\eps)\cdot n_0\dd\Gamma\lesssim \frac{\eps}{4}\int_{\partial\Sscr(0)}|\overline p_{\Theta_1}^\eps-\overline p_{\Theta_2}^\eps|^2\dd\Gamma+c(\eps)\|\Theta_1-\Theta_2\|_{C^m[0, T]}^2,$$
and also by the estimate \rfb{smoothmatrix}
\begin{align*}
&\int_{\Fscr(0)}\partial_t A^0[ \Theta_1, U_{\Theta_1}^\eps](U_{\Theta_1}^\eps-U_{\Theta_2}^\eps)\cdot (U_{\Theta_1}^\eps-U_{\Theta_2}^\eps)\dd x\\
&\quad\quad \quad +\int_{\Fscr(0)}\div \Ascr[ \Theta_1, U_{\Theta_1}^\eps](U_{\Theta_1}^\eps-U_{\Theta_2}^\eps)\cdot (U_{\Theta_1}^\eps-U_{\Theta_2}^\eps)\dd x\\
&\lesssim \left(\|A^0[\Theta_1, U_{\Theta_1}^\eps]\|_{X^m_T}+\|\Ascr[\Theta_1, U_{\Theta_1}^\eps]\|_{X^m_T}\right)\|U_{\Theta_1}^\eps-U_{\Theta_2}^\eps\|_{L^2(\Fscr(0))}^2\\
&\leq C_\eps[ \|U_{\Theta_1}^\eps\|_{X^m_T}, \| \Theta_1\|_{C^m[0, T]}]\|U_{\Theta_1}^\eps-U_{\Theta_2}^\eps\|_{L^2(\Fscr(0))}^2.
\end{align*}
For the integrals related to the difference of $A^0$ and $\Ascr$, using their definition in \rfb{matrices} we have
\begin{align*}
&\int_{\Fscr(0)}(A^0[ \Theta_1, U_{\Theta_1}^\eps]-A^0[ \Theta_2, U_{\Theta_2}^\eps])\partial_t U_{\Theta_2}^\eps\cdot (U_{\Theta_1}^\eps-U_{\Theta_2}^\eps)\dd x \\
&\quad \quad \qquad +\int_{\Fscr(0)}((\Ascr[ \Theta_1, U_{\Theta_1}^\eps]-\Ascr[ \Theta_2, U_{\Theta_2}^\eps])\cdot \nabla) U_{\Theta_2}^\eps (U_{\Theta_1}^\eps-U_{\Theta_2}^\eps)\dd x\\
&\lesssim \|A^0[ \Theta_1, U_{\Theta_1}^\eps]-A^0[\Theta_2, U_{\Theta_2}^\eps]\|_{L^2(\Fscr(0))}\|\partial_t U_{\Theta_2}^\eps\|_{L^\infty(\Fscr(0))}\|U_{\Theta_1}^\eps-U_{\Theta_2}^\eps\|_{L^2(\Fscr(0))}\\
&\quad  +\|\Ascr[\Theta_1, U_{\Theta_1}^\eps]-\Ascr[\Theta_2, U_{\Theta_2}^\eps]\|_{L^2(\Fscr(0))}\|\nabla U_{\Theta_2}^\eps\|_{L^\infty(\Fscr(0))}\|U_{\Theta_1}^\eps-U_{\Theta_2}^\eps\|_{L^2(\Fscr(0))}\\
&\lesssim   C_\eps\left[\|U_{\Theta_1}^\eps\|_{X^m_T}, \|U_{\Theta_2}^\eps\|_{X^m_T}, \|\Theta_1\|_{C^m[0,T]}, \|\Theta_2\|_{C^m[0, T]}\right]\|U_{\Theta_1}^\eps-U_{\Theta_2}^\eps\|_{L^2(\Fscr(0))}^2\\
&\quad   +\|U_{\Theta_2}^\eps\|_{X^m_T}\|\Theta_1-\Theta_2\|_{C^m[0, T]}\|U_{\Theta_1}^\eps-U_{\Theta_2}^\eps\|_{L^2(\Fscr(0))}.
\end{align*}
Similarly, together with the estimate \rfb{smoothmatrix} we treat the integrals with $B$ as below:
\begin{align*}
&\int_{\Fscr(0)}B[ \Theta_1, U_{\Theta_1}^\eps](U_{\Theta_1}^\eps-U_{\Theta_2}^\eps)\cdot (U_{\Theta_1}^\eps-U_{\Theta_2}^\eps)\dd x +\int_{\Fscr(0)}(B[ \Theta_1, U_{\Theta_1}^\eps]-B[ \Theta_2, U_{\Theta_2}^\eps])U_{\Theta_2}^\eps \cdot (U_{\Theta_1}^\eps-U_{\Theta_2}^\eps)\dd x\\
&\lesssim\|B[ \Theta_1, U_{\Theta_1}^\eps]-B[\Theta_2, U_{\Theta_2}^\eps]\|_{L^2(\Fscr(0))}\|U_{\Theta_2}^\eps\|_{L^\infty(\Fscr(0))}\|U_{\Theta_1}^\eps-U_{\Theta_2}^\eps\|_{L^2(\Fscr(0))}\\
&\quad \qquad + \|B[\Theta_1, U_{\Theta_1}^\eps]\|_{L^\infty(\Fscr(0))}\|U_{\Theta_1}^\eps-U_{\Theta_2}^\eps\|_{L^2(\Fscr(0))}^2\\
&\lesssim C_\eps\left[\|U_{\Theta_1}^\eps\|_{X^m_T}, \|\Theta_1\|_{C^m[0, T]}\right]\|U_{\Theta_1}^\eps-U_{\Theta_2}^\eps\|_{L^2(\Fscr(0))}^2
 + \|U_{\Theta_1}^\eps\|_{X^m_T}\|\Theta_1-\Theta_2\|_{C^m[0, T]}\|U_{\Theta_1}^\eps-U_{\Theta_2}^\eps\|_{L^2(\Fscr(0))}.
\end{align*}
Combining with the above estimate and using the ellipticity of $A^0$ we conclude that
\begin{equation}\label{diff}
\begin{aligned}
&\sup_{t\in [0, T]}\|U_{\Theta_1}^\eps-U_{\Theta_2}^\eps\|_{L^2(\Fscr(0))}^2+\eps \int_0^T\|U_{\Theta_1}^\eps-U_{\Theta_2}^\eps\|_{L^2(\partial\Fscr(0))}^2\dd t\\
&\lesssim T \|\Theta_1- \Theta_2\|_{C^m[0, T]}^2 + C_\eps\left[\|U_{\Theta_1}^\eps\|_{X^m_T}, \|U_{\Theta_2}^\eps\|_{X^m_T}, \|\Theta_1\|_{C^m[0,T]}, \|\Theta_2\|_{C^m[0, T]}\right]\int_0^T\|U_{\Theta_1}^\eps-U_{\Theta_2}^\eps\|_{L^2(\Fscr(0))}^2\dd t.
\end{aligned}
\end{equation}
According to Gr\"onwall's lemma, we obtain that
$$\sup_{t\in [0, T]}\|U_{\Theta_1}^\eps-U_{\Theta_2}^\eps\|_{L^2(\Fscr(0))}^2\leq C_\eps[M_{\eps}, R] T\|\Theta_1-\Theta_2\|_{C^m[0, T]}^2,$$
where $M_{\eps}$ is given by \eqref{latroisieme}.
Therefore, we derive from \rfb{diff} directly that
$$\int_0^T\|U_{\Theta_1}^\eps-U_{\Theta_2}^\eps\|_{L^2(\partial\Fscr(0))}^2\dd t\leq C_\eps[M_{\eps}, R] T\|\Theta_1-\Theta_2\|_{C^m[0, T]}^2.$$
 This ends the proof. 
\end{proof}

Now we are ready to give the proof of Theorem \ref{existeps}. 

\begin{proof}[Proof of Theorem \ref{existeps}]
By using the fixed point provided by Proposition \ref{fixedpoint}
and Proposition \ref{schochet}
 we obtain the existence of $T^\varepsilon>0$ and  $U^\eps= U_{\Theta^\eps},$ $ \Theta^\eps$ which solve  \rfb{forcc}--\rfb{solideps} satisfying $$(U^\eps,\Theta^\eps) \in (X^{m+1}([0, T^\varepsilon]; \Fscr(0)))^3\times (C^{m+1}[0, T^\varepsilon])^2.$$
 Finally, let us observe that using that $\Theta^\varepsilon$ is a fixed point and the definition of the map $\Lambda$ in \eqref{map}, we get
 that for every $T \leq T^\varepsilon$
 $$ \| \Theta^\eps \|_{C^{m+1}([0, T])}^2 \leq C_\eps[ C_{0},  \| \Theta^\eps \|_{C^{m}([0, T])}]\left( 1 
  +  T \| U^\varepsilon \|_{H^{m+1}([0, T] \times \mathcal{F}(0))}^2 \right).$$
  By using \eqref{bornuschochetnonunif}, we thereby have
 $$  \| \Theta^\eps \|_{C^{m+1}([0, T])}^2 \leq C_\eps\left[ C_{0},  \| \Theta^\eps \|_{C^{m}([0, T])}, \| U^\varepsilon\|_{X^m_{T} }\right]\biggl( 1  + \int_{0}^T  \| \Theta^\varepsilon\|_{C^{m+1}[0, t]}^2  + \|U^\varepsilon\|_{X^{m+1}_{t}}^2 \, dt \biggr).
  $$ 
  Next, by using  \eqref{bornuschochetnonunif2}, we get from the above estimate and \eqref{bornuschochetnonunif} that
  \begin{multline*}
   \| \Theta^\eps \|_{C^{m+1}([0, T])}^2 + \| U^\varepsilon\|_{L^\infty_{T} H^{m+1}_{co}}^2\\ \leq C_\eps\left[ C_{0},  \| \Theta^\eps \|_{C^{m}([0, T])}, \| U^\varepsilon\|_{X^m_{T} }\right]\biggl( 1  + \int_{0}^T  \| \Theta^\varepsilon\|_{C^{m+1}[0, t]}^2  + \|U^\varepsilon\|_{L^\infty_{t} H^{m+1}_{co}}^2 \, dt \biggr).
  \end{multline*}
  According to the Gr\"onwall inequality, this yields
  $$  \| \Theta^\eps \|_{C^{m+1}([0, T])} + \| U^\varepsilon\|_{L^\infty_{T} H^{m+1}_{co}}
   \leq C_\eps\left[ C_{0}, \| U^\varepsilon\|_{X^m_{T} }, \| \Theta^\eps \|_{C^{m}([0, T])}\right]$$
   and the estimate \eqref{continuationcriterion} follows by using again \eqref{bornuschochetnonunif}. 
\end{proof}



\section{Uniform estimates}\label{uniform}

In this section, we begin the proof of Theorem \ref{main}.
We shall  derive for the solution given by Theorem \ref{existeps}
 estimates which are uniform  with respect to the  parameter $\varepsilon$ in \rfb{forcc}--\rfb{solideps}. 
 This will   allow us  to prove that the solution exists on an interval of time which is independent of $\varepsilon$
 and then thanks to the uniform estimates to pass to the limit in order to get Theorem \ref{main}.

To derive the uniform estimate, we take into account the solid equations to be able to obtain the same regularity control. 
We thus 
 define the following high-order energy functional  for the fluid-solid coupled system \rfb{forcc}--\rfb{solideps}, for $m\geq 3$,
\begin{equation*}\label{tangen}
\begin{aligned}
E_{m, tan}[\Theta^\eps,U^\eps](t,x):&=\sum_{|\beta|\leq m}\int_{\Fscr(0)}\left((\alpha^\eps)^{-1}|Z^\beta\overline p^\eps|^2+\eta^\eps(M Z^\beta \overline u^\eps)\cdot Z^\beta\overline u^\eps+|Z^\beta \overline s^\eps|^2\right)\dd x\\
&\quad +m|\partial_t^{\beta_0}\overline l^\eps|^2+J_0(\partial_t^{\beta_0}\overline \omega^\eps)\cdot (\partial_t^{\beta_0}\overline\omega^\eps),
\end{aligned}
\end{equation*}
which we call the ``conormal energy" in the proof.

Note that the existence time interval in Theorem \ref{existeps} depends on the parameter $\eps$. We denote a  maximal existence time by $T^\eps$ for the solution $(U^\eps, \Theta^\eps)$ defined by 
\begin{equation}\label{superT}
T^\eps=\sup \left\{T>0\m\left| \m  U^\eps \in \overline{\mathcal{V}} \text{ and } \|U^\eps\|_{X_T^m}+\|\Theta^\eps\|_{C^m[0, T]}\leq \Xi \right.\right\},
\end{equation} 
where  $\Xi >0$ will be chosen independent of $\eps$ later.  The goal is to get by the choice
 of  $\Xi>0$ a time $T^*>0$ which is independent of $\eps$ and such that  $T^*\leq T^\eps$.
 
  In this section, the estimates do not depend on $\eps \in (0, 1]$. Nevertheless, they depend
  on $\overline{\mathcal{V}}$.

\subsection{Conormal estimate of $U^\eps$.}

We shall first prove the estimate:
\begin{prop}\label{cornormaluni}
For $m\geq 3$, we have the following  conormal energy control for $T \in [0, T^\eps]$:
$$\sup_{t\in [0, T]}E_{m, tan}[\Theta^\eps,U^\eps](t)+\eps\int_0^T\|U^\eps\|_{H^m_{co}(\partial\Fscr(0))}^2\dd t\leq E_{m, tan}[\Theta^\eps,U^\eps](0)+T C[\Xi] + T \|F\|_{X^m_{T}}^2.$$ 
\end{prop}
In this section $C[\cdots]$ stand for a  positive continuous nondecreasing function which
may change from line to line which may depend on $\overline{\mathcal{V}}$ but which 
 does
not depend on $\eps \in (0, 1]$.
\begin{proof}
Note that by Theorem \ref{existeps},  for  $T \in [0, T^\eps]$ the solution is actually $X^{m+1}_{T}$ so that the solution has enough regularity for  the following energy 
estimates and integration by parts  to be  justified.

Applying the conormal vector $Z^\beta$ with $|\beta|\leq m$ to the system \rfb{forcc} and taking the inner product with $Z^\beta U^\eps$ we obtain the resulting equation:
\begin{equation}\label{corstage}
\begin{aligned}
&\frac{1}{2}\frac{\dd}{\dd t}\int_{\Fscr(0)}A^0[\Theta^\eps,U^\eps]Z^\beta U^\eps\cdot Z^\beta U^\eps\dd x+\frac{\eps}{2}\int_{\partial\Fscr(0)}|Z^\beta U^\eps|^2\dd \Gamma+\int_{\partial\Fscr(0)}Z^\beta \overline p^\eps (Z^\beta \overline u^\eps)\cdot n_0\dd \Gamma\\
&\lesssim \int_{\Fscr(0)}[A^0[\Theta^\eps,U^\eps]\partial_t,\m Z^\beta]U^\eps\cdot Z^\beta U^\eps\dd x+\int_{\Fscr(0)}[\Ascr[\Theta^\eps,U^\eps]\cdot \nabla,\m Z^\beta]U^\eps\cdot Z^\beta U^\eps\dd x\\
&\quad +\int_{\Fscr(0)}\partial_tA^0[ \Theta^\eps, U^\eps]Z^\beta U^\eps\cdot Z^\beta U^\eps\dd x+\int_{\Fscr(0)}\div \Ascr[\Theta^\eps, U^\eps]Z^\beta U^\eps\cdot Z^\beta U^\eps\dd x\\
&\quad +\int_{\Fscr(0)}Z^\beta\left(-B[\Theta^\eps,U^\eps]U^\eps+F\right)\cdot Z^\beta U^\eps \dd x.
\end{aligned}
\end{equation}
The last three integrals in \rfb{corstage} can be treated similarly as in the proof of Proposition \ref{schochet}. By using \rfb{smoothmatrix} we then have
\begin{align*}
&\int_{\Fscr(0)}\partial_tA^0[ \Theta^\eps, U^\eps]Z^\beta U^\eps\cdot Z^\beta U^\eps\dd x+\int_{\Fscr(0)}\div \Ascr[\Theta^\eps, U^\eps]Z^\beta U^\eps\cdot Z^\beta U^\eps\dd x\\
&\quad +\int_{\Fscr(0)}Z^\beta\left(-B[\Theta^\eps,U^\eps]U^\eps+F\right)\cdot Z^\beta U^\eps \dd x\\
&\leq C[\Xi]\|Z^\beta U^\eps\|_{L^2(\Fscr(0))}^2 + \|F\|_{X^m_T}^2.
\end{align*}
The integrals involving the commutators can be treated by \rfb{comm}, which implies that
\begin{align*}
&\int_{\Fscr(0)}[A^0[\Theta^\eps,U^\eps]\partial_t,\m Z^\beta]U^\eps\cdot Z^\beta U^\eps\dd x+\int_{\Fscr(0)}[\Ascr[\Theta^\eps,U^\eps]\cdot \nabla,\m Z^\beta]U^\eps\cdot Z^\beta U^\eps\dd x\\
&\leq C[\Xi]\|U^\eps\|_{X^m_T}\|Z^\beta U^\eps\|_{L^2(\Fscr(0))}.
\end{align*}

Now we mainly deal with the boundary integral depending on the pressure. Recalling the boundary condition in \rfb{forcc} and Remark \ref{usexplain}, we have 
\begin{equation*}
Z^\beta(\overline u^\eps\cdot n_0)=
\left\{\begin{aligned}
& 0 &\qquad \forall \m  x\in \partial\Om,\\
& Z^\beta(\overline l^\eps\cdot n_0)+Z^\beta(\overline \omega^\eps\cdot (x\times n_0)) &\qquad \forall \m  x\in \partial\Sscr(0).
\end{aligned}\right.
\end{equation*}
We thus have
\begin{align*}
&\int_{\partial\Fscr(0)}Z^\beta \overline p^\eps (Z^\beta \overline u^\eps)\cdot n_0\dd \Gamma=\int_{\partial\Fscr(0)}Z^\beta\overline p^\eps \left(Z^\beta(\overline u^\eps\cdot n_0)-\sum_{|\gamma|\neq 0, |\sigma+\gamma|\leq m} Z^\sigma\overline u^\eps\cdot Z^\gamma n_0\right)\dd \Gamma\\
&=-\int_{\partial\Fscr(0)}Z^\beta\overline p^\eps\sum_{|\gamma|\neq 0, |\sigma+\gamma|\leq m}Z^\sigma \overline u^\eps\cdot Z^\gamma n_0\dd \Gamma+\int_{\partial\Sscr(0)}Z^\beta\overline p^\eps (Z^\beta(\overline \omega^\eps\cdot (x\times n_0)) )\dd \Gamma.
\end{align*}
For the first boundary integral, since $n_0$ is smooth we estimate for $|\beta|\leq m$:
\begin{align*}
-\int_{\partial\Fscr(0)}Z^\beta\overline p^\eps\sum_{|\gamma|\neq 0, |\sigma+\gamma|\leq m}Z^\sigma \overline u^\eps\cdot Z^\gamma n_0\dd \Gamma
&\lesssim \sum_{|\sigma|\leq m-1}\|Z^\beta\overline p^\eps\|_{H^{-1/2}(\partial\Fscr(0))}\|Z^\sigma\overline u^\eps\|_{H^{1/2}(\partial\Fscr(0))}\\
&\lesssim \sum_{|\sigma|\leq m-1}\|Z^\sigma U^\eps\|_{H^{1/2}(\partial\Fscr(0))}^2\lesssim \sum_{|\sigma|\leq m-1}\|Z^\sigma U^\eps\|_{H^1(\Fscr(0))}^2\\
&\lesssim \|U^\eps\|_{X^m_{T}}^2.
\end{align*}
For the boundary integral at $\partial\Sscr(0)$, we discuss the following two cases for the sake of clarity. If $Z^\beta$ includes at least one  space derivative, i.e. $\beta_0\leq m-1$, integrating by parts along the (smooth) boundary,  we obtain that
$$
\int_{\partial\Sscr(0)}Z^\beta\overline p^\eps (Z^\beta(\overline \omega^\eps\cdot (x\times n_0)) )\dd \Gamma
\lesssim \|\partial_t^{\beta_0}\overline p^\eps\|_{L^2(\partial\Sscr(0))}\left(|\partial_t^{\beta_0}\overline l^\eps|+|\partial_t^{\beta_0}\overline \omega^\eps|\right)
\lesssim \|\Theta^\eps\|_{C^m[0, T]}^2+\|\overline p^\eps\|_{X^m_{T}}^2.
$$
If $Z^\beta$ only includes time derivative, i.e. $Z^\beta=\partial_t^{\beta_0}$ with $\beta_0\leq m$, then we directly have
\begin{align}\label{solipure}
\int_{\partial\Sscr(0)}Z^\beta\overline p^\eps (Z^\beta(\overline \omega^\eps\cdot (x\times n_0)) )\dd \Gamma=\int_{\partial\Sscr(0)}\partial_t^{\beta_0}\overline p^\eps (\partial_t^{\beta_0}\overline l^\eps \cdot n_0+\partial_t^{\beta_0}\overline \omega^\eps\cdot (x\times n_0))\dd\Gamma.
\end{align} 
Applying $\partial_t^{\beta_0}$ to the solid equation of $\overline l^\eps$ in \rfb{solideps} and multiplying the resulting equation by $\partial_t^{\beta_0}\overline l^\eps$ we derive that 
\begin{align*}
\frac{m}{2}\frac{\dd}{\dd t}|\partial_t^{\beta_0}\overline l^\eps|^2&=m\m\partial_t^{\beta_0}(\overline l^\eps\times \overline \omega^\eps)\cdot \partial_t^{\beta_0}\overline l^\eps+\int_{\partial\Sscr(0)}\partial_t^{\beta_0}\overline p^\eps\m \partial_t^{\beta_0}\overline l^\eps\cdot n_0\dd\Gamma\\
&=m\m\sum_{|\gamma_0|\neq 0, |\sigma_0+\gamma_0|\leq m}(\partial_t^{\sigma_0}\overline l^\eps\times \partial_t^{\gamma_0}\overline \omega^\eps)\cdot \partial_t^{\beta_0}\overline l^\eps+\int_{\partial\Sscr(0)}\partial_t^{\beta_0}\overline p^\eps\m \partial_t^{\beta_0}\overline l^\eps\cdot n_0\dd\Gamma.
\end{align*}
Doing the similar multiplication for the equation of $\overline \omega^\eps$ in \rfb{solideps}, we also have
\begin{align*}
\frac{1}{2}\frac{\dd}{\dd t}J_0\partial_t^{\beta_0}\overline \omega^\eps\cdot \partial_t^{\beta_0}\overline \omega^\eps=\sum_{|\gamma_0|\neq 0, |\sigma_0+\gamma_0|\leq m}((J_0\partial_t^{\gamma_0}\overline \omega^\eps)\times \partial_t^{\sigma_0}\overline \omega^\eps)\cdot \partial_t^{\beta_0}\overline \omega^\eps+\int_{\partial\Sscr(0)}(x\times \partial_t^{\beta_0}\overline p^\eps n_0)\cdot \partial_t^{\beta_0}\overline \omega^\eps\dd\Gamma.
\end{align*}
Hence we obtain from \rfb{solipure} that
\begin{multline*}
\int_{\partial\Sscr(0)}Z^\beta\overline p^\eps (Z^\beta(\overline \omega^\eps\cdot (x\times n_0)) )\dd \Gamma
=\frac{m}{2}\frac{\dd}{\dd t}|\partial_t^{\beta_0}\overline l^\eps|^2+\frac{1}{2}\frac{\dd}{\dd t}J_0\partial_t^{\beta_0}\overline \omega^\eps\cdot \partial_t^{\beta_0}\overline \omega^\eps\\
\quad -\m\sum_{|\gamma_0|\neq 0, |\sigma_0+\gamma_0|\leq m}\left(m (\partial_t^{\sigma_0}\overline l^\eps\times \partial_t^{\gamma_0}\overline \omega^\eps)\cdot \partial_t^{\beta_0}\overline l^\eps+((J_0\partial_t^{\gamma_0}\overline \omega^\eps)\times \partial_t^{\sigma_0}\overline \omega^\eps)\cdot \partial_t^{\beta_0}\overline \omega^\eps\right).
\end{multline*}
Moreover, we have the estimate
\begin{align*}
&\sum_{|\gamma_0|\neq 0, |\sigma_0+\gamma_0|\leq m}\left(m (\partial_t^{\sigma_0}\overline l^\eps\times \partial_t^{\gamma_0}\overline \omega^\eps)\cdot \partial_t^{\beta_0}\overline l^\eps+((J_0\partial_t^{\gamma_0}\overline \omega^\eps)\times \partial_t^{\sigma_0}\overline \omega^\eps)\cdot \partial_t^{\beta_0}\overline \omega^\eps\right)\\
&\lesssim \|\overline l^\eps\|_{C^m[0, T]}^2\|\overline \omega^\eps\|_{C^m[0, T]}+\|\overline \omega^\eps\|_{C^m[0, T]}^2\|\overline \omega^\eps\|_{C^m[0, T]}\\
&\leq C[\Xi].
\end{align*}
Summing over $1\leq |\beta|\leq m$ and taking the integration on $[0, T]$, we conclude from \rfb{corstage} that
\begin{equation*}\label{cornorest}
\begin{aligned}
\sup_{t\in [0, T]}E_{m, tan}[\Theta^\eps,U^\eps](t)+\eps\int_0^T\|U^\eps\|_{H^m_{co}(\partial\Fscr(0))}^2\dd t\leq E_{m, tan}[\Theta^\eps,U^\eps](0)+T C[\Xi] + T \|F\|_{X^m_{T}}^2.
\end{aligned}
\end{equation*}
The proof is complete.
\end{proof}

\subsection{Vorticity estimate}

From the estimates in Proposition \ref{cornormaluni}, only the estimates of normal derivatives of $U^\eps$ close to the boundary are missing. Since the boundary becomes characteristic when $\eps$ tends
to zero, in order to get uniform estimates,  we cannot use the system \eqref{forcc} to express normal derivatives in terms of conormal ones
close to the boundary for all the components of $U^\eps$. We shall thus follow for the estimates
the scheme which is used in the characteristic case in \cite{schochet1986compressible} which rely
on additional direct $X^m$ energy estimates for the vorticity and the entropy.

Here, since we have in the velocity equation  $M^{-1} \nabla p$ with a matrix $M^{-1}$, 
we  consider a modified vorticity  in order to avoid the appearance of second order derivatives of the pressure.
We thus instead take  the curl of the product $M\overline u^\eps$, i.e. $\text{Curl}\m(M\overline u^\eps)$.
This is actually a good choice because of the behavior of the matrix $M$ near the boundary $\partial\Fscr(0)$, which we have indicated in Remark \ref{usexplain}.  

We obtain that
\begin{equation}\label{curlnew}
\partial_t\text{Curl}\m(M\overline u^\eps)+\left((\overline u^\eps-\overline u^\eps_\Sscr+\eps(\eta^\eps M)^{-1}\nu)\cdot \nabla\right)\text{Curl}\m(M\overline u^\eps)=\Rscr(U^\eps, \overline l^\eps, \overline \omega^\eps),
\end{equation}
with 
\begin{equation}\label{Rscr}
\begin{aligned}
\Rscr[\Theta^\eps, U^\eps]:&=\text{Curl}\m(\partial_tM\overline u^\eps)+``\nabla(M\overline u^\eps)\cdot \nabla(\overline u^\eps-\overline u^\eps_\Sscr)"+\text{Curl}\m(\overline u^\eps \nabla M\cdot (\overline u^\eps-\overline u^\eps_\Sscr))\\
&\quad +\eps``\nabla(M\overline u^\eps)\cdot \nabla((\eta^\eps M)^{-1}\nu)"+\eps\text{Curl}\m\left((\eta^\eps M)^{-1}\overline u^\eps\nabla M\cdot \nu\right)+(\eta^\eps)^{-2}\nabla\eta^\eps\times \nabla\overline p^\eps\\
&\quad +\text{Curl}\m \left(M\left(\nabla\Jscr\m\Jscr^{-1} \overline u^\eps\cdot (\overline u-\overline u_\Sscr^\eps)+\partial_t\Jscr \Jscr^{-1}\overline u^\eps\right)+\eps(\eta^{\eps})^{-1}F_{u}\right).
\end{aligned}
\end{equation}
During the derivation of \rfb{curlnew} we used the fact that
$$\text{Curl}\m\left( -(\eta^\eps)^{-1}\nabla\overline p^\eps\right)=(\eta^\eps)^{-2}\nabla\eta^\eps\times \nabla\overline p^\eps, $$
$$\text{Curl}\m (u\cdot \nabla) v=(u\cdot \nabla) \text{Curl}\m v + ``\nabla v \cdot \nabla u".$$
In the above the notation $``\nabla v\cdot \nabla u"$ represents any  product of  first derivatives of $u$ and $v$. 

Based on the equation of $\text{Curl}\m (M\overline u^\eps)$ in \rfb{curlnew} we have the following estimate.

\begin{prop}\label{curlestimate}
For $m\geq 3$, we have the $X^{m-1}_T$-norm control of $\text{Curl}\m (M\overline u^\eps)$:
\begin{equation*}
\|\text{Curl}\m(M\overline u^\eps)(t)\|_{X^{m-1}_T(\Fscr(0))}^2
\leq \interleave\text{Curl}\m(M\overline u^\eps)(0)\interleave_{m-1, \Fscr(0)}^2+TC[\Xi] + T \| F\|_{X^m_{T}}^2.
\end{equation*}  
\end{prop}
\begin{proof}
Applying the full derivative $D^\beta$, defined in \rfb{D}, with $|\beta|\leq m-1$ to the equation \rfb{curlnew}, we obtain that
\begin{equation}\label{new1}
\begin{aligned}
&\partial_t D^\beta\text{Curl}(M\overline u^\eps)+\left((\overline u^\eps-\overline u^\eps_\Sscr+\eps(\eta^\eps M)^{-1}\nu)\cdot \nabla\right)D^\beta\text{Curl}\m(M\overline u^\eps)\\
&=\left[(\overline u^\eps-\overline u^\eps_\Sscr+\eps(\eta^\eps M)^{-1}\nu)\cdot \nabla, D^\beta\right]\text{Curl}(M\overline u^\eps)+D^\beta \Rscr[\Theta^\eps,U^\eps],
\end{aligned}
\end{equation}
where $\Rscr$ is introduced in \rfb{Rscr}. Taking the inner product of \rfb{new1} and $D^\beta\text{Curl}\m (M\overline u^\eps)$ in $L^2(\Fscr(0))$, we note that
\begin{align*}
&\int_{\Fscr(0)}\left((\overline u^\eps-\overline u^\eps_\Sscr+\eps(\eta^\eps M)^{-1}\nu)\cdot \nabla\right)D^\beta\text{Curl}\m(M\overline u^\eps)\cdot D^\beta \text{Curl}\m (M\overline u^\eps)\dd x\\
&=\frac{\eps}{2}\int_{\partial\Fscr(0)}(\eta^\eps)^{-1} D^\beta\text{Curl}\m(M\overline u^\eps)\cdot D^\beta \text{Curl}\m (M\overline u^\eps)\dd\Gamma
\\ &\quad -\frac{1}{2}\int_{\Fscr(0)}\div \left(\overline u^\eps-\overline u^\eps_\Sscr+\eps(\eta^\eps M)^{-1}\nu\right)D^\beta\text{Curl}\m(M\overline u^\eps)\cdot D^\beta \text{Curl}\m (M\overline u^\eps)\dd x,
\end{align*}
where we used the fact \rfb{Mbound}.
This gives us the following expression:
\begin{equation}\label{aim}
\begin{aligned}
&\frac{1}{2}\frac{\dd}{\dd t}\|D^\beta\text{Curl}\m(M\overline u^\eps)\|_{L^2(\Fscr(0))}^2+\frac{\eps}{2}\int_{\partial\Fscr(0)}(\eta^\eps)^{-1} D^\beta\text{Curl}\m(M\overline u^\eps)\cdot D^\beta \text{Curl}\m (M\overline u^\eps)\dd\Gamma\\
&=\frac{1}{2}\int_{\Fscr(0)}\div \left(\overline u^\eps-\overline u^\eps_\Sscr+\eps(\eta^\eps M)^{-1}\nu\right)D^\beta\text{Curl}\m(M\overline u^\eps)\cdot D^\beta \text{Curl}\m (M\overline u^\eps)\dd x\\
&\quad +\int_{\Fscr(0)}\left[(\overline u^\eps-\overline u^\eps_\Sscr+\eps(\eta^\eps M)^{-1}\nu)\cdot \nabla, D^\beta\right]\text{Curl}(M\overline u^\eps)\cdot D^\beta \text{Curl}\m (M\overline u^\eps)\dd x\\
&\quad +\int_{\Fscr(0)}D^\beta \Rscr[\Theta^\eps, U^\eps]\cdot D^\beta \text{Curl}\m (M\overline u^\eps)\dd x.
\end{aligned}
\end{equation}
We estimate the right hand side of \rfb{aim} and have
\begin{align*}
&\int_{\Fscr(0)}\div\left(\overline u^\eps-\overline u^\eps_\Sscr+\eps(\eta^\eps M)^{-1}\nu\right)D^\beta\text{Curl}\m(M\overline u^\eps)\cdot D^\beta \text{Curl}\m (M\overline u^\eps)\dd x\\
&\lesssim \left(\|\div(\overline u^\eps-\overline u_\Sscr^\eps)\|_{L^\infty(\Fscr(0))}+\eps \|\nabla((\eta^\eps M)^{-1}\nu)\|_{L^\infty(\Fscr(0))}\right)\|D^\beta\text{Curl}\m(M\overline u^\eps)\|_{L^2(\Fscr(0))}^2\\
&\lesssim C[\Xi]\|D^\beta\text{Curl}\m(M\overline u^\eps)\|_{L^2(\Fscr(0))}^2.
\end{align*}
For the integral including the commutator, using the fact that $[\nabla, D^\beta]=0$ we obtain from \rfb{comm} 
if $m-1 \geq 3$ or \eqref{comm=2} if $m-1=2$ 
that
\begin{align*}
&\int_{\Fscr(0)}\left[(\overline u^\eps-\overline u^\eps_\Sscr+\eps(\eta^\eps M)^{-1}\nu)\cdot \nabla, D^\beta\right]\text{Curl}(M\overline u^\eps)\cdot D^\beta \text{Curl}\m (M\overline u^\eps)\dd x\\
&\lesssim \|[(\overline u^\eps-\overline u_\Sscr^\eps+\eps(\eta^\eps M)^{-1}\nu)\cdot \nabla, D^\beta]\text{Curl}\m (M\overline u^\eps)\|_{L^2(\Fscr(0))}\|D^\beta\text{Curl}\m(M\overline u^\eps)\|_{L^2(\Fscr(0))}\\
&\lesssim C\left[\|\overline u^\eps-\overline u^\eps_\Sscr\|_{X^{m}_T}, \eps\|(\eta^\eps M)^{-1}\nu\|_{X^{m}_T}\right] \|\text{Curl}\m(M\overline u^\eps)\|_{X^{m-1}_T}\|D^\beta\text{Curl}\m(M\overline u^\eps)\|_{L^2(\Fscr(0))}\\
&\lesssim C[\Xi]\|\text{Curl}\m(M\overline u^\eps)\|_{X^{m-1}_T}\|D^\beta\text{Curl}\m(M\overline u^\eps)\|_{L^2(\Fscr(0))},
\end{align*}
where we used the product estimate in Remark \ref{productestimate}.
Finally for the integral with the lower order term $\Rscr$, we find that
\begin{align*}
\int_{\Fscr(0)}D^\beta \Rscr[\Theta^\eps, U^\eps]\cdot D^\beta \text{Curl}\m (M\overline u^\eps)\dd x
&\lesssim \|D^\beta\Rscr[\Theta^\eps, U^\eps]\|_{L^2(\Fscr(0))}\|D^\beta\text{Curl}\m(M\overline u^\eps)\|_{L^2(\Fscr(0))}\\
&\lesssim C[\Xi] + \|F\|_{X^m_{T}}^2.
\end{align*}
Combining the above estimates and taking the integration with respect to time on $(0, T)$ we obtain from \rfb{aim}, by summing over $|\beta|\leq m-1$, that
\begin{align*}
&\|\text{Curl}\m(M\overline u^\eps)(t)\|_{X^{m-1}_T(\Fscr(0))}^2+\eps\Xi^{-1}\|\text{Curl}\m(M\overline u^\eps)\|_{H^{m-1}((0, T)\times \partial\Fscr(0))}^2\\
&\leq \interleave\text{Curl}\m(M\overline u^\eps)(0)\interleave_{m-1, \Fscr(0)}^2+TC[ \Xi] + T \|F\|_{X^m_{T}}^2,
\end{align*}
which ends the proof.
\end{proof}

\subsection{Estimate of the entropy $\overline s^\eps$}

Now we take the equation of $\overline s^\eps$ from \rfb{forcc}, which is a transport system. As for  the estimate of curl, we estimate direclty the $X_T^{m}(\Fscr(0))$-norm.

\begin{prop}\label{entroprop}
For the entropy $\overline s^\eps$, we obtain its $X_T^m(\Fscr(0))$ control as below:
\begin{align*}
\|\overline s^\eps\|_{X^m_T(\Fscr(0))}^2\leq \interleave\overline s^\eps(0)\interleave_{m,\Fscr(0)}^2+ TC[\Xi]
 + T \|F\|_{X^m_{T}}^2.
\end{align*}
\end{prop}
\begin{proof}
Applying the full derivative $D^\beta$ with $|\beta|\leq m$ to the equation of $\overline s^\eps$, we have
\begin{equation*}\label{entropy}
\partial_tD^\beta\overline s^\eps+\left((\overline u^\eps-\overline u^\eps_\Sscr+\eps\nu)\cdot \nabla\right) D^\beta\overline s^\eps=\left[(\overline u^\eps-\overline u^\eps_\Sscr+\eps\nu)\cdot \nabla, D^\beta\right]\overline s^\eps+\eps D^\beta(F_{s}).
\end{equation*}
The proof of the energy estimate then follows the same lines as above, this is left to the reader.
\end{proof}

\subsection{Control of normal derivatives close to the boundary}
The goal of this subsection is to prove:
\begin{prop}\label{propnorm}
For $m \geq 3$, $T \leq T^\eps$,  we have
  \begin{multline}
  \label{tobedone}
\| U^\eps\|_{X^m_{T}}
  \leq  C[ \|U^\eps \|_{X^{m-1}_{T}}, \| \Theta^\eps \|_{C^{m-1}([0, T])} ] (  \| U^\eps
   \|_{L^\infty_{T}H^m_{co}}   \\ 
  +  \|U^\eps \|_{X^{m-1}_{T}} + \|F \|_{X^{m-1}_{T}}) + \|\overline s^\eps \|_{X^{m}_{T}}  +  \| \text{Curl}\m M \overline u^\eps \|_{X^{m-1}_{T}}.
\end{multline}  
\end{prop}
\begin{proof}
Again the difficulty is to get the control of normal derivatives close to the boundary.
By using the same arguments as in the derivation of \eqref{notice}, we still have
\begin{equation}
\label{normal1}
\|U^\eps\|_{X^{m}_T(\mathcal{O}_i)} \lesssim  \sum_{i=0}^{m-1}\|\partial_t^i\partial_z^{m-i} U^\eps\|_{X^0_{T}(\mathcal{O}_i)}
+ \|U^\eps\|_{X^{m-1}_{T}} +  \|U^\eps\|_{L^\infty_{T} H^{m}_{co}}.
\end{equation}
where in  $\mathcal{O}_{i}$ the boundary can be locally parametrized as a graph $x_{3}= \varphi_{i}(x_{1}, x_{2}).$
In local coordinates, the unit  normal is given by 
$$n(\Phi_i(y, z))=\frac{1}{\sqrt{1+|\nabla\varphi_i(y)|^2}}\begin{bmatrix}
\partial_{x_1}\varphi_i(y) \\ \partial_{x_2}\varphi_i(y) \\ -1
\end{bmatrix}. $$
We note that the normal vector above can be smoothly extended away from the boundary as a vector field independent of $x_3$. With this we define the orthogonal projection
of a vector $X$:
\begin{equation*}
\Pi(\Phi_i(y,z))X=X-X\cdot n(\Phi_i(y,z))n(\Phi_i(y,z)),
\end{equation*}
which projects on the tangent space to the boundary that is the space generated by
 $(\partial_{y_{1}}, \partial_{y_{2}})$.
We can then use the decomposition of the vector field $\overline{u}^\eps$:
$$\overline u^\eps=\overline u^\eps_\tau+\overline u^\eps_n n, \quad  \overline u^{\eps}_{\tau}
 = \Pi\overline u^\eps, \quad u^\eps_n = \overline u^\eps \cdot n .$$

From \eqref{normal1}, we can then obtain
\begin{equation}
\label{normal2}
\begin{aligned}
\|U^\eps\|_{X^{m}_T(\mathcal{O}_i)} &\lesssim  \sum_{i=0}^{m-1} \left(\|\partial_t^i\partial_z^{m-i} (\overline p^\eps, 
 \overline u^\eps \cdot n)\|_{X^0_{T}(\mathcal{O}_i)} +  \|\partial_t^i\partial_z^{m-i} 
 \overline u^\eps_{\tau}\|_{X^0_{T}(\mathcal{O}_i)} \right) \\ &\quad \quad + \|\overline s^\eps \|_{X^{m}_{T}} 
+ \|U^\eps\|_{X^{m-1}_{T}} +  \|U^\eps\|_{L^\infty_{T} H^{m}_{co}}.
\end{aligned}
\end{equation}
To control the part involving $\overline u^\eps_{\tau}$, we shall use our modified vorticity.
Note that  thanks to  \eqref{Mbound} in Remark \eqref{usexplain}, we can assume  that $M(t,x)$ is the identity
matrix in $\mathcal{O}_{i}$.  
We thus have that 
$$ \| \text{Curl}\m \overline u^\eps \|_{X^{m-1}_{T}(\mathcal{O}_{i})} \leq   \| \text{Curl}\m M \overline u^\eps \|_{X^{m-1}_{T}}.$$
 According to the definition of curl, for any  vector field  $X$ we have 
\begin{equation*}\label{tangenofnormal}
\text{Curl}\m X\times n = \Pi\left(\text{Curl}\m X\times n\right)=\frac{1}{2}\Pi(\nabla X-\nabla X^\intercal)n=\frac{1}{2}\left( \Pi \partial_n X- (\Pi\nabla) X  n\right).
\end{equation*}
This yields
$$ \| \partial_{z} \overline u^\eps_{\tau} \|_{X^{m-1}_{T}(\mathcal{O}_i)}
 \lesssim  \| \text{Curl}\m M \overline u^\eps \|_{X^{m-1}_{T}}  + \|\nabla_{y}U^\eps\|_{X^{m-1}_{T}(\mathcal{O}_{i})} 
  + \|U^\eps\|_{X^{m-1}_{T}}.$$
 We thus have from \eqref{normal2} that
 \begin{equation}\label{normal3}
 \begin{aligned}
\|U^\eps\|_{X^{m}_T(\mathcal{O}_i)} &\lesssim  \sum_{i=0}^{m-1} \|\partial_t^i\partial_z^{m-i} (\overline p^\eps, 
 \overline u^\eps \cdot n)\|_{X^0_{T}(\mathcal{O}_i)}  + \| \nabla_{y} U^\eps \|_{X^{m-1}_{T}(\mathcal{O}_{i})}  \\ 
 &\quad + \|\overline s^\eps \|_{X^{m}_{T}}  +  \| \text{Curl}\m M \overline u^\eps \|_{X^{m-1}_{T}}
+ \|U^\eps\|_{X^{m-1}_{T}} +  \|U^\eps\|_{L^\infty_{T} H^{m}_{co}}.
\end{aligned}  
\end{equation}

It thus remains to estimate   $\|\partial_t^i\partial_z^{m-i} (\overline p^\eps, 
 \overline u^\eps \cdot n)\|_{X^0_{T}(\mathcal{O}_i)}$.


Still using the notation introduced in Subsection \ref{notation} and $\nabla_y$ defined in \rfb{newgradient}, we shall consider  the equations  for  $\overline p^\eps$ and $\overline u^\eps_n$ from \rfb{forcc}.
For any vector field $X$ we denote by $X_{h}= (X_{1}, X_{2})$ and  $X_{3}$ the 
coordinates in the canonical basis of $\mathbb{R}^3$ and we still use
 $X_{n}= X \cdot n$, $X_{\tau}= \Pi X$.

Since we can use again  \rfb{Mbound} in Remark \ref{usexplain}, we get that in $\mathcal{O}_{i}$,
\begin{equation*}\label{equanor}
\begin{aligned}
\mathscr{A}_z^\eps[\Theta^\eps, U^\eps]
\begin{bmatrix}
\partial_z\overline p^\eps \\ \partial_z\overline u^\eps_n 
\end{bmatrix}&=\begin{bmatrix}
-\partial_t\overline p^\eps-(\overline u^\eps-\overline u^\eps_\Sscr)_h\cdot \nabla_y\overline p^\eps-\eps\alpha^\eps \nu_h\cdot \nabla_y\overline p^\eps-\alpha^\eps\nabla_y\cdot \overline u^\eps_h\\
-\partial_t\overline u^\eps_n-(\overline u^\eps-\overline u_\Sscr)_h\cdot \nabla_\tau\overline u^\eps_n-\eps(\eta^\eps)^{-1}\nu_h\cdot \nabla_y\overline u^\eps_n-(\eta^\eps)^{-1}\nabla_y\overline p^\eps\cdot\nabla_y\varphi_i
\end{bmatrix}\\
&\quad +
\begin{bmatrix}
	-\alpha^\eps\m \text{tr}(\Jscr_2\m\overline u^\eps)+\alpha^\eps F_{p}n\\
	\left(-\Jscr_1^{-1}\Jscr_2\m \overline u^\eps\cdot (\overline u-\overline u_\Sscr^\eps)-\nabla_\phi V\m\overline u^\eps\right)_n+(\eta^\eps)^{-1} (F_{u})_n
\end{bmatrix}\\
&:= \begin{bmatrix}
\Mscr_1 \\ \Mscr_2 \end{bmatrix} +\begin{bmatrix} 
\Rscr_1 \\ \Rscr_2 \end{bmatrix},
\end{aligned}
\end{equation*}
where the  matrix on the left hand side is 
$$\mathscr{A}_z^\eps[\Theta^\eps, U^\eps]:=-\begin{bmatrix}
	(\overline u^\eps-\overline u^\eps_\Sscr)_n+\eps\alpha^\eps \nu_n & \alpha^\eps \\
	(\eta^\eps)^{-1}(1+|\nabla_y\varphi_i|^2)& (\overline u^\eps-\overline u^\eps_\Sscr)_n+\eps(\eta^\eps)^{-1}\nu_n
\end{bmatrix}.$$
 We note  that we can assume that  $\mathscr{A}_z^\eps[\Theta^\eps, U^\eps]$ is uniformly  invertible in $\mathcal{O}_i$. Indeed,
 since $(\overline u^\eps-\overline u_\Sscr^\eps)_n$ vanishes on the boundary, we have that
 in a $\delta$ neighborhood of the boundary
 $$ |(\overline u^\eps-\overline u_\Sscr^\eps)_n|\leq \delta C[\Xi],$$
 and therefore by choosing $\delta$ small enough that the determinant $\det(\mathscr{A}_z^\eps)$ satisfies
$$|\det(\mathscr{A}_z^\eps)|\geq  \kappa,$$
where $\kappa$ depends only on $\mathcal{V}$.

We thus get  the expression of $\partial_z\overline p^\eps$ and $\partial_z\overline u^\eps_n$ in $\mathcal{O}_i$:
\begin{equation*}\label{acousticnormal}
\begin{bmatrix}\partial_z\overline p^\eps \\ \partial_z\overline u^\eps_n 
\end{bmatrix}=(\mathscr{A}_z^\eps)^{-1}[\Theta^\eps, U^\eps]\left(
\begin{bmatrix}
\Mscr_1 \\ \Mscr_2 \end{bmatrix} +\begin{bmatrix} 
\Rscr_1 \\ \Rscr_2 \end{bmatrix}\right).
\end{equation*}
We then apply $\partial_{t}^i \partial_{z}^{m-1 - i}$ to the above expression.
By using again Proposition \ref{productestimate} for $n=m-1$, $k=0$,  the terms involving $\mathcal{R}_{j}$, $j=1\, 2$,
which contains only zero order terms,  can be estimated as
$$ \left\| \partial_{t}^i \partial_{z}^{m-1 - i} \left( (\mathscr{A}_z^\eps)^{-1}
\begin{bmatrix} 
\Rscr_1 \\ \Rscr_2 \end{bmatrix} \right) \right\|_{X^0_{T}(\mathcal{O}_{i})}
 \leq C[ \|U^\eps \|_{X^{m-1}_{T}}, \| \Theta^\eps \|_{C^{m-1}([0, T])} ] (  \|U^\eps \|_{X^{m-1}_{T}} + 
 \|F\|_{X^{m-1}_{T}}).$$
 To estimate the terms involving  $\mathcal{M}_{j}$, we  write
 \begin{align*}    
 &\left\| \partial_{t}^i \partial_{z}^{m-1 - i} \left( (\mathscr{A}_z^\eps)^{-1}
\begin{bmatrix} 
\Mscr_1 \\ \Mscr_2 \end{bmatrix} \right) \right\|_{X^0_{T}(\mathcal{O}_{i})} \\
  &\leq C[ \|U^\eps \|_{X^{m-1}_{T}}, \| \Theta^\eps \|_{C^{m-1}([0, T])} ] (  \| \partial_{t}^{i+1}
  \partial_{z}^{m-1-i} (\overline p^\eps, 
 \overline u^\eps \cdot n)  \|_{X^0_{T}(\mathcal{O}_{i})} + \| \nabla_{y} U^\eps \|_{X^{m-1}_{T}(\mathcal{O}_{i})})  + 
  \mathcal{C},
  \end{align*}
  where $\mathcal{C}$ is a commutator term. This commutator can be controlled by using
   \eqref{calculus2} in Proposition \ref{calculus}, 
   we obtain for every $\mu>0$ to be chosen
  \begin{align*}    
 &\left\| \partial_{t}^i \partial_{z}^{m-1 - i} \left( (\mathscr{A}_z^\eps)^{-1}
\begin{bmatrix} 
\Mscr_1 \\ \Mscr_2 \end{bmatrix} \right) \right\|_{X^0_{T}(\mathcal{O}_{i})} \\
 & \leq C[ \|U^\eps \|_{X^{m-1}_{T}}, \| \Theta^\eps \|_{C^{m-1}([0, T])} ] (  \| \partial_{t}^{i+1}
  \partial_{z}^{m-1-i}  (\overline p^\eps, 
 \overline u^\eps \cdot n) \|_{X^0_{T}(\mathcal{O}_{i})} + \| \nabla_{y} U^\eps \|_{X^{m-1}_{T}(\mathcal{O}_{i})}  \\ 
 &\quad + \mu \| U^\eps \|_{X^m_{T}(\mathcal{O}_{i})} 
  + C_{\mu} \|U^\eps \|_{X^{m-1}_{T}}).
  \end{align*}
  This yields
\begin{align*}
\|\partial_t^i\partial_z^{m-i} (\overline p^\eps, 
 \overline u^\eps \cdot n)\|_{X^0_{T}(\mathcal{O}_i)}
 & \leq  C[ \|U^\eps \|_{X^{m-1}_{T}}, \| \Theta^\eps \|_{C^{m-1}([0, T])} ] (  \| \partial_{t}^{i+1}
  \partial_{z}^{m-1-i} (\overline p^\eps, 
 \overline u^\eps \cdot n) \|_{X^0_{T}(\mathcal{O}_{i})}\\ 
 &\quad + \| \nabla_{y} U^\eps \|_{X^{m-1}_{T}(\mathcal{O}_{i})} + \mu \| U^\eps \|_{X^m_{T}(\mathcal{O}_{i})} 
  + C_{\mu} \|U^\eps \|_{X^{m-1}_{T}} + \|F \|_{X^{m-1}_{T}}),
\end{align*}  
  for every $0 \leq i \leq m-1$.  By iterating the above estimate as before, we thus get
  \begin{align*}
\|\partial_t^i\partial_z^{m-i} (\overline p^\eps, 
 \overline u^\eps \cdot n)\|_{X^0_{T}(\mathcal{O}_i)}
&  \leq  C[ \|U^\eps \|_{X^{m-1}_{T}}, \| \Theta^\eps \|_{C^{m-1}([0, T])} ] (  \| U^\eps
   \|_{L^\infty_{T}H^m_{co}} + \| \nabla_{y} U^\eps \|_{X^{m-1}_{T}(\mathcal{O}_{i})}  \\ 
   &\quad + \mu \| U^\eps \|_{X^m_{T}(\mathcal{O}_{i})} 
  + C_{\mu} \|U^\eps \|_{X^{m-1}_{T}} + \|F \|_{X^{m-1}_{T}}).
\end{align*}  
We can then plug this estimate into \eqref{normal3} to obtain that
  \begin{align*}
\| U^\eps\|_{X^m_{T}(\mathcal{O}_i)}
 & \leq  C[ \|U^\eps \|_{X^{m-1}_{T}}, \| \Theta^\eps \|_{C^{m-1}([0, T])} ] (  \| U^\eps
   \|_{L^\infty_{T}H^m_{co}} + \| \nabla_{y} U^\eps \|_{X^{m-1}_{T}(\mathcal{O}_{i})}  \\ 
   &\quad + \mu \| U^\eps \|_{X^m_{T}(\mathcal{O}_{i})} 
  + C_{\mu} \|U^\eps \|_{X^{m-1}_{T}} + \|F \|_{X^{m-1}_{T}}) + \|\overline s^\eps \|_{X^{m}_{T}}  +  \| \text{Curl}\m M \overline u^\eps \|_{X^{m-1}_{T}}.
\end{align*} 
To conclude, we can use again \eqref{calculus1} into Proposition  \eqref{calculus}  and choose
$\mu$ sufficiently small.
 Then, the estimate \eqref{tobedone} follows by finite covering of the boundary.

\end{proof}

\section{Proof of  Theorem \ref{main}}\label{limit}

We can now finish the proof of our main result.
By using the uniform estimates of the previous section, 
we shall first  prove that we can choose $\Xi$ large enough independent of the regularization
parameters so that the maximal existence time $T^\eps$ defined in \eqref{superT}
is such that $T^\eps \geq T_{*}$ for some fixed $T_{*}>0$ independent of $\eps$.

{\bf Step 1}: {\em Uniform existence time.}

We can first  link the parameters  $\eps$ and $n$ to get an overall estimate which is uniform with respect to both parameters.
Recalling  the definitions of $F(t, x)$ and $\tilde U^n$ introduced in \rfb{matrice2} and \rfb{sourceapprox}, we define 
$$c(n):=\|(\nu\cdot \nabla)\tilde U^n\|_{X^m_{T_L}(\Fscr(0))}, \qquad
\tilde \eps(n):=\frac{1}{ 1 + c(n)},  $$
which, for $\eps\leq \tilde \eps(n)$, implies that
\begin{equation}
\label{choixpourF}
\|F(t,x)\|_{X^m_{T_L}(\Fscr(0))}\leq \eps c(n) \leq 1.
\end{equation} 
Moreover, since $U_0^n\to U_0$ in $H^m(\Fscr(0))$, we have thereby for $n$ large enough,
$$\|U_0^n\|_{H^m(\Fscr(0))}\leq \frac{3}{2}\|U_0\|_{H^m(\Fscr(0))}.$$

Let us then define
$$
Y_{m}(T):=\|\Theta^\eps\|_{C^m[0, T]}+\|U^\eps\|_{X^m_{T}},
$$
which is the quantity that we want to control.
By using  Proposition \ref{propnorm} and Proposition \ref{cornormaluni},   Proposition \ref{curlestimate}, Proposition \ref{entroprop}, together with  \eqref{choixpourF}, we get that
$$ Y_{m}(T) \leq C[ Y_{m-1}(T)](C_{0} + T^{1 \over2} (1 + C[ \xi])),$$
where $C_{0}$ depends only $ \|U_0\|_{H^m(\Fscr(0))}$ and $l_{0},\, \omega_{0}$.

From a rough integration in time of the system \eqref{forcc}-\eqref{solideps}, 
we also get that
$$ Y_{m-1}(T) \leq C_{0} +  T (C[ \Xi] + 1), \quad \|U^\eps(t) -  U_{0}^n\|_{L^\infty} \leq T C[\Xi].$$
This allows to choose $\Xi$ sufficiently large and $T_{*}$ sufficiently small
 so that
 $$  C[ C_{0} + T_{*}(C[ \xi]+ 1)] ( C_{0} + T_{*}^{1\over 2} ( 1 + C[ \Xi])) \leq {\Xi \over 2},$$
 and 
 $ \pi U^\eps(t) \in \mathcal{V}_{1}, \forall t \leq \min(T_{*}, T^\eps)$ where 
we fix $ \mathcal{V}_{1}$ such that
$ \mathcal{C} \subset\mathcal{V}_{1} \subset \overline{\mathcal{V}}_{1} \subset \mathcal{V}.$

By a standard bootstrap argument, we then get that the solution exists on $[0, T_*]$
 and that the estimate
 \begin{equation}
 \label{uniformefin}
 Y_{m}(T) \leq \Xi, \quad \forall T \in [0, T_{*}]
 \end{equation}
  holds.

%
%

Based on the uniform estimates in Section \ref{uniform}, we are going to pass to the limit $\eps\to 0$ and $n\to \infty$ as choosen in the previous subsection  to  get  the proof Theorem \ref{main}.

{\bf Step 2}: {\em Passing to the limit.}

We set  $\eps(n):=\min\{\tilde\eps(n), \frac{1}{n}\}$ and we send $n\to \infty$.
Note that we have uniform estimate  \eqref{uniformefin}, therefor, we have that
$ U^{\eps(n)}$ is uniformly bounded in $C([0, T^*], H^{m})$ and
 $\partial_{t}U^{\eps(n)}$ is uniformly bounded in $C([0, T^*], H^{m-1})$.
 We thus obtain from the Ascoli Theorem that  up to a subsequence 
 \begin{equation*}\label{strong1}
U^{\eps(n)}\to U \quad \text{strongly in} \quad  C([0, T^*]; H^s(\Fscr(0))),
\end{equation*}
for any $0\leq s<m$.
From the uniform estimate in $ C([0, T_*], H^{m})$, we deduce further that
$ U \in L^\infty(0, T^*, H^m)$ and that $U \in C_{w}([0, T_{*}], H^m)$
(meaning continuous valued in $H^m$ equipped with the weak topology).

For the solid part, from the equation for $\Theta^{\eps(n)}$ in \eqref{solideps}, we deduce that
$$| \partial_{t}^m \Theta^{\eps(n)} (t) -    \partial_{t}^m \Theta^{\eps(n)} (s)| \lesssim  
 |t-s | C[\Xi] + \| \partial_{t}^{m-1}p^{\eps(n)}(t) - \partial_{t}^{m-1} p^{(\eps(n))}(s) \|_{L^2(\partial\mathcal{S}(0))}.$$  
 This yields from the trace Theorem
 \begin{align*}
& | \partial_{t}^m \Theta^{\eps(n)} (t) -    \partial_{t}^m \Theta^{\eps(n) }(s)|  \\ &\lesssim  
 |t-s | C[\Xi] + \| \partial_{t}^{m-1}p^{\eps(n)}(t) - \partial_{t}^{m-1} p^{\eps(n)}(s) \|_{H^1(\mathcal{F}(0))}^{1 \over 2 }  \| \partial_{t}^{m-1}p^{\eps(n)}(t) - \partial_{t}^{m-1} p^{\eps(n)}(s) \|_{L^2(\mathcal{F}(0))}^{1 \over 2 } \\
& \lesssim C[ \Xi] ( |t-s | C[\Xi]  + |t-s|^{1 \over 2}),
 \end{align*}
  where we have used in the final estimate that 
  $$  \| \partial_{t}^{m-1}p^{\eps(n)}(t) - \partial_{t}^{m-1} p^{\eps(n)}(s) \|_{L^2(\mathcal{F}(0))}^{1 \over 2 }  
  \leq |t-s|^{1 \over 2} C[ \Xi].$$
  From the Ascoli Theorem this yields that up to a subsequence
  $ \Theta^{\eps(n)}$ converges to $\Theta$ in $C^m([0, T_{*}])$.
  
  Note that from the second  part  of  System \eqref{solideps}, we obtain that $\partial_{t}^{m+1} \Upsilon^{\eps(n)}$
   is uniformly bounded so that $ \Upsilon^{\eps(n)}$ also converges up to subsequence
   towards $\Upsilon$ in $C^m([0, T_{*}]).$
   
   From these strong convergences, we clearly get that $(U, \Theta, \Upsilon)$ solves
    \eqref{first}, \eqref{solidbar}, that $\pi U \in \overline{\mathcal{V}}$ and that it is a $C^1([0, T] \times \mathcal{F}(0))$ solution
    by Sobolev embedding.
    
   From this regularity, we can easily get the uniqueness, by estimating $E_{0, tan}[ \Theta_{1}- \Theta_{2}, U_{1}- U_{2}]$
  for  two solutions.

To conclude it remains to show that  $U \in X^m_{T_*}$ 
 that is to say that  $\partial_{t}^k U$ is continuous in time with values in $H^{m-k}(\Fscr(0))$
  for $0 \leq k \leq m$.
We can follow the argument for the compressible Euler system in \cite{schochet1986compressible}.
Indeed, we already know that $U \in C_{w}([0, T_{*}], H^m)$ and  $(\Theta, \Upsilon) \in C^m([0, T_{*}])$
 so that thanks to the system \eqref{first}
 $ \partial_{t}^k U \in C_{w}([0, T_{*}], H^{m-k})$ for $0 \leq k \leq m$.
  To prove the continuity it thus suffices to prove that $\| \partial_{t}^k  U\|_{H^{m-k}}$ 
   is continuous and since the system  is reversible  it suffices to prove the right continuity
    of  $\| \partial_{t}^k  U\|_{H^{m-k}}$. 
    We shall prove this at the initial time. Since the uniqueness, is already established, 
    the same argument holds by considering  the Cauchy data taken at any time $s\in [0, T]$ instead of $0$.
     Passing to the limit in the estimate
     of  Proposition \ref{cornormaluni} and using that we have strong $H^m$ convergence
     at the initial time  we obtain
      $$ E_{m, tan}[ \Theta, U] (t) \leq  E_{m, tan}[ \Theta, U] (0) + C[\Xi]t.$$
      This yields the right continuity of   $E_{m, tan}[ \Theta, U] (t)$ at $0$
      and hence of $\|U(t)\|_{H^m_{co}}$ since $\Theta \in C^m([0, T_{*}])$ and $U \in C([0, T_{*}], H^{m-1})$. 
     In a similar way proceeding as in Proposition \ref{curlestimate}
      and Proposition \ref{entroprop} we also get that $ \|\partial_{t}^k\text{Curl } ( M \overline u^\eps) \|_{H^{m-1-k}}$
      and $ \|\partial_{t}^k \overline s^\eps \|_{H^{m-k}}$ are right continuous. 
      We then use the system \eqref{first} to deduce that $\|\partial_{t}^k U \|_{H^{m-k}}$ is continuous
       close to the boundary.

\appendix

\section{Useful estimates}\label{append}

We now present some important product estimates, which are frequently used in the energy estimates. In the following we still use the similar notation introduced in \rfb{defmatricesfix} i.e. $\Mscr[u](t,x):=\Mscr(t,x, u(t,x))$.

\begin{prop}\label{productestimate}
For $k  =0, 1$,  we have the following estimates:
\begin{enumerate} 
\item For $n\geq 2$ and $u$, $v\in X^{n+k}([0, T]; \Oscr)$, the product $u v\in X^{n+k}([0, T];\Oscr)$ and
\begin{equation}\label{proXm}
\|u\m v\|_{L^\infty_T(H^{n+k}_{co}(\Oscr))}\lesssim\|u\m v\|_{X_T^{n+k}(\Oscr)}\lesssim \|u\|_{X_T^{n}(\Oscr)} \|v\|_{X_T^{n+k}(\Oscr)}
+\|u\|_{X_T^{n+k}(\Oscr)}\|v\|_{X_T^{n}(\Oscr)}. 
\end{equation}
\item For $n\geq 2$, assume further that $\Mscr(t,x, u)\in C^{n+k}([0, T]\times \overline{O} \Oscr\times \Uscr)$ with $u(t,x)\in \mathcal{V}$, we have
\begin{equation}\label{smoothmatrix}
\|\Mscr[u](t,x))\|_{X^{n+k}_T(\Oscr)}\leq
\left\{\begin{aligned}
&C\left[\|\Mscr(t,x,u)\|_{C^n}\right](1+\|u\|_{X_T^n(\Oscr)}^n) \, &k=0,\\
& C[\|\Mscr(t,x,u)\|_{C^{n+1}}](1+\|u\|_{X^n_T(\Oscr)}^n)(1+\|u\|_{X_T^{n+1}(\Oscr)}) \,&k=1.
\end{aligned}
\right. 
\end{equation}
\item With the above assumptions, for $m\geq 3$ and $|\beta|\leq m+k$, we have the commutator estimates:
\begin{equation}\label{comm}
\|[\Mscr[u](t,x)D, Z^\beta]v\|_{L^2(\Oscr)}\leq	
\left\{\begin{aligned}
&C[\|\Mscr(t,x,u)\|_{C^m}](1+\|u\|_{X^{m-1}_T(\Oscr)}^{m-1})(1+\|u\|_{X^m_T(\Oscr)})\|v\|_{X^m_T(\Oscr)} \,&k=0,\\
&C[\|\Mscr(t,x,u)\|_{C^{m+1}}](1+\|u\|_{X^m_T(\Oscr)}^m)\\
&\quad \times\left((1+\|u\|_{X^m_T(\Oscr)})\|v\|_{X^{m+1}_T(\Oscr)}+(1+\|u\|_{X^{m+1}_T(\Oscr)})\|v\|_{X^{m}_T(\Oscr)} \right)\, &k=1.
\end{aligned}\right.
\end{equation}
\end{enumerate}
The above estimate also  holds  for $Z^\beta$ replaced by  $D^\beta$.\\
The norms of the matrices are defined as $\| \mathcal{M}\|_{\mathcal{C}^k}= \sup_{[0, T] \times \overline{O} \times \overline{\mathcal{V}}} | \partial^\alpha \mathcal{M}(t,x,u)|.$
\end{prop}

\begin{proof}

We first prove \rfb{proXm} for $k=0$. For $u$, $v\in X^n([0, T];\Oscr)$, we have in particular,
\begin{align*}
&\|u\m v\|_{L^\infty_T(H^n_{co}(\Oscr))}\leq \sup_{t\in [0, T]}\sum_{|\beta+\gamma|\leq n}\|Z^\beta u\cdot Z^\gamma v\|_{L^2(\Oscr)}\\
&\lesssim \sup_{t\in [0, T]}\sum_{|\beta+\gamma|\leq n}\sum_{0\leq \beta_0\leq |\beta|, 0\leq \gamma_0\leq |\gamma|}\|\partial_t^{\beta_0}Z^{|\beta|-\beta_0}u\cdot \partial_t^{\gamma_0}Z^{|\gamma|-\gamma_0}v\|_{L^2(\Oscr)}\\
&\lesssim\sup_{t\in [0, T]} \sum_{|\beta+\gamma|\leq n}\|\partial_t^{\beta_0}Z^{|\beta|-\beta_0}u\|_{H^{n-|\beta|}(\Oscr)}\|\partial_t^{\gamma_0}Z^{|\gamma|-\gamma_0}v\|_{H^{|\beta|}(\Oscr)}\\
&\lesssim\sup_{t\in [0, T]} \sum_{|\beta+\gamma|\leq n}\|\partial_t^{\beta_0}u\|_{H^{n-\beta_0}(\Oscr)}\|\partial_t^{\gamma_0}v\|_{H^{n-\gamma_0}(\Oscr)}\\
& \lesssim \|u\|_{X^n_T(\Oscr)}\|v\|_{X_T^n(\Oscr)},
\end{align*}
where we used the classical  inequality,
$$\|u\m v\|_{H^{\sigma-k-l}(\Oscr)}\lesssim \|u\|_{H^{\sigma-k}(\Oscr)}\|v\|_{H^{\sigma-l}(\Oscr)}\qquad \forall \m  0\leq k+l\leq \sigma \text{ and } \sigma>\frac{3}{2}$$
with $\sigma=n$, $k=|\beta|$ and $l=n-|\beta|$. 
Similarly, we have 
$$\|u\m v\|_{X_T^n(\Oscr)}\lesssim \|u\|_{X_T^n(\Oscr)}\|v\|_{X_T^n(\Oscr)}. $$
When  $k =1 $ it suffices to use the above inequality and we apply one more derivative $Z$ (for time or space) to the product. Hence we obtain the inequality \rfb{proXm}. 


For \rfb{smoothmatrix}, for $k=0$ we see that
\begin{align*}
&\|\Mscr[u](t,x)\|_{X^n_T(\Oscr)}=\sup_{t\in [0, T]}\sum_{|\alpha|+p\leq n,  p=\sum_i|\beta_i|}\|(Z^\alpha\nabla_u^p\Mscr)Z^{\beta_1}u\m Z^{\beta_2}u\cdots Z^{\beta_p}u\|_{L^2(\Oscr)}\\
&\leq C\left[\|\Mscr(t,x,u)\|_{C^n}\right]\sup_{t\in[0, T]}\sum_{0\leq p\leq n}\|Z^{\beta_1}u\m Z^{\beta_2}u\cdots Z^{\beta_p}u\|_{L^2(\Oscr)}
\leq C\left[\|\Mscr(t,x,u)\|_{C^n}\right](1+\|u\|_{X_T^n(\Oscr)}^n),
\end{align*}
where we used the product estimate in \rfb{proXm}. Similarly, for $k=1$ we have
\begin{align*}
\|\Mscr[u](t,x)\|_{X^{n+1}_T(\Oscr)}&=\sup_{t\in [0, T]}\sum_{|\alpha|+1+p\leq n+1,  p=\sum_i|\beta_i|}\|(Z^{\alpha+1}\nabla_u^p\Mscr)Z^{\beta_1}u\m Z^{\beta_2}u\cdots Z^{\beta_p}u\|_{L^2(\Oscr)}\\
&\quad\quad  +\sup_{t\in [0, T]}\sum_{|\alpha|+p+1\leq n+1,  p=\sum_i|\beta_i|}\|(Z^{\alpha}\nabla_u^{p+1}\Mscr)Z(Z^{\beta_1}u\m Z^{\beta_2}u\cdots Z^{\beta_p}u)\|_{L^2(\Oscr)}\\
&\leq C[\|\Mscr(t,x,u)\|_{C^{n+1}}](\|u\|_{X^n_T(\Oscr)}^n+\|u\|_{X^{n}_T(\Oscr)}^{n-1}\|u\|_{X_T^{n+1}(\Oscr)})\\
&\leq C[\|\Mscr(t,x,u)\|_{C^{n+1}}](1+\|u\|_{X^n_T(\Oscr)}^n)(1+\|u\|_{X_T^{n+1}(\Oscr)}).
\end{align*}

Now we estimate the commutators in \rfb{comm}. From the definition of conormal vector $Z^\beta$ in Subsection \ref{notation}, we note that $[\partial_t, Z^\beta]=0$. Then for $m\geq 3$ we derive that
\begin{align*}
\|[\Mscr[u](t,x)\partial_t, Z^\beta]v\|_{L^2(\Oscr)}&=
\|\sum_{|\gamma|\neq 0,|\sigma+\gamma|\leq m+k} Z^\gamma \Mscr[u](t,x)\cdot Z^\sigma \partial_t v\|_{L^2(\Oscr)}\\
&\lesssim \sum_{|\gamma'+\sigma|\leq m+k-1}\|Z^{\gamma'}(Z \Mscr[u](t,x))\cdot Z^\sigma \partial_t v\|_{L^2(\Oscr)}.
\end{align*}
For $k=0$, we see that
\begin{align*}
\|[\Mscr[u](t,x)\partial_t, Z^\beta]v\|_{L^2(\Oscr)}&\lesssim \|Z\Mscr[u](t,x)\cdot \partial_t v\|_{X^{m-1}_T(\Oscr)}\lesssim \|\Mscr[u](t,x)\|_{X^{m}_T(\Oscr)}\|v\|_{X^{m}_T(\Oscr)}\\
&\lesssim C[\|\Mscr(t,x,u)\|_{C^m}](1+\|u\|_{X^{m-1}_T(\Oscr)}^{m-1})(1+\|u\|_{X^m_T(\Oscr)})\|v\|_{X^m_T(\Oscr)},
\end{align*}
where we used the inequality in \rfb{proXm} for $k=0$ and  \rfb{smoothmatrix} for $k=1$. In the case of $k=1$, we derive that
\begin{align*}
\|[\Mscr[u](t,x)\partial_t, Z^\beta]v\|_{L^2(\Oscr)}&\lesssim \|Z\Mscr[u](t,x)\cdot \partial_t v\|_{X^{m}_T(\Oscr)}\\
&\lesssim \|\Mscr[u](t,x)\|_{X^m_T(\Oscr)}\|v\|_{X^{m+1}_T(\Oscr)}+\|\Mscr[u](t,x)\|_{X^{m+1}_T(\Oscr)}\|v\|_{X^m_T(\Oscr)}\\
&\lesssim C[\|\Mscr(t,x,u)\|_{C^m}](1+\|u\|_{X^{m-1}_T(\Oscr)}^{m-1})(1+\|u\|_{X^m_T(\Oscr)})\|v\|_{X^{m+1}_T(\Oscr)}\\
&\quad +C[\|\Mscr(t,x,u)\|_{C^{m+1}}](1+\|u\|_{X^m_T(\Oscr)}^m)(1+\|u\|_{X^{m+1}_T(\Oscr)})\|v\|_{X^{m}_T(\Oscr)}.
\end{align*}
In the above, we used \rfb{proXm} for $k=1$ and \rfb{smoothmatrix} for $k=1$ as well. 

Similarly, we realize that $[\nabla, Z^\beta]$ is of order $m+k$ if $|\beta|=m+k$ with $k=0, 1$. Thus,
\begin{equation*}
\|[\Mscr[u](t,x)\cdot \nabla, Z^\beta]v|_{L^2(\Oscr)}
\lesssim 
\|\Mscr[u](t,x)\cdot [\nabla,\m Z^\beta] v\|_{L^2(\Oscr)}+\sum_{|\gamma|\neq 0, |\gamma+\sigma|\leq m}\|(Z^\gamma \Mscr[u](t,x)\cdot Z^\sigma\nabla) v\|_{L^2(\Oscr)}.
\end{equation*}
This can be estimated in a similar way as above. Putting together the above estimates, we obtain \rfb{comm}. 
\end{proof}

We shall also need a refined version of the previous  commutator estimates when  we only assume that
$ | \beta | \leq k$ with $k \geq 2$.

\begin{prop}

\label{calculus}
Consider $\mathcal{O}_{i}$ a small coordinate patch near the boundary, 
\begin{itemize}
\item    For every $\alpha, \, \beta, \,\gamma$, $|\gamma \geq 1$,  for every $\mu >0$, there exists
$C_{\mu}>0$ such that
\begin{equation}
\label{calculus1} \| \partial_{t}^\alpha \partial_{z}^\beta \partial_{y}^\gamma u \|_{X^0_{T}(\mathcal{O}_{i})}
 \leq \mu \| \partial_{t}^\alpha \partial_{z}^{\beta + \gamma} u \|_{X^0_{T}(\mathcal{O}_{i})}
  + C_{\mu}\left( \|u\|_{L^\infty_{T} H^{| \alpha | + | \beta | + |\gamma|}_{co}} +  \| u\|_{X^{ | \alpha |+ |\beta | + | \gamma |-1}_{T}} \right).
  \end{equation}
 \item For every $n \geq 2$, every $\beta, \, |\beta|= n$ and $\mu >0$, there exists $C_{\mu}$
  such that
  \begin{equation} 
  \label{calculus2}
  \|D^\beta[ \mathcal{M}[u], \partial_{i}] u\|_{X^0_{T}(\mathcal{O}_{i})}
   \leq \mu \| u \|_{X^{n+1}_{T}(\mathcal{O}_{i})} + C_{\mu} C[ \|u\|_{X^n_{T}(\mathcal{O}_{i})}, \|\mathcal{M}\|_{C^n}]  \| u \|_{X^n_{T}(\mathcal{O}_{i})}, \quad i=0, \, 1, \, 2, \, 3.
 \end{equation}
 \end{itemize}

\end{prop}

\begin{proof}

For \eqref{calculus1}, we use the usual extension operator to extend $\partial_{t}^\alpha u$ as a
 function on $H^{|\beta| + | \gamma|} (\mathbb{R}^3)$.
we then use Fourier analysis and  divide in the frequency space  into $|(\xi_{1}, \xi_{2})| \leq \mu | \xi_{3}|$ and the complementary.

For  \eqref{calculus2},  if $n \geq 3$, we can use \eqref{comm}, we thus only have to deal with $|\beta|= n=2$.
We write
$$  \|D^\beta[ \mathcal{M}[u], \partial_{i}] u\|_{X^0_{T}(\mathcal{O}_{i})}
 \leq \|D^\beta ( \mathcal{M}[u])\|_{X^0_{T}} \| \partial_{i} u\|_{L^\infty_{t,x}} + \| \mathcal{M}\|_{C^1}
  \| D u \|_{L^\infty_{t,x}} \| D^2 u \|_{L^\infty_{T} L^2(\mathcal{O}_{i})}.$$
  We can thus use \eqref{smoothmatrix} and the Sobolev-Gagliardo-Nirenberg  inequality, 
  $$  \| D u\|_{L^\infty_{t,x}} \lesssim  \| Du \|_{L^\infty_{T}H^2(\mathcal{O}_{i})}^{ 1 \over 2}
    \|Du \|_{L^\infty_{T}H^1(\mathcal{O}_{i})}^{ 1 \over 2},$$
    together with the Young inequality.

\end{proof}

\begin{rmk}
From the same type of  arguments as above, we also have that when $|\beta|= k \leq 2$, 
\begin{equation}
\label{comm=2}
  \|D^\beta[ \mathcal{M}[u], \partial_{i}] \omega \|_{X^0_{T}}
   \leq  C[ \| \mathcal{M}\|_{C^2},   \|u \|_{X^3_{T}}] \|\omega\|_{X^2_{T}}.
   \end{equation}
Indeed, 
we just write 
 \begin{align*}
  \|D^\beta[ \mathcal{M}[u], \partial_{i}] \omega \|_{X^0_{T}(\mathcal{O}_{i})}
& \leq  \| \mathcal{M}\|_{C^1}\|D^\beta u \|_{L^\infty_{T}L^4} \| \partial_{i} \omega \|_{L^\infty_{T}L^4}
 +   \| \mathcal{M}\|_{C^2} \|Du \|_{L^\infty_{t,x}}^2 \| \partial_{i} \omega \|_{L^\infty_{T}L^2}
   \\ 
   &\quad +\| \mathcal{M}\|_{C^1}
  \| D u \|_{L^\infty_{T} L^\infty}  \|  D\partial_{i} \omega  \|_{L^\infty_{T} L^2} .
  \end{align*}
  and the result follows by using the Sobolev embeddings $L^4( \mathcal{O}) \subset H^1$, 
   $L^\infty(\mathcal{O}) \subset H^2$.
\end{rmk}


Next we provide here more details about the approximation of the initial data $U_0\in H^m(\Fscr(0))$ at the beginning of Section \ref{approx}. 

\begin{prop}\label{condapproxim}
Assume that the initial data $(U_0, l_0,\omega_0)\in H^m(\Fscr(0))\times (\rline^3)^2$ satisfy the compatibility condition \rfb{compat} of order $m-1$, then there exists $U_0^n\in H^{m+2}(\Fscr(0))$  such that $U_0^n\to U_0$ in $H^m(\Fscr(0))$ and in particular $(U_0^n, l_0,\omega_0)$ satisfy \rfb{compat} of order $m$. 
\end{prop}
\begin{proof}
For  linear hyperbolic systems with non-characteristic boundary, 
Rauch and Massey constructed in \cite[Lemma 3.3]{rauch1974differentiability} an approximate sequence $U_0^n\in H^{m+2}(\Fscr(0))$ such that  the convergence $U_0^n\to U_0$ happens in $H^m(\Fscr(0))$ and  $U_0^n$ verify  the compatibility conditions up to order $m$. Then in \cite{schochet1986compressible} Schochet extended this result to the compressible Euler equations explaining that in this case the singularity of the boundary matrix is actually not an issue. Here we aim to explain how to apply this to the fluid-solid coupled system \rfb{first}--\rfb{solidbar}. 

We focus on the issue of compatibility conditions close to the solid boundary.
We follow the idea in \cite{rauch1974differentiability}. Using \eqref{compatibegin}, 
the compatibility conditions \eqref{compat}
 at the boundary $\partial\Fscr(0)$ can be rewritten as 
\begin{equation*}
 G A_n^k  \gamma \partial_n^k U_0=   F_{k}((C_{i, k}\gamma \partial_n^i U_0)_{0 \leq i \leq k-1}, l_{0}, \omega_{0}) \qquad \forall \m 0\leq k\leq m,
\end{equation*}
where $F_{k}$ is a scalar nonlinear function of its arguments, $\gamma$ stands for the trace on the boundary, 
$C_{i, k}$ are operator of order $k-i$ only involving the tangential derivatives.
 The boundary matrix  $A_n$ was also  introduced above \rfb{compatibegin}.
The main contribution of the dynamics of the solid compared to the situation in \cite{schochet1986compressible} is that the operators $C_{i,k}$ involve both  a local part acting pointwise
 on the trace of  $ \partial_n^i U_0$ and a nonlocal part involving integrals
 on the boundary  that show up in \eqref{compatibegin}. Nevertheless, since they are still bounded
 operators on $H^{k-i}(\partial \mathcal{S}(0)$, we can still use the same arguments.
  As in the proof of \cite[Lemme 3.3]{rauch1974differentiability}, we first find $g^n\in H^{2m+2}(\Fscr(0))$ such that $g^n\to  U_0$ in $H^m(\Fscr(0))$. Then we get  the desired sequence $ U^n_0$ under the form 
$U_0^n=g^n-h^n$ with $h^n\in H^{m+2}(\Fscr(0))$, $h^n\to 0$ in $H^m(\Fscr(0))$ and 
such that
\begin{equation}\label{recom}
G  A_n^k\partial_n^k \gamma h_{n}= G  A_n^k\partial_n^k \gamma g_{n} - F_{k}((C_{i, k}\gamma \partial_n^i (g_{n}- h_{n})_{0 \leq i \leq k-1}, l_{0}, \omega_{0}) \qquad  0\leq k\leq m.
\end{equation}
As observed in  
 \cite{schochet1986compressible}, $\operatorname{Rang}(G  A_n^k)=\operatorname{Rang}(G)=1$. Therefore, $\partial_n^k h^n$ can be solved from \rfb{recom}. The remaining part of the proof 
  of the construction of $h_{n}$ then also follows  from \cite[Lemma 3.3]{rauch1974differentiability}
 since $F_{k}$ involves only terms with lower order normal derivatives.
\end{proof}

\section*{Acknowledgements}
F. Rousset is supported by the BOURGEONS project, grant ANR-23-CE40-0014-01 of the French National Research Agency (ANR). P. Su is partially supported by PEPS ``Jeunes Chercheuses et Jeunes Chercheurs
de l'Insmi" from CNRS and the Sophie Germain program of the Fondation Math\'ematique Jacques Hadamard (FMJH).


%
%


\begin{thebibliography}{10}
	
	\bibitem{benzoni2006multi}
	{\sc S.~Benzoni-Gavage and D.~Serre}, {\em Multi-dimensional Hyperbolic Partial
		Differential Equations: First-order Systems and Applications}, OUP Oxford,
	2006.
	
	\bibitem{desjardins1999existence}
	{\sc B.~Desjardins and M.~J. Esteban}, {\em Existence of weak solutions for the
		motion of rigid bodies in a viscous fluid}, Arch. Ration. Mech. Anal., 146 (1999), pp.~59--71.
	
	\bibitem{ervedoza2014long}
	{\sc S.~Ervedoza, M.~Hillairet, and C.~Lacave}, {\em Long-time behavior for the
		two-dimensional motion of a disk in a viscous fluid}, Comm. Math. Phys., 329 (2014), pp.~325--382.
	
	\bibitem{ervedoza2023large}
	{\sc S.~Ervedoza, D.~Maity, and M.~Tucsnak}, {\em Large time behaviour for the
		motion of a solid in a viscous incompressible fluid}, Math. Ann.,
	385 (2023), pp.~631--691.
	
	\bibitem{feireisl2003motion}
	{\sc E.~Feireisl}, {\em On the motion of rigid bodies in a viscous compressible
		fluid}, Arch. Ration. Mech. Anal., 167 (2003),
	pp.~281--308.
	
	\bibitem{feireisl2021motion}
	{\sc E.~Feireisl and V.~M{\'a}cha}, {\em On the motion of rigid bodies in a
		perfect fluid}, Nonlinear Differ. Equ. Appl., 28
	(2021), p.~35.
	
	\bibitem{glass2012smoothness}
	{\sc O.~Glass, F.~Sueur, and T.~Takahashi}, {\em Smoothness of the motion of a
		rigid body immersed in an incompressible perfect fluid}, Ann. Sci. École Norm. Sup., 45 (2012), pp.~1--51.
	
	\bibitem{gues1990probleme}
	{\sc O.~Gu\`es}, {\em Probleme mixte hyperbolique quasi-lin{\'e}aire
		caract{\'e}ristique}, Comm. Partial Differential Equations, 15
	(1990), pp.~595--654.
	
	\bibitem{gunzburger2000global}
	{\sc M.~D. Gunzburger, H.~C. Lee, and G.~A. Seregin}, {\em Global existence of
		weak solutions for viscous incompressible flows around a moving rigid body in
		three dimensions}, J. Math. Fluid Mech., 2 (2000),
	pp.~219--266.
	
	\bibitem{he2024vanishing}
	{\sc J.~He and P.~Su}, {\em The vanishing limit of a rigid body in
		three-dimensional viscous incompressible fluid}, Math. Ann., 391
	(2025), pp.~5341--5373.
	
	\bibitem{hormander1963linear}
	{\sc L.~H{\"o}rmander}, {\em Linear Partial Differential Operators}, vol.~116,
	New York: Springer-Verlag, 1964.
	
	\bibitem{houot2010existence}
	{\sc J.~G. Houot, J.~San~Martin, and M.~Tucsnak}, {\em Existence of solutions
		for the equations modeling the motion of rigid bodies in an ideal fluid},
	J. Funct. Anal., 259 (2010), pp.~2856--2885.
	
	\bibitem{lannes20232}
	{\sc T.~Iguchi and D.~Lannes}, {\em The {$2 D$} nonlinear shallow water
		equations with a partially immersed obstacle}, to appear J. Eur. Math. Soc.,  (2023).
	
	\bibitem{lacave2017small}
	{\sc C.~Lacave and T.~Takahashi}, {\em Small moving rigid body into a viscous
		incompressible fluid}, Arch. Ration. Mech. Anal., 223
	(2017), pp.~1307--1335.
	
	\bibitem{maity2023motion}
	{\sc D.~Maity and M.~Tucsnak}, {\em Motion of rigid bodies of arbitrary shape
		in a viscous incompressible fluid: Well-posedness and large time behaviour},
J. Math. Fluid Mech., 25 (2023), p.~74.
	
	\bibitem{majda1983stability}
	{\sc A.~Majda}, {\em The Stability of Multi-dimensional Shock Fronts},
	vol.~275, American Mathematical Soc., 1983.
	
	\bibitem{metivier2001stability}
	{\sc G.~M{\'e}tivier}, {\em Stability of Multi-dimensional Shocks}, Advances in
	the Theory of Shock Waves, 47 (2001), pp.~25--103.
	
	\bibitem{nevcasova2022motion}
	{\sc {\v{S}}.~Ne{\v{c}}asov{\'a}, M.~Ramaswamy, A.~Roy, and
		A.~Schl{\"o}merkemper}, {\em Motion of a rigid body in a compressible fluid
		with {N}avier-slip boundary condition}, J. Differential Equations,
	338 (2022), pp.~256--320.
	
	\bibitem{ohkubo1989well}
	{\sc T.~Ohkubo}, {\em Well posedness for quasi-linear hyperbolic mixed problems
		with characteristic boundary}, Hokkaido Math. J., 18 (1989),
	pp.~79--123.
	
	\bibitem{ortega2007motion}
	{\sc J.~Ortega, L.~Rosier, and T.~Takahashi}, {\em On the motion of a rigid
		body immersed in a bidimensional incompressible perfect fluid}, Ann. Inst. H. Poincaré C Anal. Non Linéaire, 24 (2007),
	pp.~139--165.
	
	\bibitem{rauch1974differentiability}
	{\sc J.~B. Rauch and F.~J. Massey}, {\em Differentiability of solutions to
		hyperbolic initial-boundary value problems}, Trans. Amer. Math. Soc., 189 (1974), pp.~303--318.
	
	\bibitem{rosier2009smooth}
	{\sc C.~Rosier and L.~Rosier}, {\em Smooth solutions for the motion of a ball
		in an incompressible perfect fluid}, J. Funct. Anal., 256
	(2009), pp.~1618--1641.
	
	\bibitem{san2002global}
	{\sc J.~A. San~Mart{\'\i}n, V.~Starovoitov, and M.~Tucsnak}, {\em Global weak
		solutions for the two-dimensional motion of several rigid bodies in an
		incompressible viscous fluid}, Arch. Ration. Mech. Anal.,
	161 (2002), pp.~113--147.
	
	\bibitem{schochet1986compressible}
	{\sc S.~Schochet}, {\em The compressible {E}uler equations in a bounded domain:
		Existence of solutions and the incompressible limit}, Commun. Math. Phys., 104 (1986), pp.~49--75.
	
	\bibitem{secchi1996well}
	{\sc P.~Secchi}, {\em Well-posedness of characteristic symmetric hyperbolic
		systems}, Arch. Ration. Mech. Anal., 134 (1996),
	pp.~155--197.
	
	\bibitem{takahashi2003analysis}
	{\sc T.~Takahashi}, {\em Analysis of strong solutions for the equations
		modeling the motion of a rigid-fluid system in a bounded domain}, Adv. Differ. Equ., 8 (2003), pp.~1499--1532.
	
	\bibitem{wang2011analyticity}
	{\sc Y.~Wang and Z.~Xin}, {\em Analyticity of the semigroup associated with the
		fluid--rigid body problem and local existence of strong solutions}, J. Funct. Anal., 261 (2011), pp.~2587--2616.
	
\end{thebibliography}

\end{document}